\documentclass[12pt,reqno]{amsart}

\usepackage[utf8]{inputenc}
\usepackage[T1]{fontenc}

\usepackage[usenames,dvipsnames]{color}

\usepackage{amsfonts, amsmath, amssymb, amscd, latexsym, amsthm}

\usepackage[english]{babel}

\usepackage{enumerate}
\usepackage{longtable}
\usepackage{multirow,longtable,array}
\usepackage{geometry}
\usepackage{float}
\usepackage{adjustbox}
\usepackage{rotating}
\usepackage{tabularx}
\usepackage{booktabs}
\usepackage{array}

\usepackage{tikz}
\usetikzlibrary{shapes,arrows}
\usepackage{xypic}
\usepackage{pifont}

\usepackage[pagebackref]{hyperref}
\renewcommand*{\backref}[1]{}
\newcolumntype{Y}{>{\centering\arraybackslash}X}
\renewcommand*{\backrefalt}[4]{%
	\ifcase #1
	(\tt not cited)%
	\or
	(\tt cited on page~#2)%
	\else
	(\tt cited on pages~#2)%
	\fi}
\usepackage{hyperref}
\hypersetup{
	colorlinks = true,
	urlcolor   = blue,
	citecolor  = blue
}

\newtheorem{theorem}{Theorem}[section]
\newtheorem{lemma}[theorem]{Lemma}

\theoremstyle{remark}
\newtheorem{remark}[theorem]{\bf Remark}

\newtheorem{notation}[theorem]{\bf Notation}

\newtheorem{definition}[theorem]{\bf Definition}

\def\NN{\mathbb{N}}

\def\QQ{\mathbb{Q}}

\def\ZZ{\mathbb{Z}}

\def\sqf{\mathrm{sqf}}
\begin{document}
	
	\def\2r{\mathrm{rank}}
	\def\rg{\mathrm{rank}}
	\def\C{\mathrm{C}}
	\def\Y{\mathbf{Y}}
	\def\V{\mathbf{V}}
	\def\vep{\varepsilon}
	
	\selectlanguage{english}
	
	\title[On the class number of real multiquadratic fields]{On the class number of real multiquadratic fields}

	\author[M. M. Chems-Eddin]{Mohamed Mahmoud Chems-Eddin}
	\address{Mohamed Mahmoud CHEMS-EDDIN: Department of Mathematics, Faculty of Sciences Dhar El Mahraz, Sidi Mohamed Ben Abdellah University, Fez,  Morocco}
	\email{2m.chemseddin@gmail.com}
	
	\author[H. El Mamry]{Hamza El Mamry}
	\address{Hamza EL MAMRY: Departement of Mathematics, Faculty of Sciences Dhar El Mahraz,  Sidi Mohamed Ben Abdellah University, Fez, Morocco}
	\email{Hamza.elmamry@usmba.ac.ma}
	
	\subjclass[2010]{11R27, 11R11, 11R20, 11R04, 11R29.}
	\keywords{ $2$-Class Number,   Multiquadratic Number Fields,  Class Number Parity, Unit Group}
	
	\begin{abstract} 
		In their 1988 monograph, Conner and Hurrelbrink classified all real quadratic and biquadratic fields with odd  class number. They also investigated real quadratic fields whose $2$-class numbers  are equal to  $2$. Their results,  together with some well-known results of Ku{\v{c}}era, provide a complete classification of real quadratic number fields with $2$-class numbers equal to  $2$.
		
		The main aim of this paper is to provide a complete classification of all real biquadratic fields whose $2$-class numbers   equal   $2$; and  moreover, we  draw the complete list all real multiquadratic fields with odd class numbers.  
		Finally, as an appendix we give explicit formulas for $\sqrt{\varepsilon_d}$, where $\varepsilon_d$ denotes the fundamental unit of the real quadratic field $\QQ(\sqrt{d})$.	
	\end{abstract}
	
	\maketitle
	\tableofcontents%
	\section{\bf Introduction}\label{Sec1}

	The investigation of the  $2$-class group and the $2$-class number of a number field $k$ (denoted $A(k)$ and $h(k)$ respectively) was the subject of intensive old and recent interesting studies. This   problem goes back to  Gauss who introduced the first   techniques based on quadratic forms  in Section V of his   book {\it Disquisitiones Arithmeticae} to investigate this problem for real quadratic fields. Then several authors get interested to this problem and focussed on the investigation of the $2$-class group and the $2$-class number field multiquadaratic field. For instance Kaplan \cite{kaplan76} used quadratic forms to study the structure of $2$-class group of some families real quadratic fields. Thereafter,  several   authors used  different techniques and ideas to investigate this problem for the case of real multiquadratic fields  \cite{kuvcera1995parity,frohlich,BLS98,BS95,BS25I,ChemseddinGreenbergConjectureI,paritymohib,MouhibMouvahhedi,DMW,conner88}.   
	In particular, we have the following  interesting classifications which is due to Conner and Hurrelbrink (cf. \cite[Corollary 18.4]{conner88}).
	\begin{lemma}[\cite{{conner88}}, Corollary 18.4]\label{realQuadodd} Let $F$ be a real quadratic field. The class number of $F$ is odd   if and only if it takes one of the following forms:
		\begin{enumerate}[$    1)$]
			\item $F=\QQ(\sqrt{2})$, $\QQ(\sqrt{p})$ where $p$  is a   prime number congruent to $1\pmod 4$. 
			\item $F=\QQ(\sqrt{q})$, $\QQ(\sqrt{2q})$ or $\QQ(\sqrt{q_1q_2})$  where  $q$, $q_1$, $q_2$ are prime  numbers congruent to $3\pmod 4$.
		\end{enumerate}
	\end{lemma}
	
	\medskip 
	
	In the  previous result, Conner and Hurrelbrink classified all real quadratic number fields with odd class number. Moreover, they investigated real  quadratic number fields of $2$-class number equals to $2$. By combining their results (cf. \cite[Corollaries 19.7 and   19.8]{conner88})  with   Ku{\v{c}}era's work (cf. \cite[Theorems 1 and 2]{kuvcera1995parity}), we get the following complete classification of real quadratic fields of $2$-class numbers equal $2$.
	
	\medskip 
	
	\begin{lemma}\label{realQuadeven1New}
		Let $F$ be a real quadratic field $F$. Then $h_2(F)=2$ if and only if $F$ takes one of the following forms:
		\begin{enumerate}[$1)$]
			\item $F=\QQ(\sqrt{pq})$ or $\QQ(\sqrt{2pq})$  with  $\left(\frac{p}{q}\right)=-1$ or $p\equiv5\pmod8$.
			\item $F=\QQ(\sqrt{2q_1q_2})$ with $q_1\equiv3\pmod8$ or $q_2\equiv3\pmod8.$
			\item $F=\QQ(\sqrt{pq_1q_2})$ with $\left(\frac{p}{q_1}\right)=-1$ or $\left(\frac{p}{q_2}\right)=-1.$

			\item $F=\QQ(\sqrt{2p})$ with $\left(\frac{2}{p}\right)=-1$ or $[\left(\frac{2}{p}\right)=1$ and $\left(\frac{2}{p}\right)_4\not=\left(\frac{p}{2}\right)_4]$.

			\item $F=\QQ(\sqrt{p_1p_2})$ with $\left(\frac{p_1}{p_2}\right)=-1$ or $[\left(\frac{p_1}{p_2}\right)=1$ and $\left(\frac{p_1}{p_2}\right)_4\not=\left(\frac{p_2}{p_1}\right)_4]$.

		\end{enumerate}
		Here $p\equiv 1 \pmod4$ and $q\equiv q_1\equiv q_2\equiv3\pmod4$ are distinct prime numbers.
	\end{lemma}


		
		\medskip 
		For real multiquadratic fields of degrees $\geq 8$, a complete classification of those with odd class number is still lacking in the literature, however, we have some very interesting results in this direction on real multiquadratic fields of the form $ P(n)=\QQ(\sqrt{p_1},\sqrt{p_2},\cdots,\sqrt{p_n})$, where all the $p_i$   are prime numbers. The  multiquadratic fields of this form are called Fröhlich multiquadratic fields\label{frohfi} (cf. \cite{MJAGA}). In particular, according to Fröhlich (cf. \cite{frohlich}), when $n\geq 5$ the class number of $P(n)$ is even. For $n=4$, Mouhib (cf. \cite{paritymohib}) classified the Fröhlich multiquadratic fields $P(4)$ whose class number is odd.
		Later, we provide the list of all real multiquadratic fields $D(n):=\QQ(\sqrt{d_1},\sqrt{d_2},\cdots,\sqrt{d_n})$, \label{Dn} where $n\geq 1$ and all the $d_i$ are positive square-free integers,  for  which the class number is odd (cf. Lemma \ref{realBiquad}, Theorem \ref{triquawithodd} and Corollary \ref{duaqdriwithodd} below).

		\medskip


		For real biquadratic fields (i.e., $D(n)$ for $n=2$), Conner and Hurrelbrink (cf. \cite[Corollaries 21.2, 21.4 and Proposition 21.5]{conner88}) provided a complete classification of these fields with odd class number. The following  lemma provides   corrections to a minor inaccuracy   in their well-known classification.
		
		\medskip

		\begin{lemma}[cf. Lemma \ref{realBiquadCorre} below]\label{realBiquad} Let $K$ be a real biquadratic number field. The class number of $K$ is odd if and only if   it takes one of the following forms:
			\begin{enumerate}[$1)$]
				\item $K=\QQ(\sqrt{2},\sqrt{q})$,   $\QQ(\sqrt{q_1},\sqrt{q_2})$, $\QQ(\sqrt{2q_1},\sqrt{2q_2})$ or $\QQ(\sqrt{q_1q_2},\sqrt{q_2q_3}),$
				
				\item $K=\QQ(\sqrt{q},\sqrt{p})$ or $\QQ(\sqrt{2q},\sqrt{p})$ with $\left(\frac{p}{q}\right)=-1$ or $p\equiv5 \pmod 8$,
				\item	$K=\QQ(\sqrt{2},\sqrt{q_1q_2})$ with   $q_1\equiv3 \pmod 8$ or $q_2\equiv3 \pmod 8$, 
				\item $K=\QQ(\sqrt{p},\sqrt{q_1q_2})$  with $\left(\frac{q_1}{p}\right)=-1$ or $\left(\frac{q_2}{p}\right)=-1$,
				
				\item  $K=\QQ(\sqrt{2},\sqrt{p})$ with $\left(\frac{2}{p}\right)=-1$ or $[\left(\frac{2}{p}\right)=1$ and $\left(\frac{2}{p}\right)_4\not=\left(\frac{p}{2}\right)_4]$.
				
				\item $K=\QQ(\sqrt{p_1},\sqrt{p_2})$ with $\left(\frac{p_1}{p_2}\right)=-1$ or $[\left(\frac{p_1}{p_2}\right)=1$ and $\left(\frac{p_1}{p_2}\right)_4\not=\left(\frac{p_2}{p_1}\right)_4]$.
			\end{enumerate}
			Here, $p$, $p_1$, $p_2$  denote  primes numbers congruent to $1\pmod 4$ and $q$, $q_1$, $q_2$ denote  primes numbers congruent to $3\pmod 4$.
		\end{lemma}
		\medskip 
		The classifications and results presented above naturally lead to the following questions:
		\medskip
		\begin{center}
			{\bf What are real multiquadratic fields with odd class numbers (resp. with $2$-class numbers equal  $2$)?}
		\end{center}
		\medskip 
		
		Our main goal in this paper is to classify all real biquadratic number fields whose $2$-class numbers equal  to $2$, moreover, as applications we draw the list of all real multiquadratic number fields with odd class numbers. We believe that this is a   challenging problem, which explains why, to the best of our knowledge, no classification of real biquadratic fields with $2$-class number equal to $2$ has been obtained to date. This natural question has remained open since the publication of the book by Conner and Hurrelbrink in 1988. We believe that the subsequent development of computational techniques for determining the unit groups of multiquadratic fields (cf., e.g.,  \cite{Benchems,BS25I,	MJAGA,firstpaper,wada}), together with the theory of $2$-class field towers of real quadratic fields developed by Benjamin, Lemmermeyer and Snyder (cf.  \cite{BLS98,BS95,BS25I,BS25}),
		were an essential ingredients in obtaining a complete classification. In particular, the  classification of real quadratic fields whose $2$-class field tower is abelian provided a crucial tool that allowed us to substantially simplify our computations.

		\medskip 
		
		Before describing the strategy implemented to obtain our classification, we recall the following definition introduced in \cite{ChemseddinGreenbergConjectureI} and fix the notation that will be used throughout the paper.

		\begin{definition}
			A number field $k$  is said QO-field if it is a quadratic extension of certain number field $k'$ whose class number is odd. We shall call $k'$ a base field of the QO-field $k$ and that $k/k'$ is a QO-extension. An extension of  QO-fields is an extension of number fields $L/M$ such that $L$ and $M$ are QO-fields.	\end{definition}
		
		Let $h (d)$ (resp. $h_2(d)$, $\varepsilon_d$) denote the  class number (resp. $2$-class number, fundamental unit) of 
		the real quadratic field $\QQ(\sqrt{d})$, with $d$ a positive square-free integer. Let $E_k$ denote the unit group of a number field $k$. We also define the following notation.

		\begin{notation}\label{squarenot}We say that $\sqrt{\varepsilon_d }$, the square-root of the fundamental unit of  a real quadratic field $\QQ(\sqrt{d})$ equals $[\sqrt{m_1},\sqrt{m_2}]$, where $m_1$ and $m_2$ are two positive integers (i.e,  $\sqrt{\varepsilon_d }=[\sqrt{m_1},\sqrt{m_2}]$), if $\sqrt{\varepsilon_d }=n_1\sqrt{m_1}+n_2\sqrt{m_2}$ for some     integers or semi-integers  $n_1$ and $n_2$. 
		\end{notation}
		
		Now let $K=\QQ(\sqrt{d_1},\sqrt{d_2})$ be a real biquadratic number field, where $d_1$ and $d_2$ are two positive square-free integers. Let $d_3=\sqf(d_1d_2)$ (i.e. the square-free part of $d_1d_2$).	Our strategy for classifying all real biquadratic fields $K$ with $2$-class number equals to $2$ is as follows. 
		
		By Kuroda's class number formula (cf. \cite[p. 247]{Lemmermeyer94}), we have:
		\begin{equation}\label{equa1}
			h_2(K) = \frac{1}{4}\cdot q(K)\cdot h_2(d_1)\cdot h_2(d_2)\cdot h_2(d_3).
		\end{equation}
		Here   $q(K):= [E_K: \prod_{i=1}^{3}E_{\QQ(\sqrt{d_i})}]$ denotes the unit index of $K$.
		The investigation of unit index  $q(K)$ is difficult and  was the subject of intensive interesting studies (cf. for example \cite{AzZektaous18,Chanetall}).
		In the last section of this paper we give some useful computations that can be used to deduce the value of $q(K)$ whenever
	$d_1$ and $d_2$  have	a small number of prime divisors (i.e., $2$, $3$ or $4$ prime number of divisors).

	It follows form Kuroda's class number formula (i.e., \eqref{equa1}) that we have:
		\begin{equation}\label{equa2}
			h_2(K)=2  \text{ if and only if  }  q(K) h_2(d_1) h_2(d_2)h_2(d_3)=8. 
		\end{equation}
		So assuming that $h_2(K)=2$, it follows (from \eqref{equa2}) that $K$   satisfies one of the following conditions:
	
	\begin{equation}
		K \text{ is a QO-field},
	\end{equation}
	or
	\begin{equation}
		h_2(d_1)= h_2(d_2)=h_2(d_3)=2 \text{ and }  q(K)=1 .
	\end{equation}

	Notice that, by Lemma \ref{realQuadodd},  a real biquadratic field  $K$  that is a QO-field  is in one of the following forms:
	\begin{eqnarray}\label{forms1ofk}
		K=\QQ(\sqrt{q},\sqrt{d}), \; \QQ(\sqrt{2q},\sqrt{d}), \; \QQ(\sqrt{q_1q_2},\sqrt{d}),   \;  \QQ(\sqrt{2},\sqrt{d}) \text{ or }   \QQ(\sqrt{p},\sqrt{d}).
	\end{eqnarray}
	where $q\equiv q_1\equiv q_2\equiv 3\pmod 4$, $p\equiv1 \pmod 4$ are prime numbers and $d$ is a positive square-free integer.
	Moreover,  the forms of real biquadratic fields $K$ such that $h_2(d_1)= h_2(d_2)=h_2(d_3)=2$ are given in Table \ref{biquaNotQO}, by using Lemma  \ref{realQuadeven1New}.

	{
		\setlength{\tabcolsep}{12pt}
		\renewcommand{\arraystretch}{1.6}
		
		\begin{longtable}{|c|c|c| p{6cm}|c|}
			
			\toprule
			$d_1$ & $d_2$ & $d_3$ & \text{Conditions} & $K$ \\
			\midrule
			\endfirsthead
			
			\toprule
			$d_1$ & $d_2$ & $d_3$ & \text{Conditions} & $K$ \\
			\midrule
			\endhead
			
			\hline
			\endfoot
			
			\bottomrule
			\caption{Fields $K$ such that $h_2(d_1)= h_2(d_2)=h_2(d_3)=2$} \label{biquaNotQO} \\
			\endlastfoot
			
			$pq_1$ & $2q_1q_2$  & $2pq_2$ &[$\left(\frac{p}{q_1}\right)=\left(\frac{p}{q_2}\right)=-1$ or $p\equiv5\pmod8$] and [$q_1\equiv 3\pmod 8$ or $q_2\equiv 3\pmod 8$] & 	$ \mathbb{Q}(\sqrt{pq_1},\sqrt{2q_1q_2})$\\
			\hline
			$pq_1q_2$  & $2q_1q_2$  &  $2p$ & $p\equiv5\pmod8$, [$\left(\frac{p}{q_1}\right)=-1$ or $\left(\frac{p}{q_2}\right)=-1$] \text{ and } [$q_1\equiv 3\pmod 8$ \text{ or }$q_2\equiv 3\pmod 8$]&$ \mathbb{Q}(\sqrt{pq_1q_2},\sqrt{2q_1q_2})$ \\
			
			\hline
			$pq_1q_2$  & $2q_1q_2$  &  $2p$ & $p\equiv1\pmod8$, $\left(\frac{2}{p}\right)_4\neq \left(\frac{p}{2}\right)_4$,  [$\left(\frac{p}{q_1}\right)=-1$ or $\left(\frac{p}{q_2}\right)=-1$] \text{ and } [$q_1\equiv 3\pmod 8$ \text{ or }$q_2\equiv 3\pmod 8$]&$ \mathbb{Q}(\sqrt{pq_1q_2},\sqrt{2q_1q_2})$ \\

			\hline
			$p_1q$ &  $p_2q$  & $p_1p_2$ &[$\left(\frac{p_1}{q}\right)=-1$ or $p_1\equiv 5\pmod 8$], [$\left(\frac{p_2}{q}\right)=-1$ or $p_2\equiv 5\pmod 8$] \text{ and } $\left(\frac{p_1}{p_2}\right)=-1$ &$ \mathbb{Q}(\sqrt{pq},\sqrt{p_2q})$ \\

			\hline
			$2p_1q$ &  $p_1p_2$  & $2qp_2$ &[$\left(\frac{p_1}{q}\right)=-1$ or $p_1\equiv 5\pmod 8$], [$\left(\frac{p_2}{q}\right)=-1$ or $p_2\equiv 5\pmod 8$] \text{ and } $\left(\frac{p_1}{p_2}\right)=-1$ &$ \mathbb{Q}(\sqrt{2p_1q},\sqrt{p_1p_2})$ \\
			
			\hline
			$p_1q$ &  $p_2q$  & $p_1p_2$ &[$\left(\frac{p_1}{q}\right)=-1$ or $p_1\equiv 5\pmod 8$], [$\left(\frac{p_2}{q}\right)=-1$ or $p_2\equiv 5\pmod 8$],  $\left(\frac{p_1}{p_2}\right)=1$\text{ and } $\left(\frac{p_1}{p_2}\right)_4\neq \left(\frac{p_2}{p_1}\right)_4$ &$ \mathbb{Q}(\sqrt{pq},\sqrt{p_2q})$ \\
			
			\hline
			$2p_1q$ &  $p_1p_2$  & $2qp_2$ &[$\left(\frac{p_1}{q}\right)=-1$ or $p_1\equiv 5\pmod 8$], [$\left(\frac{p_2}{q}\right)=-1$ or $p_2\equiv 5\pmod 8$],  $\left(\frac{p_1}{p_2}\right)=1$ \text{ and } $\left(\frac{p_1}{p_2}\right)_4\neq \left(\frac{p_2}{p_1}\right)_4$ &$ \mathbb{Q}(\sqrt{2p_1q},\sqrt{p_1p_2})$ \\

			\hline
			$p_1q_1q_2$ &  $p_2q_1q_2$  & $p_1p_2$ &[$\left(\frac{p_1}{q_1}\right)=-1$ or $\left(\frac{p_1}{q_2}\right)=-1$], [$\left(\frac{p_2}{q_1}\right)=-1$ or $\left(\frac{p_2}{q_2}\right)=-1$] \text{and} $\left(\frac{p_1}{p_2}\right)=-1$ &$ \mathbb{Q}(\sqrt{p_1q_1q_2},\sqrt{p_2q_1q_2})$ \\

			\hline
			$p_1q_1q_2$ &  $p_2q_1q_2$  & $p_1p_2$ &[$\left(\frac{p_1}{q_1}\right)=-1$ or $\left(\frac{p_1}{q_2}\right)=-1$], [$\left(\frac{p_2}{q_1}\right)=-1$ or $\left(\frac{p_2}{q_2}\right)=-1$],  $\left(\frac{p_1}{p_2}\right)=1$ \text{and} $\left(\frac{p_1}{p_2}\right)_4\neq \left(\frac{p_2}{p_1}\right)_4$ &$ \mathbb{Q}(\sqrt{p_1q_1q_2},\sqrt{p_2q_1q_2})$ \\

			\hline
			$2p_1$ &  $2p_2$  & $p_1p_2$ &$p_1\equiv p_2\equiv5\pmod8 $,  \text{ and } $\left(\frac{p_1}{p_2}\right)=-1$ &$ \mathbb{Q}(\sqrt{2p_1},\sqrt{2p_2})$ \\
			
			\hline
			$2p_1$ &  $2p_2$  & $p_1p_2$ &$p_1\equiv p_2\equiv5\pmod8 $,  $\left(\frac{p_1}{p_2}\right)=1$ \text{ and } $\left(\frac{p_1}{p_2}\right)_4\neq \left(\frac{p_2}{p_1}\right)_4$ &$ \mathbb{Q}(\sqrt{2p_1},\sqrt{2p_2})$ \\

			\hline
			$2p_1$ &  $2p_2$  & $p_1p_2$ &$p_1\equiv5\pmod8$, $p_2\equiv1\pmod8 $,   $\left(\frac{2}{p_2}\right)_4\neq \left(\frac{p_2}{2}\right)_4$\text{ and } $\left(\frac{p_1}{p_2}\right)=-1$ &$ \mathbb{Q}(\sqrt{2p_1},\sqrt{2p_2})$ \\

			\hline
			$2p_1$ &  $2p_2$  & $p_1p_2$ &$p_1\equiv5\pmod8$, $p_2\equiv1\pmod8 $, $\left(\frac{2}{p_2}\right)_4\neq \left(\frac{p_2}{2}\right)_4$, $\left(\frac{p_1}{p_2}\right)=1$  \text{ and }  $\left(\frac{p_1}{p_2}\right)_4\neq \left(\frac{p_2}{p_1}\right)_4$  &$ \mathbb{Q}(\sqrt{2p_1},\sqrt{2p_2})$ \\

			\hline
			$2p_1$ &  $2p_2$  & $p_1p_2$ &$p_1\equiv p_2\equiv1\pmod8 $, $\left(\frac{2}{p_1}\right)_4\neq \left(\frac{p_1}{2}\right)_4$, $\left(\frac{2}{p_2}\right)_4\neq \left(\frac{p_2}{2}\right)_4$  \text{ and } $\left(\frac{p_1}{p_2}\right)=-1$     &$ \mathbb{Q}(\sqrt{2p_1},\sqrt{2p_2})$ \\

			\hline
			$2p_1$ &  $2p_2$  & $p_1p_2$ &$p_1\equiv p_2\equiv1\pmod8 $, $\left(\frac{2}{p_1}\right)_4\neq \left(\frac{p_1}{2}\right)_4$, $\left(\frac{2}{p_2}\right)_4\neq \left(\frac{p_2}{2}\right)_4$,   $\left(\frac{p_1}{p_2}\right)=1$ \text{ and }    $\left(\frac{p_1}{p_2}\right)_4\neq \left(\frac{p_2}{p_1}\right)_4$     &$ \mathbb{Q}(\sqrt{2p_1},\sqrt{2p_2})$ \\

			\hline
			$p_1p_2$ &  $p_2p_3$  & $p_1p_3$ & $\left(\frac{p_1}{p_2}\right)=\left(\frac{p_2}{p_3}\right)=\left(\frac{p_1}{p_3}\right)=-1$ &$ \mathbb{Q}(\sqrt{p_1p_2},\sqrt{p_2p_3})$ \\

			\hline
			$p_1p_2$ &  $p_2p_3$  & $p_1p_3$ & $\left(\frac{p_1}{p_2}\right)=\left(\frac{p_2}{p_3}\right)=-1$, $\left(\frac{p_1}{p_3}\right)=1$\text{ and } $\left(\frac{p_1}{p_3}\right)_4\neq \left(\frac{p_3}{p_1}\right)_4$  &$ \mathbb{Q}(\sqrt{p_1p_2},\sqrt{p_2p_3})$ \\

			\hline
			$p_1p_2$ &  $p_2p_3$  & $p_1p_3$ & $\left(\frac{p_1}{p_2}\right)=-1$, $\left(\frac{p_2}{p_3}\right)=\left(\frac{p_1}{p_3}\right)=1$,  $\left(\frac{p_1}{p_3}\right)_4\neq \left(\frac{p_3}{p_1}\right)_4$ \text{ and } $\left(\frac{p_2}{p_3}\right)_4\neq \left(\frac{p_3}{p_2}\right)_4$ &$ \mathbb{Q}(\sqrt{p_1p_2},\sqrt{p_2p_3})$ \\
			
			\hline
			$p_1p_2$ &  $p_2p_3$  & $p_1p_3$ &  $\left(\frac{p_1}{p_2}\right)=\left(\frac{p_2}{p_3}\right)=\left(\frac{p_1}{p_3}\right)=1$,  $\left(\frac{p_1}{p_3}\right)_4\neq \left(\frac{p_3}{p_1}\right)_4$, $\left(\frac{p_2}{p_3}\right)_4\neq \left(\frac{p_3}{p_2}\right)_4$  \text{ and } $\left(\frac{p_1}{p_2}\right)_4\neq \left(\frac{p_2}{p_1}\right)_4$  &$ \mathbb{Q}(\sqrt{p_1p_2},\sqrt{p_2p_3})$ \\

		\end{longtable}
	}

	Now we state the following remark  based on  class field theory that will play an important role later in our proofs.
	\begin{remark}\label{abelianquaexten}
		Let $M/F$ be a quadratic unramified extension of  number fields.  Then $h_2(M)=2$ if and only if $h_2(F)=4$ and $F^{(1)}=F^{(2)}$, where 
		$F^{(1)}$ (resp. $ F^{(2)}$) the Hilbert $2$-class field  of $F$ (resp. $F^{(1)}$). In fact, by class field theory we have the following diagram: 
		\[
		\xymatrix@C=1.8cm{
			F
			\ar@{-}[r]
			\ar@/^1.5pc/[rr]^{A(F)}
			&
			M
			\ar@{-}[r]
			\ar@/_1.5pc/[rr]_{A(M)}
			&
			F^{(1)}
			\ar@{-}[r]
			&
			M^{(1)}
			\ar@{-}[r]
			&
			\mathcal{L}(M)=\mathcal{L}(F)
		}
		\]

		Here $\mathcal{L}(k)$ denotes the maximal unramified $2$-extension of a number field $k$.	So one deduces directly the result.	
\end{remark}

This remark will be used later for the case when $F$ is a real quadratic field; and thanks to the works  of 
Benjamin, Lemmermeyer  and Snyder (cf. \cite{BLS98,BS25I,BS25}) when know when we have $F^{(1)}=F^{(2)}$.

\section{\bf Real biquadratic fields that are QO-field with $2$-class number equals $2$}
Let $K$ be a real biquadratic number field that is a QO-field. By 	Lemma \ref{realQuadodd}, $K$ takes one of the following forms: 
$$K=\QQ(\sqrt{q},\sqrt{d}), \; \QQ(\sqrt{2q},\sqrt{d}), \; \QQ(\sqrt{q_1q_2},\sqrt{d}),   \;  \QQ(\sqrt{2},\sqrt{d}) \text{ or }   \QQ(\sqrt{p},\sqrt{d}),$$
where $q\equiv q_1\equiv q_2\equiv 3\pmod 4$, $p\equiv1 \pmod 4$ are prime numbers and $d$ is a positive square-free integer. The what follows, we  investigate each case in a  subsection. 

\medskip
Before starting our investigations, let us recall the following algorithm due to Wada \cite{wada}, that we will use to compute the   unit index $  q(K)$ appearing in Kuroda's formula (i.e., \eqref{equa1}). In fact, it describes a fundamental system  of units of a    multiquadratic field $k$. Let  $\tau_1$ and 
$\tau_2$ be two distinct elements of order $2$ in the Galois group of $k/\mathbb{Q}$. Let $k_1$, $k_2$ and $k_3$ be the three subextensions of $k$ invariant by  $\tau_1$,\label{wadaalgo}
$\tau_2$ and $\tau_3= \tau_1\tau_2$, respectively. Let $\varepsilon$ denote a unit of $k$. Then \label{algo wada}
$$\varepsilon^2=\varepsilon\varepsilon^{\tau_1}  \varepsilon\varepsilon^{\tau_2}(\varepsilon^{\tau_1}\varepsilon^{\tau_2})^{-1},$$
and we have, $\varepsilon\varepsilon^{\tau_1}\in E_{k_1}$, $\varepsilon\varepsilon^{\tau_2}\in E_{k_2}$  and $\varepsilon^{\tau_1}\varepsilon^{\tau_2}\in E_{k_3}$.
It follows that the unit group of $k$  
is generated by the elements of  $E_{k_1}$, $E_{k_2}$ and $E_{k_3}$, and the square roots of elements of   $E_{k_1}E_{k_2}E_{k_3}$ which are perfect squares in $k$.

\subsection{Real biquadratic fields  $K=\QQ(\sqrt{q},\sqrt{d})$ with $2$-class numbers equal  $2$}

The following theorem provides a complete classification of the real biquadratic fields
$K=\QQ(\sqrt{q},\sqrt{d})$ whose $2$-class number is~$2$.

\begin{theorem}\label{classqrs}
	Let $q\equiv3\pmod{4}$   be a prime number    and $d$ a positive square-free integer co-prime with $q$. Put 
	$K=\QQ(\sqrt{q},\sqrt{d})$.
	Then $h_2(K)=2$ if and only if $d$ is in one of the following forms:
	\begin{enumerate}[$A)$]
		\item $d=r$, where  $r\equiv1\pmod 8$ is a prime number such that     $\left(\frac{r}{q}\right)=1$ and  
		$h_2(qr)=4$.  
		\item  $d=rs$, where $r$ and $s$ are prime numbers such that,
		after a suitable permutation of $r$  and $s$, we have one of the following conditions:
		\begin{enumerate}[$1)$]
			\item $r\equiv5\pmod{8}$, $ s \equiv3\pmod{4}$, and
			\[
			\left(\left(\frac{q}{r}\right),\left(\frac{q}{s}\right),
			\left(\frac{r}{s}\right)\right)
			\in\{(i,j,k):i,j,k=\pm1\}
			\setminus\{(1,1,1),(1,-1,1)\},
			\]
			
			\item    $r\equiv1\pmod{8}$, $ s \equiv3\pmod{4}$ and  $\left(\dfrac{q}{r}\right)=\left(\dfrac{r}{s}\right)=-1$,
			\item $r\equiv q \equiv 7\pmod{8}$, $s\equiv3\pmod4$ and  $\left(\dfrac{q}{r}\right)=-\left(\dfrac{q}{s}\right)=\left(\dfrac{r}{s}\right)=1$,
			\item  $r\equiv s \equiv 3\pmod{8}$ and  $\left(\dfrac{q}{r}\right)=-\left(\dfrac{q}{s}\right)=\left(\dfrac{r}{s}\right)=1$,
			
			
			\item  $s\equiv7\pmod8, q \equiv 3\pmod{8}$, $r\equiv3\pmod4$  and  $\left(\dfrac{q}{r}\right)=\left(\dfrac{q}{s}\right)=-\left(\dfrac{r}{s}\right)=-1$,
			
			
			\item $s\equiv3\pmod8, q \equiv 7\pmod{8}$, $r\equiv3\pmod4$  and  $\left(\dfrac{q}{r}\right)=\left(\dfrac{q}{s}\right)=-\left(\dfrac{r}{s}\right)=-1$,
			

			\item $r\equiv1\pmod8, s \equiv 5\pmod{8}$ and  $\left(\dfrac{q}{r}\right)=-\left(\dfrac{q}{s}\right)=\left(\dfrac{r}{s}\right)=-1$,

			\item $r\equiv s\equiv5\pmod8$ and  $\left(\dfrac{q}{r}\right)=-\left(\dfrac{q}{s}\right)=-\left(\dfrac{r}{s}\right)=1$,
			\item $r\equiv1\pmod8, s\equiv5\pmod8,q\equiv3\pmod8$ and  $\left(\dfrac{q}{r}\right)=\left(\dfrac{q}{s}\right)=\left(\dfrac{r}{s}\right)=-1$.

		\end{enumerate}	
		\item $d=2p$, where $p\equiv 5\pmod8$  is a prime number.
		\item $d=2q'$, where $q'\equiv 3\pmod4$ is a prime number   such that $q$ or $q'$ is congruent to $3\pmod8.$
		
	\end{enumerate}
\end{theorem}

\begin{proof}$\bullet$ Let us firstly note that if $h_2(K)=2$ and $d$ is odd, then the rank of the $2$-class group of $K$ equals  $1$; and so,
	according to \cite[Lemma 3.1]{ChemseddinGreenbergConjectureI},   $d$ is in one of the following forms:
	\begin{enumerate}[$a)$]
		\item $d=r$ is a prime number such that $r\equiv1\pmod 8$ and  $\left(\frac{r}{q}\right)=1$,
		\item  $d=rs$ where $r$ and $s$ are two distinct prime numbers.
	\end{enumerate}
	Assume that $d$ takes the first form (i.e. as in $a)$). So the extension $K/\QQ(\sqrt{qr})$ is an unramified quadratic extension. Since the $2$-class group of 
	$	\QQ(\sqrt{qr})$ is cyclic (cf. \cite[Corollary 18.6]{conner88}), it follows that
	$h_2(K)=2$ if and only if 
	$h_2(qr)=4$ (this justifies the form $A)$).

	Now assume that $d$ takes the second form (i.e. as in $b)$). Let us characterize  the integers $d$ such that $h(K)=2$. Put $k=\QQ(\sqrt{qrs})$. We distinguish the following two cases: 
	
	\noindent\ding{229}	 	Assume that $K/k$ is  {ramified}, i.e., $rs\equiv 3\pmod4$. So $h_2(k)$ divides $h_2(K)$, and thus   $h_2(k)=1$ or $2$.  The case $h_2(k)=1$ is rejected as  $h_2(k)$ is divisible by $2$   (cf. \cite[Corollary 18.6]{conner88}). Therefore, $h_2(k)=2$;
	and so, by Lemma \ref{realQuadeven1New}, we have	$q\equiv s \equiv3 \pmod4$ and $r\equiv1\pmod4$ with $\left(\dfrac{q}{r}\right)=-1$ or $\left(\dfrac{s}{r}\right)=-1$.  We distinguish two sub-cases:
	\begin{enumerate}
		\item Let $r\equiv5\pmod 8$. 
		Using Tables \ref{tab7} and \ref{tab8}, and Wada's algorithm (cf. Page \pageref{wadaalgo}),   we check that $E_K= \langle{-1,\vep_q,\vep_{rs},\eta(r,s,q)\rangle}$, where $\eta(r,s,q)$ is as in Table \ref{formofeta}:
		\begin{table}[H]  
			\centering
			\scriptsize 
			\setlength{\tabcolsep}{4pt} 
			\renewcommand{\arraystretch}{1.3} 
			\begin{tabular}{c|c|c|c|c}
				\toprule
				$\left(\left(\dfrac{q}{r}\right),\left(\dfrac{q}{s}\right),\left(\dfrac{r}{s}\right)\right)$ & $\eta(5,7,3)$ & $\eta(5,3,3)$ & $\eta(5,3,7)$& $\eta(5,7,7)$\\
				\midrule
				$(1,1,-1)$	&  $\sqrt{\vep_{qrs}}$ or $\sqrt{\vep_{rs}\vep_{q}\vep_{qrs}}$ &  $\sqrt{\vep_{qrs}}$ & $\sqrt{\vep_{qrs}}$ or $\sqrt{\vep_{rs}\vep_{qrs}}$&$\sqrt{\vep_{qrs}}$\\
				\hline
				$(-1,-1,1)$  &  $\sqrt{\vep_{q}\vep_{qrs}}$ or $\sqrt{\vep_{rs}\vep_{q}\vep_{qrs}}$ & $\sqrt{\vep_{rs}\vep_{qrs}}$& $\sqrt{\vep_{q}\vep_{qrs}}$, $\sqrt{\vep_{rs}\vep_{q}\vep_{qrs}}$ & $\sqrt{\vep_{rs}\vep_{qrs}}$\\
				\hline
				$(-1,-1,-1)$  & $\sqrt{\vep_{q}\vep_{qrs}},\sqrt{\vep_{rs}\vep_{qrs}}$ or $\sqrt{\vep_{rs}\vep_{q}\vep_{qrs}}$ & $\sqrt{\vep_{rs}\vep_{q}\vep_{qrs}}$& $\sqrt{\vep_{q}\vep_{qrs}}$ or $\sqrt{\vep_{rs}\vep_{q}\vep_{qrs}}$ &$\sqrt{\vep_{q}\vep_{rs}\vep_{qrs}}$ \\
				\hline
				$(-1,1,1)$ &  $\sqrt{\vep_{q}\vep_{qrs}},\sqrt{\vep_{rs}\vep_{qrs}}$ or $\sqrt{\vep_{rs}\vep_{q}\vep_{qrs}}$ &  $\sqrt{\vep_{rs}\vep_{qrs}}$&$\sqrt{\vep_{q}\vep_{qrs}}$, $\sqrt{\vep_{rs}\vep_{qrs}}$ or $\sqrt{\vep_{q}\vep_{rs}\vep_{qrs}}$ & $\sqrt{\vep_{rs}\vep_{qrs}}$\\
				\hline
				$(-1,1,-1)$ & $\sqrt{\vep_{q}\vep_{qrs}}$ or $\sqrt{\vep_{rs} \vep_{q}\vep_{qrs}}$ & $\sqrt{\vep_{rs}\vep_{q}\vep_{qrs}}$&$\sqrt{\vep_{q}\vep_{qrs}}$, $\sqrt{\vep_{rs}\vep_{qrs}}$ or $\sqrt{\vep_{q}\vep_{rs}\vep_{qrs}}$ & $\sqrt{\vep_{q}\vep_{rs}\vep_{qrs}}$\\
				\hline
				$(1,-1,-1)$   & $\sqrt{\vep_{qrs}}$ or $\sqrt{\vep_{rs}\vep_{qrs}}$ &  $\sqrt{\vep_{qrs}}$& $\sqrt{\vep_{qrs}}$ or $\sqrt{\vep_{rs}\vep_{qrs}}$ &$\sqrt{\vep_{qrs}}$ \\
				\bottomrule
			\end{tabular}
			\caption{The form of $\eta(r,s,q)$}
			\label{formofeta}
		\end{table}

		Thus,	 $q(K)=2$.
		Since $h_2(rs)=h_2(qrs)=2$ (cf. Lemma \ref{realQuadeven1New}), it follows  by Kuroda's class number formula that we have $h_2(K)=2$ (this justifies the form $B)$-$1)$).	
		\item Now let $r\equiv1\pmod 8$. According to Lemma \ref{realQuadeven1New}, $[h_2(qrs)=2$ and $h_2(rs)=2]$  if and only if  $\left(\dfrac{s}{r}\right)=-1$. Furthermore, from
		Tables \ref{tab7} and \ref{tab8},  we deduce that $E_K= \langle{-1,\vep_q,\sqrt{\vep_{q}\vep_{rs}},\eta\rangle}$ where $\eta=\vep_{qrs}$ or $\sqrt{\vep_{qrs}}$ according to whether $\left(\dfrac{s}{r}\right)=\left(\dfrac{s}{r}\right)=-1$ or not.  Therefore, 
		$q(K)=2$ exactly when $\left(\dfrac{s}{r}\right)=\left(\dfrac{q}{r}\right)=-1$.
		It follows, by  Kuroda's class number formula  (cf. \eqref{equa1}), that    $h_2(K)=2$ if and only if $\left(\dfrac{q}{r}\right)= \left(\dfrac{s}{r}\right)=-1$ (this justifies the form $B)$-$2)$).
	\end{enumerate}
	
	\noindent\ding{229} 	Now assume that  $K/k$ is unramified, i.e., $rs\equiv1\pmod 4$. We distinguish two sub-cases: 
	\begin{enumerate}
		\item Let $r\equiv s\equiv 3\pmod4$, then   \cite[Theorem 1]{BS95} implies that $k^{(1)}=k^{(2)}$. Therefore, 
		$	h_2(K)=\frac{h_2(k)}{2}; $
	and thus $h_2(K)=2$ if and only if $h_2(k)=4$. Under the  conditions of this sub-case, the   quadratic fields $k=\QQ(\sqrt{qrs})$ such that $h_2(k)=4$ can be deduced from \cite[Form $d_i$, for $i\in \{1,...,8\}$, page 11]{BS25I}. In fact, we have $h_2(k)=4$ if and only if,    after a suitable permutation of the prime numbers  $q,r$ and $s$, we have $\left(\dfrac{r}{s}\right)=\left(\dfrac{s}{q}\right)=1$ and one of the    conditions in Table \ref{refddd} :

	\begin{center}
		\renewcommand{\arraystretch}{1.6}
		\begin{table}[H]
			\begin{tabular}{c|ccc|c}
				
				\toprule
				Cas
				& $r \pmod{8}$
				& $s \pmod{8}$
				& $q \pmod{8}$  &  $\left(\dfrac{r}{q}\right)$ \\
				\midrule
				$d_{1}$ & $7$ & $7$ & $7$ &  $-1$ \\
				$d_{2}$ & $7$ & $3$ & $7$ & $-1$  \\
				$d_{3}$ & $3$ & $3$ & $7$ & $-1$  \\
				$d_{4}$ & $3$ & $3$ & $3$ & $-1$ \\
				$d_{5}$ & $7$ & $3$ & $7$ & $1$\\
				$d_{6}$ & $7$ & $7$ & $3$ & $1$ \\
				$d_{7}$ & $3$ & $7$ & $3$ & $1$  \\
				$d_{8}$ & $3$ & $3$ & $7$ & $1$  \\
				\bottomrule
				
			\end{tabular}
			\caption{Conditions on $r$, $s$ and $q$ ensuring $h_2(qrs)=4$ in case $r\equiv s\equiv3\pmod4$.}
			\label{refddd}
		\end{table}
	\end{center}
	This justifies the   forms listed from $B)$-$3)$ to $B)$-$6)$.

	\item Let $r\equiv s \equiv 1\pmod4$. Notice that  if  $h_2(K)=2$, then the class field tower of $k$ is abelian, i.e., $k^{(1)}= k^{(2)}$. Thus,  using the classification in \cite[Form $c_i$, for $i\in \{15,...,20\}$, page 177-178]{BS95}, by interpreting the graphs corresponding to the abelian case to those of type 7  (for a definition of these graphs reader can see page 168 in the same reference), we deduce that $r$, $s$ and $q$ satisfy one of the conditions in Table \ref{qrsabelien}:
\end{enumerate}

\begin{center}
	\renewcommand{\arraystretch}{1.6}
	\begin{table}[H]
		\begin{tabular}{c|ccc|ccc}
			\toprule
			Cas
			& $r \pmod{8}$
			& $s \pmod{8}$
			& $q \pmod{8}$ & $\left(\dfrac{r}{s}\right)$
			& $\left(\dfrac{r}{q}\right)$ &  $\left(\dfrac{s}{q}\right)$ \\
			\midrule
			$c_{15}$ & $1$ & $5$ & $7$ & $-1$ & $-1$ & $1$ \\
			$c_{16}$ & $1$ & $5$ & $3$ & $-1$ & $-1$&$1$  \\
			$c_{17}$ & $5$ & $5$ & $7$ & $-1$ & $-1$ &$-1$ \\
			$c_{18}$ & $1$ & $5$ & $3$ & $-1$ & $-1$  &$-1$\\
			$c_{19}$ & $5$ & $5$ & $7$ & $-1$ & $1$  &$-1$\\
			$c_{20}$ & $5$ & $5$ & $3$ & $-1$ & $1$  &$-1$ \\
			\bottomrule
		\end{tabular}
		
		\caption{Conditions on $r$, $s$ and $q$ ensuring $h_2(qrs)=4$ in abelian $2$-tower case and $r\equiv s\equiv1\pmod4$} \label{qrsabelien}
	\end{table}
\end{center}

So let us assume that $q$, $r$ and $s$ satisfy one of the conditions presented in Table  \ref{qrsabelien}.
Notice that $h_2(qrs)=4$ and  since in all  cases we have $\left(\dfrac{r}{s}\right)=-1$,  it follows that $h_2(rs)=2$ (cf. Lemma \ref{realQuadeven1New}). Thus, by Kuroda's class number formula,  $h_2(K)=2$ if and only if $q(K)=1$.

On the other hand, using  Tables \ref{tab7} and \ref{tab8}, we check that  $E_K=\langle-1, \varepsilon,\varepsilon,\varepsilon\text{ or } \sqrt{\varepsilon}  \rangle$ according to whether $[r\equiv s\equiv5\pmod8, q\equiv7\pmod8   \text{ and } \left(\dfrac{q}{r}\right)=\left(\dfrac{r}{s}\right)=\left(\dfrac{s}{q}\right)=-1]$ or not. Hence, $h_2(K)=4$ or $2$, more precisely,
\begin{equation*}
	h_2(K)=4 \text{ if and only if } r\equiv s\equiv5\pmod8, q\equiv7\pmod8   \text{ and } \left(\dfrac{q}{r}\right)=\left(\dfrac{r}{s}\right)=\left(\dfrac{s}{q}\right)=-1.
\end{equation*}
Which 		  justifies the   forms listed from $B)$-$7)$ to $B)$-$9)$.

\indent$\bullet$	Finally let $d$ be an even positive square-free integer co-prime with $q$, i.e.,   $d=2d'$, where $d'$ is an odd positive square-free integer (co-prime with $q$). Notice that if $h_2(K)=2$, then the fact that the extension $K/\QQ(2qd')$ is ramified, implies that  $h_2(2qd')\in\{1,2\}$. It follows, by Lemma  \ref{realQuadeven1New}, that the integer $d'$ is in one of the following sub-cases:   $[d'=1]$, $[d'=p  \text{ is a prime number with }  p\equiv5\pmod8   \text{ or }  \left(\dfrac{p}{q}\right)=-1]$ or $[d'=q' \;\text{ is a prime numer with }   q'\equiv3\pmod8]$.

We note that the sub-case $d'=1$  is eliminated, as it implies, by  Lemma \ref{realBiquad},  that $h_2(K)=1$.
So let us assume that $d'=p$ is a prime number with  $p\equiv5\pmod8$ or  $\left(\dfrac{p}{q}\right)=-1$. If $p\equiv5\pmod8$, then $h_2(2p)=h_2(2pq)=2$ (cf. Lemma  \ref{realQuadeven1New}); and on the other hand, we have $E_K=\langle{-1,\vep_q,\vep_{2p}, \sqrt{\vep_{q}\vep_{2pq}}\rangle}$ (cf. Table \ref{tab7}), i.e.,  $q(K)=2$.
Thus, by Kuroda's class number formula, we have $h_2(K)=2$.  This justifies the form $C)$.

Now  let  $p\equiv1\pmod8$ and $\left(\dfrac{p}{q}\right)=-1$. So we have $q(K)\geq 2$ (In fact $\sqrt{\vep_{2p}\vep_{q}}$ or $\sqrt{\vep_{q}\vep_{2pq}} \in K$). As   Lemma  \ref{realQuadeven1New}  implies that $h_2(2pq)=2$ and $h_2(2p)\geq 4$, we  deduce that $h_2(K)\geq 4$.

To complete our proof let us consider that last sub-case, i.e., $d'=q'$ is a prime number such that $q'\equiv3\pmod8$. In this sub-case we have $E_K=\langle{-1,\vep_q,\sqrt{\vep_{q}\vep_{2q'}}, \sqrt{\vep_{q}\vep_{2qq'}}\rangle}$ (cf. Table \ref{tab7}) and $h_2(2qq')=2$ (cf. Lemma \ref{realQuadeven1New}), which gives $h_2(K)=2$. This justifies the form $D)$.

Hence, putting all of the above together, we obtain the complete classification stated in the theorem. This completes the proof.	
\end{proof}


\subsection{Real biquadratic fields  $K=\QQ(\sqrt{2q},\sqrt{d})$ with $2$-class numbers equal   $2$   } The theorem below yields a complete classification of  real biquadratic fields $K=\QQ(\sqrt{2q},\sqrt{d})$ whose $2$-class number is equal to $2$.
\begin{theorem}\label{class2qrs} Let $q\equiv3\pmod4$ be a prime number and $d$ a positive square-free integer. Put  $K=\QQ(\sqrt{2q},\sqrt{d})$. Then $h_2(K)=2$ if and only if, $d$ is in  one of the following forms: 
\begin{enumerate}[$A)$]
	\item $d=r,$ where $r$ is a prime number such that we have one of the following conditions:
	\begin{enumerate}[$1)$]
		\item $r\equiv3\pmod4$ and  $[ q\equiv3\pmod 8  \text{ or }   r\equiv3\pmod 8 ]$, 
		\item $r\equiv1\pmod4$ and $h_2(2qr)=4.$ 
	\end{enumerate}  
	\item $d=rs$, where $r$ and $s$ are prime numbers such that, after a suitable permutation of $r$ and $s$, we have one of the following conditions:
	
	\begin{enumerate}[\rm $1)$]
		
		\item $q\equiv s\equiv7\pmod{8}$, $r\equiv3\pmod{8}$ and
		$\left(\dfrac{q}{r}\right)=-\left(\dfrac{q}{s}\right)=\left(\dfrac{r}{s}\right)=1$,

		\item $r\equiv s\equiv3\pmod{8}$, $q\equiv7\pmod{8}$ and
		$\left(\dfrac{q}{r}\right)=-\left(\dfrac{q}{s}\right)=-\left(\dfrac{r}{s}\right)=-1$,
		\item $r\equiv5\pmod8$, $s\equiv1\pmod{8}$, $q\equiv3\pmod{8}$ and
		$\left(\dfrac{q}{r}\right)=-\left(\dfrac{q}{s}\right)=-\left(\dfrac{r}{s}\right)=1$,

		\item $r\equiv s\equiv5\pmod8$  and
		$\left(\dfrac{q}{r}\right)=-\left(\dfrac{q}{s}\right)=-\left(\dfrac{r}{s}\right)=1$,
		\item $r\equiv1\pmod8, s\equiv5\pmod8$  and
		$\left(\dfrac{q}{r}\right)=\left(\dfrac{q}{s}\right)=\left(\dfrac{r}{s}\right)=-1$.

	\end{enumerate}
	
	\item $d=2p,$   where $p\equiv1\pmod4$  prime number such that we have one of the following conditions:
	\begin{enumerate}[$1)$]
		\item $p\equiv5\pmod8$,
		\item $p\equiv1\pmod8$,  $\left(\frac{2}{p}\right)_4\left(\frac{p}{2}\right)_4=\left(\frac{p}{q}\right)=-1$.
	\end{enumerate}

	\item $d=2pq_1,$   where $p\equiv-q_1\equiv1\pmod4$ and $q_1\neq q$  are two prime numbers such that one of the following conditions holds:\begin{enumerate}[$1)$]
		\item $p\equiv5\pmod8$ and $\left(\dfrac{p}{q_1}\right)=-1$ or $[\left(\dfrac{p}{q_1}\right)=-\left(\dfrac{p}{q}\right)=1]$,
		
		\item $p\equiv1\pmod8$ and  $\left(\dfrac{p}{q_1}\right)=\left(\dfrac{p}{q}\right)=-1$. 
	\end{enumerate} 
\end{enumerate}
\end{theorem} 
\begin{proof}
Notice that if $h_2(K)=2$, then by Kuroda's class number formula     and Lemma \ref{realQuadodd}, we deduce that $h_2(2qd)\leq 8$; and so  by genus theory we conclude that $d$ has at most  {three} distinct prime divisors. So we  have the following cases to investigate:
\begin{enumerate}[$1)$]
	\item Let $d=r$ be a prime number that is co-prime with $2q$. Assume that $h_2(K)=2$. If $r\equiv3\pmod4$, the quadratic extension  $K/\QQ(2qr)$ is  ramified, and so    $h_2(2qr)=1$ or $2$. As there is no real quadratic field of this form that has an odd class number (cf. Lemma \ref{realQuadodd}), we have $h_2(2qr)=2$,  which is equivalent to $q$ or $r \equiv3\pmod8$ (cf. Lemma \ref{realQuadeven1New}). Conversely, if $r\equiv3\pmod 4$ and $q$ or $r \equiv3\pmod8$, then by using Wada's algorithm (cf. Page \pageref{wadaalgo}),   we check that $E_K=\langle{-1,\vep_r, \sqrt{\vep_{r}\vep_{2q}}, \sqrt{\vep_{2q}\vep_{2qr}}\rangle}$ (cf. Table \ref{tab7} and Lemma \ref{unitsrs}), i.e., $q(K)=4$. It follows, by  Kuroda's class number formula,   that $h_2(K)=2$.
	
	Now if $r\equiv1\pmod4$, then $K/\QQ(2qr)$ is an unramified quadratic extension. Therefore $h_2(K)=2$ if and only if  $h_2(2qr)=4$. This justifies the form $A)$. 
	\item  Let $d=rs$ for some distinct prime numbers $r$ and $s$ coprime with $2q$. 
	Put $k=\QQ(\sqrt{2qrs})$ and assume that  $h_2(K)=2$. Notice  that if $K/k$ is ramified, then $h_2(k)\in \{1,2\}$, which is impossible  according to  Lemmas \ref{realQuadodd} and \ref{realQuadeven1New}. Therefore,
	$K/k$ is unramified, and so  necessarily $rs\equiv1\pmod4$. Yields that we have two cases to investigate:
	
	\noindent\ding{229}  Assume that  $r\equiv s\equiv 3\pmod4$. According to \cite[Theorem 1]{BS95} we have $k^{(1)}=k^{(2)}$, and so 
	$	h_2(K)=\frac{h_2(k)}{2}$. Therefore,    
	$h_2(K)=2$ if and only if $h_2(k)=4$.
	On the other hand, the quadratic fields $k=\QQ(\sqrt{2qrs})$ such that $h_2(k)=4$ can be deduced from  \cite[Form $c_i$, for $i\in \{1,...,3\}$,  page 11]{BS25I}. In fact, we have $h_2(k)=4$ if and only if, after a suitable permutation of  $q, r$ and $s$, we have $r\equiv3\pmod8$,  $q\equiv7\pmod8$,   $\left(\dfrac{r}{s}\right)=1$ and one of the   conditions in Table \ref{table5}.
	
	{\footnotesize	\begin{center}
			\renewcommand{\arraystretch}{1.6}
			\begin{table}[H]
				\begin{tabular}{c|c|cc}
					\toprule
					Cas
					
					& $s \pmod{8}$
					&  $\left(\dfrac{q}{r}\right)$&  $\left(\dfrac{q}{s}\right)$ \\
					\midrule
					$c_{1}$ & $7$ & $1$& $-1$ \\
					$c_{2}$ & $7$ & $1$& $1$ \\
					$c_{3}$ & $3$ & $-1$& $1$ \\
					\bottomrule
				\end{tabular}
				\caption{Conditions on $r$, $s$ and $q$ ensuring $h_2(2qrs)=4$ in  case  $r\equiv s\equiv3\pmod4$}\label{table5}
			\end{table}
	\end{center}}

	\noindent\ding{229} Now assume that  $r\equiv s \equiv 1\pmod4$. Under the hypothesis $h_2(K)=2$, the 2-class field tower of $k$ is abelian,  i.e.,  $k^{(1)}= k^{(2)}$. Hence,  using the classification in \cite[Form $c_i$, for $i\in \{19,...,22,25,26\}$, page 176]{BS95}, by interpreting the graphs corresponding to the abelian case to those of type 6, we have :

	{	\begin{center}
			\footnotesize
			\begin{table}[H]
				\renewcommand{\arraystretch}{1.4}
				\begin{tabular}{c|ccc|ccc}
					\toprule
					Cas
					& $r \pmod{8}$
					& $s \pmod{8}$
					& $q \pmod{8}$ & $\left(\dfrac{r}{s}\right)$
					& $\left(\dfrac{r}{q}\right)$ &  $\left(\dfrac{s}{q}\right)$ \\
					\midrule
					$c_{19}$ & $5$ & $1$ & $7$ & $-1$ & $1$ & $1$ \\
					$c_{20}$ & $5$ & $1$ & $3$ & $-1$ & $1$&$-1$  \\
					$c_{21}$ & $5$ & $5$ & $7$ & $-1$ & $1$ &$-1$ \\
					$c_{22}$ & $5$ & $5$ & $3$ & $-1$ & $1$  &$-1$\\
					$c_{25}$ & $1$ & $5$ & $7$ & $-1$ & $-1$  &$-1$\\
					$c_{26}$ & $1$ & $5$ & $3$ & $-1$ & $-1$  &$-1$ \\
					\bottomrule
				\end{tabular}
				\caption{Conditions on $r$, $s$ and $q$ ensuring $h_2(2qrs)=4$ in abelian case and $r\equiv s\equiv1\pmod4$}\label{abelien2qrs}
			\end{table}
	\end{center}}

	Conversely, every row of Table \ref{abelien2qrs} guarantees that \(h_2(k)=4\). 
	Hence, to obtain \(h_2(K)=2\), it remains to verify that \(k^{(1)}=k^{(2)}\). 
	Therefore, by \cite[Theorem 2]{BLS98}, apart from the case \(c_{19}\), 
	the 2-class tower of \(k\) is abelian in all the remaining cases.

	This justifies the form $B)$.
	\item Let $d=rst$, where $r$, $s$ and $t$ are distinct prime numbers co-prime with $2q$. By genus theory, we obtain $h_2(rst)\geq 2$ and $h_2(2qrst)\geq 8$. It follows that  $h_2(K)\geq 4$, which forces the rejection of this case.
	
	\item Let $d=2d'$, for an odd positive integer $d'$ coprime with $q$. Assume that $h_2(K)=2$. As $K/\QQ(\sqrt{qd'})$ and $K/\QQ(\sqrt{2d'})$ ramifies, we have 
	$h_2(\QQ(\sqrt{qd'})), h_2(\QQ(\sqrt{2d'})) \in \{1,2\}\}$. It follows,  by Lemmas \ref{realQuadodd} and \ref{realQuadeven1New}, that we have:
	\[
	d' = \begin{cases}
		1, \;\text{ or }  \\
		q_1,  \text{ with }  q_1\equiv3\pmod 4,  \text{ or } \\
		p,  \text{ with }  p\equiv5\pmod 8  \text{ or } [  p\equiv1\pmod 8 \text{ and } \left(\frac{2}{p}\right)_4\left(\frac{p}{2}\right)_4=\left(\frac{p}{q}\right)=-1 ],  \text{ or } \\
		q_1p,   \text{ with }  \left(\frac{p}{q_1}\right)=-1   \text{ or }   [ p\equiv5\pmod8  \text{ and }  \left(\frac{p}{q}\right)=-1].
	\end{cases}
	\]
	
	Where $p\equiv-q_1\equiv1\pmod4$ distinct prime numbers.
	For the 	converse, if  $d'=1$  then $h_2(K)=1$ (cf. Lemma \ref{realBiquad}) and so this case is rejected. 
	Moreover, if  $d'=q_1$ for a prime number $q_1\equiv3\pmod4$ distinct to $q$, then according to Lemma \ref{realQuadodd} and Kuroda's class number formula, we have $h_2(K)=q(K)/4$ and  making use of  Table \ref{tab7} and    Wada's algorithm (cf. Page \pageref{wadaalgo}),  we check that  $E_K=\langle{-1,\vep_{2q},\sqrt{\vep_{2q}\vep_{2q_1}}, \sqrt{\vep_{2q}\vep_{qq_1}}\rangle}$, this gives $h_2(K)=1$. So the second case is also rejected.
	
	Now let $d'$ takes the third form,  i.e., $d'=p$, such that  $p\equiv5\pmod 8$  \text{ or } $[  p\equiv1\pmod 8 \text{ and } \left(\frac{2}{p}\right)_4\left(\frac{p}{2}\right)_4=\left(\frac{p}{q}\right)=-1 ]$. By Kuroda's class number formula and Lemma \ref{realQuadeven1New}, we have : $h_2(K)=\frac{q(K)h_2(2p)h_2(pq)}{4}=q(K).$
	Moreover, by making	use of Table \ref{tab7} and Lemma \ref{unitsrs}, we check that if $p\equiv5\pmod8$ then  $E_K=\langle{-1,\vep_{2q},\vep_{2p},\sqrt{\vep_{2q}\vep_{pq}} \;\text{or}\; \sqrt{\vep_{pq}}\rangle}$ according to whether $\left(\frac{p}{q}\right)=1$ or not, thus $h_2(K)=2$. But if  $  p\equiv1\pmod 8 \text{ and } \left(\frac{2}{p}\right)_4\left(\frac{p}{2}\right)_4=\left(\frac{p}{q}\right)=-1$, we check that $E_K=\langle{-1,\vep_{2q},\vep_{2p},  \sqrt{\vep_{2q}\vep_{2p}} \rangle}$ (notice that accordint to \cite[Theorem 2]{kuvcera1995parity} we have $N(\vep_{2p})=1$). Therefore  $h_2(K)=2$.
	This justifies the form $C)$.

	Finally let $d'=q_1p$, with  $\left(\frac{p}{q_1}\right)=-1   \text{ or }   [ p\equiv5\pmod8  \text{ and }  \left(\frac{p}{q}\right)=-1]$. Then $h_2(K)=\frac{q(K)h_2(2pq_1)h_2(pqq_1)}{4}=q(K).$ We begin by the first case where $\left(\frac{p}{q_1}\right)=-1,$ then if $p\equiv1\pmod8$ we have $E_K=\langle{-1,\vep_{2q },\sqrt{\vep_{2q}\vep_{2pq_1}}, \vep_{pqq_1}\; \text{or}\; \sqrt{\vep_{2q}\vep_{pqq_1}} \rangle}$ according to whether $\left(\frac{p}{q}\right)=-1$ or not. But if $p\equiv5\pmod8$ we have $E_K=\langle{-1,\vep_{2q },\vep_{2pq_1},\sqrt{\vep_{2q}\vep_{pqq_1}}\; \text{or}\; \sqrt{\vep_{2pq_1}\vep_{pqq_1}} \rangle}$ according to whether $\left(\frac{p}{q}\right)=1$ or not  (cf. Tables \ref{tab7} and \ref{tab8}). Therefore, if $\left(\frac{p}{q_1}\right)=-1$, we have $h_2(K)=2$ if and only if $p\equiv5\pmod8$ or $[p\equiv1\pmod8 \;\text{and}\; \left(\frac{p}{q}\right)=-1]$. At the end if $\left(\frac{p}{q_1}\right)=1$, $p\equiv5\pmod8  \text{ and }  \left(\frac{p}{q}\right)=-1$, we check that $E_K=\langle{-1,\vep_{2q },\vep_{2pq_1}, \sqrt{\vep_{2pq_1}\vep_{pqq_1}} \rangle}$ (cf. Tables \ref{tab7} and \ref{tab8}), then $h_2(K)=2$.  This justifies the form $D)$.

	\item Now let $d=qd'$ for an odd positive integer $d'$ co-prime with $q$. Then \(K=\QQ(\sqrt{2q},\sqrt{qd'})=\QQ(\sqrt{2q},\sqrt{2d'})\), and this case has already been treated in the preceding case.

\end{enumerate}
Plugging  all of the above investigations together, we   complete the proof of our classification.
\end{proof}


\subsection{Real biquadratic fields  $K=\QQ(\sqrt{q_1q_2},\sqrt{d})$ with $2$-class numbers equal   $2$ }
The following theorem provides a characterization of the fields $K=\QQ(\sqrt{q_1q_2},\sqrt{d})$ whose  class number is 2 where $q_1\equiv q_2 \equiv3\pmod4$ prime numbers. 	

\begin{theorem}\label{classq1q2d}
Let $q_1\equiv q_2 \equiv3\pmod4$ be distinct prime numbers and   $d$ a positive square-free integer. Put $K=\QQ(\sqrt{q_1q_2},\sqrt{d})$.	Then $h_2(K)=2$ if and only if $d$ is in one of the following forms : 
\begin{enumerate}[$A)$]
\item $d=2,$ such that $q_1\equiv q_2\equiv7\pmod8$ and  $h_2(2q_1q_2)=4.$
\item $d=p,$ where $p\equiv1\pmod4$ is a prime number such that     $\left(\frac{p}{q_1}\right)=\left(\frac{p}{q_2}\right)=1 $ and $h_2(pq_1q_2)=4$.
\item $d=2p,$ where $p\equiv1\pmod4$ is a prime  number such that one of the following conditions holds:
\begin{enumerate}[$1)$]
	\item $p\equiv5\pmod8,q_1\equiv q_2\equiv3\pmod8$ and $\left(\frac{p}{q_1}\right)=-\left(\frac{p}{q_2}\right)=-1,$ or
	
	\item $p\equiv5\pmod8,q_1\equiv7\pmod8, q_2\equiv3\pmod8$ and $\left(\frac{p}{q_1}\right)=-\left(\frac{p}{q_2}\right)=\left(\frac{q_2}{q_1}\right)=-1.$
\end{enumerate}

\item $d=q_3,$ where $q_3\equiv3\pmod4$ is a prime number such that  one of the following conditions holds:
\begin{enumerate}[$1)$]
	\item $q_1\equiv q_3 \equiv 7\pmod{8}$, $q_2\equiv3\pmod4$ and  $\left(\dfrac{q_3}{q_1}\right)=-\left(\dfrac{q_3}{q_2}\right)=\left(\dfrac{q_1}{q_2}\right)=1$,
	\item  $q_1\equiv q_2 \equiv 3\pmod{8}$ and  $\left(\dfrac{q_3}{q_1}\right)=-\left(\dfrac{q_3}{q_2}\right)=\left(\dfrac{q_1}{q_2}\right)=1$,
	
	
	\item  $q_2\equiv7\pmod8$, $q_3 \equiv 3\pmod{8}$, $q_1\equiv3\pmod4$  and  $\left(\dfrac{q_3}{q_1}\right)=\left(\dfrac{q_3}{q_2}\right)=-\left(\dfrac{q_1}{q_2}\right)=-1$,
	
	
	\item $q_2\equiv3\pmod8$, $q_3 \equiv 7\pmod{8}$, $q_1\equiv3\pmod4$  and  $\left(\dfrac{q_3}{q_1}\right)=\left(\dfrac{q_3}{q_2}\right)=-\left(\dfrac{q_1}{q_2}\right)=-1$.
	
	
\end{enumerate}
\item $d=2q_3,$ where $q_3\equiv3\pmod4$ is a prime number such that one of the following conditions holds:
\begin{enumerate}[$1)$]
	\item $q_3\equiv q_2\equiv7\pmod{8}$, $q_1\equiv3\pmod{8}$, and
	$\left(\dfrac{q_3}{q_1}\right)=-\left(\dfrac{q_3}{q_2}\right)=\left(\dfrac{q_1}{q_2}\right)=1$.

	\item $q_1\equiv q_2\equiv3\pmod{8}$, $q_3\equiv7\pmod{8}$, and
	$\left(\dfrac{q_3}{q_1}\right)=-\left(\dfrac{q_3}{q_2}\right)=-\left(\dfrac{q_1}{q_2}\right)=-1$.
\end{enumerate}
\item $d=p_1p_2,$ where $p_1\equiv p_2\equiv1\pmod4$ are distinct prime numbers such that one of the following conditions holds:
\begin{enumerate}[$1)$]
	\item $\left(\frac{p_1}{p_2}\right)=\left(\frac{p_1}{q_1}\right)=\left(\frac{p_1}{q_2}\right)=\left(\frac{p_1}{q_1}\right)=-\left(\frac{p_1}{q_2}\right)=-1,$
	\item $\left(\frac{p_1}{p_2}\right)=-\left(\frac{p_1}{q_1}\right)=\left(\frac{p_1}{q_2}\right)=\left(\frac{p_1}{q_1}\right)=-\left(\frac{p_1}{q_2}\right)=\left(\frac{q_1}{q_2}\right)=-1.$
\end{enumerate}
\item $d=q_3q_4,$ where $q_3\equiv q_4\equiv3\pmod4$ are distinct prime numbers such that one of the following conditions holds:
\begin{enumerate}[$1)$]
	\item $\left(\frac{q_1}{q_2}\right)=\left(\frac{q_1}{q_4}\right)=\left(\frac{q_4}{q_3}\right)=-\left(\frac{q_1}{q_3}\right)=-\left(\frac{q_2}{q_4}\right)=1,$ or
	\item $\left(\frac{q_1}{q_2}\right)=\left(\frac{q_1}{q_4}\right)=\left(\frac{q_4}{q_3}\right)=\left(\frac{q_1}{q_3}\right)=-\left(\frac{q_2}{q_3}\right)=-\left(\frac{q_2}{q_4}\right)=1.$
\end{enumerate}
\item $d=2q_1q_3,$ where $q_3\equiv3\pmod8$ is a prime number coprime with $q_1q_2$.
\item $d=pq_1,$ where $p\equiv1\pmod4$ is a prime number such that $p\equiv5\pmod8$ or $\left(\frac{p}{q_1}\right)=\left(\frac{p}{q_2}\right)=-1$.
\item $d=2pq_1,$  where $p\equiv1\pmod4$ is a prime number such that one of the following conditions holds:
\begin{enumerate}[$1)$]
	\item $\left(\frac{p}{q_1}\right)=\left(\frac{p}{q_2}\right)=-1,$ 
	\item $\left(\frac{p}{q_1}\right)=-\left(\frac{p}{q_2}\right)=1$ and $p\equiv5\pmod8.$
\end{enumerate}
\end{enumerate}
\end{theorem} 

\begin{proof}
Let $d$ be a positive square-free integer. We have two main cases to investigate.

\noindent\ding{229} Assume that $\mathrm{gcd}(d,q_1q_2)=1$.  Put $k=\QQ(\sqrt{q_1q_2d})$.  We note that $K/k$ is an unramified quadratic extension. So by Remark \ref{abelianquaexten}, we have 
$h_2(K)=2$ if and only if $h_2(k)=4$ and $k^{(1)}=k^{(2)}$. In particular, 
if $h_2(K)=2$, then the $2$-class group of $k$  is cyclic or it is the Klein $4$-group; and so by genus theory, we have one of the following forms for $d$:   
\begin{center}
[$d=2$ or $d=p$]   or 	[$d= q_3 $, $d= 2q_3 $, $d= 2p$, $d= p_1p_2 $ or $d= q_3q_4$], 
\end{center}
where $p\equiv p_i\equiv -q_i\equiv1\pmod4$ are distinct prime numbers. Now let us explore each form.
\begin{enumerate}[\noindent$i)$]
\item Let $d=2$ (resp.  $d=p$). Since $A(k)$ is cyclic, 
we deduce from class field theory and Lemma \ref{realQuadeven1New} that $h_2(K)=2$ if and only if
$q_1\equiv q_2\equiv 7\pmod 8$ and  $h(2q_1q_2)=4$ (resp.     $\left(\frac{p}{q_1}\right)=\left(\frac{p}{q_2}\right)=1 $ and $h_2(pq_1q_2)=4$). This justifies the forms $A)$ and $B)$ respectively.

\item Let $d=2p$. Since $K/\QQ(\sqrt{d})$ ramify, this implies that $h_2(2p) = 2$, and thus $p\equiv5\pmod8$ (cf. Lemmas \ref{realQuadeven1New} and \ref{realQuadodd}). According to \cite[page 175, cases $c_{12}-c_{14}$]{BS95}, we have    $k^{(1)}=k^{(2)}$  and $h_2(k)=4$  if  and only if one of the following conditions holds:
\begin{enumerate}[$\mathrm{C}1)$]
	\item $p\equiv5\pmod8,q_1\equiv q_2\equiv3\pmod8$ and $\left(\frac{p}{q_1}\right)=-\left(\frac{p}{q_2}\right)=-1,$ or
	
	\item $p\equiv5\pmod8,q_1\equiv7\pmod8, q_2\equiv3\pmod8$ and $\left(\frac{p}{q_1}\right)=-\left(\frac{p}{q_2}\right)=\left(\frac{q_2}{q_1}\right)=-1.$
\end{enumerate}
This justifies the form $C)$.
\item Let $d=q_3$. By Theorem \ref{classqrs}, $h_2(K)=2$ if and only if we have one of the conditions listed in $D)$.  
\item Let $d=2q_3.$  By Lemma \ref{realQuadeven1New},  $h_2(K)=2$ if and only if we have one of the conditions listed in $E)$.  
\item Let $d=p_1p_2.$   According to \cite[page 175, cases $c_{12}-c_{14}$]{BS95} we have $h_2(p_1p_2q_1q_2)=4$ and $k^{(1)}=k^{(2)}$  if and only if we have one of the following conditions:
\begin{enumerate}[$\mathrm{C}1)$]
	\item $\left(\frac{p_1}{p_2}\right)=\left(\frac{p_1}{q_1}\right)=\left(\frac{p_1}{q_2}\right)=\left(\frac{p_1}{q_1}\right)=-\left(\frac{p_1}{q_2}\right)=-1,$
	\item $\left(\frac{p_1}{p_2}\right)=-\left(\frac{p_1}{q_1}\right)=\left(\frac{p_1}{q_2}\right)=\left(\frac{p_1}{q_1}\right)=-\left(\frac{p_1}{q_2}\right)=\left(\frac{q_1}{q_2}\right)=-1.$
\end{enumerate}
This justifies the form $F)$.
\item Let $d=q_3q_4.$ We note that according \cite[Theorem 1]{BS95} the $2$-class field of $k$ is abelian. Moreover,  using \cite[page 11, $c_1-c_3$]{BS25I} we deduce that $h_2(q_1q_2q_3q_4)=4$ if and only if one of the following conditions holds:  
\begin{enumerate}[$\mathrm{C}1)$]
	\item $\left(\frac{q_1}{q_2}\right)=\left(\frac{q_1}{q_4}\right)=\left(\frac{q_4}{q_3}\right)=-\left(\frac{q_1}{q_3}\right)=-\left(\frac{q_2}{q_4}\right)=1,$ or
	\item $\left(\frac{q_1}{q_2}\right)=\left(\frac{q_1}{q_4}\right)=\left(\frac{q_4}{q_3}\right)=\left(\frac{q_1}{q_3}\right)=-\left(\frac{q_2}{q_3}\right)=-\left(\frac{q_2}{q_4}\right)=1.$
\end{enumerate}
This justifies the form $G)$.
\end{enumerate}

\noindent\ding{229} Now assume that $\mathrm{gcd}(d,q_1q_2)=q_1$.    Put $k=\QQ(\sqrt{q_2d'})$, where $d'=\frac {d}{q_1}$. As $K/k$ ramify, the equality $h_2(K)=2$ implies that $h_2(k)\in\{1,2\}$; and thus  Lemmas  \ref{realQuadodd} and \ref{realQuadeven1New} imply that: 
\[
d'\in \{1,2,q,2q_3,p,2p\},
\]
where $q\equiv q_3\equiv -p\equiv3\pmod4$ are  prime numbers satisfying some extra conditions as below. For the converse,
let us explore these forms of $d'$. Firstly, notice that the forms $d'=1$, $2$ and $q$ are eliminated. In fact,  each of these cases gives $h_2(K)=1$ (cf. Lemma \ref{realBiquad}). Now let us investigate each of the remaining forms.
\begin{enumerate}[$i)$]
\item $d'=2q_3,$ for some prime number $q_3\equiv3\pmod8$ distinct from $q_1$ and $q_2$. Using Table \ref{tab7} and  Wada's algorithm (cf. Page \pageref{wadaalgo})  we conclude that  $E_K= \langle{-1,\vep_{q_1q_2},\vep_{2q_1q_3},\eta\rangle}$, where $$\eta \in \{\sqrt{\vep_{2q_1q_3}\vep_{2q_2q_3}},\sqrt{\vep_{q_1q_2}\vep_{2q_2q_3}},\sqrt{\vep_{q_1q_2}\vep_{2q_2q_3}\vep_{2q_1q_3}}\}.$$  Therefore $h_2(K)=2$. This justifies the form $H)$. 
\item Let $d'=p,$ for some prime number $p\equiv1\pmod4$ such that $p\equiv5\pmod8$ or $\left(\frac{p}{q_1}\right)=\left(\frac{p}{q_2}\right)=-1$. From Table \ref{tab7} we get $E_K= \langle{-1,\vep_{q_1q_2},\vep_{pq_1}, \sqrt{\vep_{pq_1}\vep_{pq_2}}\;or\;\sqrt{\vep_{pq_1}\vep_{pq_2}\vep_{q_1q_2}}\rangle}$. So by Kuroda's class number formula and Lemma  \ref{realQuadeven1New}, we have $h_2(K)=2$. This justifies the form $I)$. 
\item Now let $d'=2p,$ for some prime number $p\equiv1\pmod4$ such that $p\equiv5\pmod8$ or $\left(\frac{p}{q_1}\right)=\left(\frac{p}{q_2}\right)=-1$. Notice that $q(K)=4$ or $2$ according to whether  $p\equiv5\pmod8$ and $\left(\frac{p}{q_1}\right)=\left(\frac{p}{q_2}\right)=1$ or not. In fact using Table \ref{tab7} and  Wada's algorithm (cf. Page \pageref{wadaalgo}) we check that:

\begin{center}
	\begin{table}[H]
		\begin{tabular}{c|cc|c}
			\toprule
			$p\pmod8$ & $\left(\dfrac{p}{q_1}\right)$ & $ \left(\dfrac{p}{q_2}\right)$ & $E_K$ \\
			\toprule
			$5$	&  $1$&$1$ &  $\langle{-1,\vep_{q_1q_2},\sqrt{\vep_{2pq_1}\vep_{2pq_2}}, \sqrt{\vep_{q_1q_2}\vep_{2pq_1}}\rangle}$ \\
			\midrule
			$5$	&  $-1$&$-1$ &  $\langle{-1,\vep_{q_1q_2},\vep_{2pq_1},\sqrt{\vep_{2pq_1}\vep_{2pq_2}}\rangle}$ \\
			
			\midrule
			$5$	&  $1$&$-1$ &  $\langle{-1,\vep_{q_1q_2},\vep_{2pq_2},\sqrt{\vep_{q_1q_2}\vep_{2pq_1}}\rangle}$  \\
			
			\midrule
			$1$	&  $-1$&$-1$ &  $\langle{-1,\vep_{q_1q_2},\vep_{2pq_1},\sqrt{\vep_{2pq_1}\vep_{2pq_2}}\rangle}$ \\

			\bottomrule
		\end{tabular}
		\caption{Unit group of $K$ when $d=2pq_1$}
	\end{table}
\end{center}
Therefore by applying Kuroda's class number formula and Lemma  \ref{realQuadeven1New}, we deduce that $h_2(K)=2$ if and only if \begin{enumerate}[$\mathrm{C}1)$]
	\item $\left(\frac{p}{q_1}\right)=\left(\frac{p}{q_2}\right)=-1,$  or
	\item $\left(\frac{p}{q_1}\right)=-\left(\frac{p}{q_2}\right)=1$ and $p\equiv5\pmod8.$
\end{enumerate}
This justifies the form $J)$.  		This completes the proof of our theorem.
\end{enumerate}
\end{proof}


\subsection{Real biquadratic fields  $K=\QQ(\sqrt{2},\sqrt{d})$ with $2$-class numbers equal   $2$  } The following theorem provides a complete classification of the real biquadratic fields $K=\QQ(\sqrt{2},\sqrt{d})$ whose $2$-class number is equal to $2$.	
\begin{theorem}\label{class2d}
Let $K=\QQ(\sqrt{2},\sqrt{d})$, where  $d$ is an odd  positive square-free integer. 
Then $h_2(K)=2$ if and only if, $d$ is in  one of the following forms: 
\begin{enumerate}[$A)$]
\item $d=pq,$ where $p\equiv-q\equiv1\pmod4$ are prime numbers such that  $p\equiv5\pmod8$ or  $\left(\frac{p}{q}\right)=-1.$
\item $d=p,$ where $p$ is a prime number such that $p\equiv1\pmod8$, $\left(\frac{2}{p}\right)_4=(-1)^{\frac{p-1}{8}}$ and $h_2(2p)=4.$
\item $d=q_1q_2,$ where $q_1\equiv q_2\equiv7\pmod8$ are distinct prime numbers such that  $h_2(2q_1q_2)=4.$
\item $d=p_1p_2$, where $p_1\equiv p_2\equiv1\pmod4$ are distinct prime numbers satisfying one of the following conditions:

\begin{enumerate}[$1)$]
	\item $\left(\frac{p_1}{p_2}\right)=-1$, $h(2p_1p_2)=4$ and \textbf{none} of the two units $\{\vep_{2}\vep_{p_1p_2}\vep_{2p_1p_2}, \vep_{2p_1p_2}\}$ is a square in $K$,
	\item $\left(\frac{p_1}{p_2}\right)=1$ and $\left(\frac{p_1}{p_2}\right)_4\not=\left(\frac{p_2}{p_1}\right)_4$,
	$h(2p_1p_2)=4$ and \textbf{none} of the two units $\{ \vep_{p_1p_2}\vep_{2p_1p_2}, \vep_{2p_1p_2}\}$ is a square in $K$.
\end{enumerate}

\item $d=pq_1q_2,$ where $p\equiv-q_1\equiv-q_2\equiv1\pmod4$ are distinct prime numbers such that, after a suitable permutation of $q_1$ and $q_2$, we have one of the following conditions:
\begin{enumerate}[$1)$]
	\item $p\equiv5\pmod8$, $q_1\equiv q_2\equiv3\pmod8$ and $\left(\frac{p}{q_2}\right)=-\left(\frac{p}{q_1}\right)=1,$
	\item $p\equiv5\pmod8$, $q_1\equiv 7\pmod8$, $q_2\equiv3\pmod8$ and $\left(\frac{p}{q_2}\right)=-\left(\frac{p}{q_1}\right)=-\left(\frac{q_2}{q_1}\right)=1.$
\end{enumerate}
\end{enumerate}
\end{theorem} 
\begin{proof}
Let $K=\QQ(\sqrt{2},\sqrt{d})$, with $d$   an odd positive square-free integer. 
Our proof consists of two parts, according to the equivalence class of $d  \pmod4$. 

\noindent\ding{229} Assume that $d\equiv3\pmod4$. So $K/\QQ(\sqrt{2d})$ ramify and thus  the conditions $h_2(K)=2$, implies that we have   $h_2(2d)=1$ or $2$. 
Conversely, notice that if $h_2(2d) =1$,   then $d=q\equiv3\pmod4$ is a prime number (cf. Lemma \ref{realQuadodd}); and  this implies that the class number of $K$ is odd (cf. Lemma \ref{realBiquad}). So this case eliminated.

Now let $h_2(2d)=2$. So Lemma \ref{realQuadeven1New}  implies that  $d=pq$,  where  $p\equiv-q\equiv1\pmod4$ are prime numbers such that $p\equiv5\pmod8$ or  $\left(\frac{p}{q}\right)=-1.$ 
Thus, we have $h(pq)=h(2pq)=2$ (cf. Lemma \ref{realQuadeven1New}) and  $q(K)=2$  (In fact $E_K=\langle{-1,\vep_{2},\vep_{pq}, \sqrt{\vep_{pq}\vep_{2pq}}\rangle}$   by Wada's algorithm and Table \ref{tab7}). Consequently, Kuroda's class number formula gives $h_2(K)=2$. This justifies the form $A)$.

\noindent\ding{229} Now assume that $d\equiv1\pmod4$. Thus, $K/\QQ(\sqrt{2d})$ is an unramified quadratic extension.   Notice that if  $h_2(K)=2$, then  $h_2(2d)= 4$    and so $h_2(d)q(K)=2$. This implies that $h_2(d)=1$ or $2$. So we have two sub-cases to investigate.
\begin{enumerate}[$\star$]
\item Assume that 	$h_2(d)=1$.  So, by Lemma \ref{realQuadodd}, we have $d=p$, where $p \equiv1\pmod4$ is a prime number, or $d=q_1q_2$, where $ q_1\equiv q_2\equiv3\pmod4$ are distinct prime numbers. Moreover, we have:
\begin{enumerate}[$a)$]
	\item If $d=p$,  then   \cite[Theorem 2]{kuvcera1995parity}  implies that, $p\equiv1\pmod 8$ and $\left(\frac{2}{p}\right)_4=(-1)^{\frac{p-1}{8}}$.  Therefore, by class field theory,
	$h_2(K)=2$ if and only if  $h_2(2p)=4$.  This justifies the form $B)$.
	\item Now if $d=q_1q_2$, then, by class field theory and  Lemma \ref{realQuadeven1New},  $h_2(K)=2$ if and only if $q_1\equiv q_2\equiv7\pmod8$ and $h_2(2q_1q_2)=4$. This justifies the form $C)$.
\end{enumerate}

\item Assume that $h_2(d)=2$. By Lemma  \ref{realQuadeven1New}, we have one of the following   cases:
	\begin{enumerate}[$a)$]
		\item 	If $d$ takes the   first form, i.e., $d=p_1p_2$ $\text{ with }  p_1\equiv p_2\equiv1\pmod 4  \text{ and }  \left(\frac{p_1}{p_2}\right)=-1$, then according to \cite[Theorem 1]{kuvcera1995parity} we have $N_{\QQ(\sqrt{p_1p_2})/\QQ}(\vep_{p_1p_2})=-1$ and   using  \cite[Theorem 2.2]{AzZektaous} we deduce  that  
		$q(K)=2$ or $1$, according to whether one of the units $\{ {\vep_{2}\vep_{p_1p_2}\vep_{2p_1p_2}} ,   {\vep_{2p_1p_2}}\}$ is a square in $K$ or not. 
		
		Since by Kuroda's class number formula, we have $h_2(K)=\frac12 q(K)h_2(2p_1p_2)$, we deduce that 
		$h_2(K)=2$ if and only if $q(K)=1$  and  $h_2(2p_1p_2)=4$. This justifies the form $D)-1)$.

		\item Now if  $d$ takes  the second form, i.e.,   $d=p_1p_2$ $\text{ with }  p_1\equiv p_2\equiv1\pmod 4  \text{ and }   [\left(\frac{p_1}{p_2}\right)=1$ and $\left(\frac{p_1}{p_2}\right)_4\not=\left(\frac{p_2}{p_1}\right)_4]$, then according to \cite[Theorem 1]{kuvcera1995parity} we have $N_{\QQ(\sqrt{p_1p_2})/\QQ}(\vep_{p_1p_2})=-1$ and   using  \cite[Theorem 2.2]{AzZektaous} we deduce  that  
		$q(K)=2$ or $1$, according to whether one of the units	$\{ {\vep_{p_1p_2}\vep_{2p_1p_2}} ,   {\vep_{2p_1p_2}}\}$ is a square in $K$ or not.
		Thus,	Kuroda's class number formula gives $h_2(K)=\frac12 q(K)h_2(2p_1p_2)$. Therefore,
		$h_2(K)=2$ if and only if $q(K)=1$  and  $h_2(2p_1p_2)=4$. This justifies the form $D)-2)$.

		\item  Now if  $d$ takes  the third form, i.e., $d=pq_1q_2$ with $q_1\equiv q_2\equiv-p\equiv3\pmod 4$ and  $(\left(\frac{p}{q_1}\right)=-1$ or $\left(\frac{p}{q_2}\right)=-1)$, then  according to \cite[  Page 175 : $c_{12}-c_{14}$]{BS95}, we have  Table \ref{2dabelien} that characterizes when do we have exactly $h(2pq_1q_2)=4$. 
		{\footnotesize	\begin{center}
				\renewcommand{\arraystretch}{1.6}
				\begin{table}[H]
					\begin{tabular}{c|ccc|ccc}
						\toprule
						Case
						& $p \pmod{8}$
						& $q_1 \pmod{8}$
						& $q_2 \pmod{8}$ & $\left(\dfrac{q_2}{p}\right)$
						& $\left(\dfrac{q_2}{q_1}\right)$ &  $\left(\dfrac{p}{q_1}\right)$ \\
						\midrule
						
						$c_{12}$ & $5$ & $3$ & $3$ & $1$ & $-1$  &$-1$\\
						$c_{13}$ & $5$ & $3$ & $3$ & $1$ & $1$  &$-1$\\
						$c_{14}$ & $5$ & $7$ & $3$ & $1$ & $-1$  &$-1$ \\
						\bottomrule
					\end{tabular}
					\caption{Conditions on $p$, $q_1$ and $q_2$ ensuring $h_2(2pq_1q_2)=4$ in abelian $2$-tower case where $p\equiv -q_1\equiv-q_2\equiv1\pmod4$}\label{2dabelien}\label{abelien2pq1q2}
				\end{table}
		\end{center}}
		On the other hand, assuming that $d=2pq_1q_2$, where $p$, $q_1$ and $q_2$ are prime numbers satisfying one of   the conditions in  Table \ref{abelien2pq1q2}, we make use of Table  \ref{tab8} and Wada's method to check that we have $q(K)=1$. Thus under each of these conditions we have $h_2(K)=2$.  This justifies the form $E)$. 	
	\end{enumerate}
	Combining all of the above investigations,   we complete  the proof of our classification in the above theorem.	
	\end{enumerate}	\end{proof}
	

\subsection{Real biquadratic fields  $K=\QQ(\sqrt{p},\sqrt{d})$ with $2$-class numbers equal   $2$ }In this subsection, we give the classification of the fields $K=\QQ(\sqrt{p},\sqrt{d})$ whose  class numbers   equal $2$.
Let us start by proving  the following useful lemma.
\begin{lemma}\label{unitspp1q1q2}
Let $p$, $p_1$, $q_1$  and $q_2$ be distinct prime numbers such that $p\equiv p_1\equiv-q_1\equiv-q_2\equiv1\pmod4$. Put $\vep_{pp_1q_1q_2}=a+b\sqrt{pp_1q_1q_2}$ (resp. $\vep_{p_1q_1q_2}=x+y\sqrt{p_1q_1q_2}$)   the fundamental unit of $\QQ(\sqrt{pp_1q_1q_2})$ (resp. $\QQ(\sqrt{p_1q_1q_2})$) for some integers or semi-integers $a$, $b$, $x$ and $y$. We have: 
\begin{enumerate}[$1)$]
	\item If $\left(\frac{p}{p_1}\right)=\left(\frac{p}{q_1}\right)=\left(\frac{p}{q_2}\right)=\left(\frac{p_1}{q_1}\right)=-\left(\frac{p_1}{q_2}\right)=-1,$ then \begin{enumerate}[$a)$]
		\item $2pq_1(a\pm1)$ and $2q_2(x\pm1)$  are  squares in $\QQ$,
		
		\item The unit group of $\QQ(\sqrt{p},\sqrt{p_1q_1q_2})$ is given by $\langle{-1,\vep_{p},\vep_{p_1q_1q_2},\vep_{pp_1q_1q_2}\rangle}$.
	\end{enumerate}
	\item If $\left(\frac{p}{p_1}\right)=-\left(\frac{p}{q_1}\right)=\left(\frac{p}{q_2}\right)=\left(\frac{p_1}{q_1}\right)=-\left(\frac{p_1}{q_2}\right)=\left(\frac{q_1}{q_2}\right)=-1,$ then 
	\begin{enumerate}[$a)$]
		\item $2pq_1(a+1)$ and $2q_2(x+1)$ are squares in $\QQ$,
		\item The unit group of $\QQ(\sqrt{p},\sqrt{p_1q_1q_2})$ is given by $\langle{-1,\vep_{p},\vep_{p_1q_1q_2},\vep_{pp_1q_1q_2}\rangle}$.
	\end{enumerate}
\end{enumerate}
\end{lemma}
\begin{proof}
As \(N(\varepsilon_{pp_1q_1q_2}) = N(\varepsilon_{p_1q_1q_2}) = 1\) and \(N(\varepsilon_p) = -1\)  we need to determine the   the forms of:
$$\sqrt{\varepsilon_{p_1q_1q_2}}, \;\;\sqrt{\varepsilon_{pp_1q_1q_2}}  \text{ and }  \sqrt{\varepsilon_{p_1q_1q_2}\varepsilon_{pp_1q_1q_2}}.$$
Here we give the proof of the first item in details, and the prove of the second item is analogous. Let us assume that 
$\left(\frac{p}{p_1}\right)=\left(\frac{p}{q_1}\right)=\left(\frac{p}{q_2}\right)=\left(\frac{p_1}{q_1}\right)=-\left(\frac{p_1}{q_2}\right)=-1.$ So we have three main cases to investigate:

	\noindent\ding{229} Assume that $x,y, a$ and $b$ are integers.  Then $x^2-p_1q_1q_2y^2=a^2-pp_1q_1q_2b^2=1$. Then by  taking into account the unique factorization in $\ZZ$ and using Lemma \ref{ZekTAo} we have one of the following systems: 
 	$$(1):\ \left\{ \begin{array}{ll}
		x\pm 1=2p_1y_1^2\\
		x\mp 1=2q_1q_2 y_2^2,
	\end{array}\right. \quad
	\text{or}\quad
	(2):\ \left\{ \begin{array}{ll}
		x\pm 1=2q_1y_1^2\\
		x\mp 1=2p_1q_2 y_2^2,
	\end{array}\right.\quad
	\text{or}\quad
	(3):\ \left\{ \begin{array}{ll}
		x\pm 1=2q_2y_1^2\\
		x\mp 1=2p_1q_1 y_2^2,
	\end{array}\right.\quad
	$$	
	where $y_1$ and $y_2$ are two integers such that  $y=2y_1 y_2$.
	
	\begin{enumerate}[$\star$]
		\item The first system can not occur, since it implies that :
		$$\left(\frac{2}{p_1}\right)=\left(\frac{\mp 2}{p_1}\right)=\left(\frac{x\pm1\mp 2}{p_1}\right)=\left(\frac{x\mp1}{p_1}\right)=\left(\frac{2q_1q_2}{p_1}\right)=-\left(\frac{2}{p_1}\right).$$ 
		\item Similarly the second system implies that: 
		$$\left(\frac{2}{p_1}\right)=\left(\frac{\pm 2}{p_1}\right)=\left(\frac{x\mp1\pm 2}{p_1}\right)=\left(\frac{x\pm1}{p_1}\right)=\left(\frac{2q_1}{p_1}\right)=-\left(\frac{2}{p_1}\right).$$ 
		
	\end{enumerate}
	From the third system we get $2q_2(x+1)$ or $2q_2(x-1)$ is a square in $\NN$ according to whether $\left(\frac{q_1}{q_2}\right)=-1$ or not. Then $\sqrt{\vep_{p_1q_1q_2}}=y_1\sqrt{q_2}+y_2\sqrt{p_1q_1}$. Similarly we have  the following possible systems :
	
	$$(1):\ \left\{ \begin{array}{ll}
		a\pm 1=2pp_1b_1^2\\
		a\mp 1=2q_1q_2 b_2^2,
	\end{array}\right. \quad
	\text{or}\quad
	(2):\ \left\{ \begin{array}{ll}
		a\pm 1=2pq_1b_1^2\\
		a\mp 1=2p_1q_2 b_2^2,
	\end{array}\right.\quad
	\text{or}\quad
	(3):\ \left\{ \begin{array}{ll}
		a\pm 1=2pq_2b_1^2\\
		a\mp 1=2p_1q_1 b_2^2,
	\end{array}\right.\quad
	$$	
	where $b_1$ and $b_2$ are two integers such that  $b=2b_1 b_2$.
	
	\begin{enumerate}[$\star$]
		\item The first system implies that: 
		$$\left(\frac{2}{q_2}\right)=\left(\frac{a+1 }{q_2}\right)=\left(\frac{a-1+ 2}{q_2}\right)=\left(\frac{2pp_1}{q_2}\right)=-\left(\frac{2}{q_2}\right), \;\;\text{or}$$ 
		
		$$\left(\frac{-2}{q_1}\right)=\left(\frac{a-1 }{q_1}\right)=\left(\frac{a+1- 2}{q_1}\right)=\left(\frac{2pp_1}{q_1}\right)=\left(\frac{2}{q_1}\right)$$
		which is a contradiction.
		\item By the same way we prove that the third system can bot occur. Finally it remains the second system and from which we get $2pq_1(x+1)$ or $2pq_1(x-1)$ is square according to whether $\left(\frac{q_1}{q_2}\right)=-1$ or not. Therefore  $\sqrt{\vep_{pp_1q_1q_2}}=b_1\sqrt{pq_1}+b_2\sqrt{p_1q_2}$. Then by using Wada's algorithm we conclude that  $q(\QQ(\sqrt{p},\sqrt{p_1q_1q_2}))=1$.
		
	\end{enumerate}
	\noindent\ding{229} Now we assume that  $x$ and $y$ are integers, $a$ and $b$ are semi integers such that  $a=\alpha/2$ and $b=\beta/2$ for some integers $\alpha$  and $\beta$. According to the previous case we have $\sqrt{\vep_{p_1q_1q_2}}=y_1\sqrt{q_2}+y_2\sqrt{p_1q_1}$. For the second square we proceed as before. Then from $\alpha^2-4=pp_1q_1q_2\beta^2$ and according to Lemma \ref{ZekTAo} we have the following possible systems: 
		{\footnotesize	\begin{center}
			\renewcommand{\arraystretch}{1.6}
			\begin{table}[H]
				\begin{tabular}{cccc}
					 $(1):\ \left\{ \begin{array}{ll}
						\alpha\pm 2=p\beta_1^2\\
						\alpha\mp 2=p_1q_1q_2 \beta_2^2,
					\end{array}\right. $
					& $(2):\ \left\{ \begin{array}{ll}
						\alpha\pm 2=p_1\beta_1^2\\
						\alpha\mp 2=pq_1q_2 \beta_2^2,
					\end{array}\right. $
					& $(3):\ \left\{ \begin{array}{ll}
						\alpha\pm 2=q_1\beta_1^2\\
						\alpha\mp 2=pp_1q_2 \beta_2^2,
					\end{array}\right. $ & $(4):\ \left\{ \begin{array}{ll}
					\alpha\pm 2=q_2\beta_1^2\\
					\alpha\mp 2=pp_1q_1 \beta_2^2,
					\end{array}\right. $\\

					$(5):\ \left\{ \begin{array}{ll}
						\alpha\pm 2=pp_1\beta_1^2\\
						\alpha\mp 2=q_1q_2 \beta_2^2,
					\end{array}\right. $& $(6):\ \left\{ \begin{array}{ll}
					\alpha\pm 2=pq_1\beta_1^2\\
					\alpha\mp 2=p_1q_2 \beta_2^2,
					\end{array}\right.$& $(7):\ \left\{ \begin{array}{ll}
					\alpha\pm 2=pq_2\beta_1^2\\
					\alpha\mp 2=p_1q_1\beta_2^2.
					\end{array}\right.$&

				\end{tabular}
				
			\end{table}
	\end{center}}
	where $\beta_1$ and $\beta_2$ are integers such that $\beta=\beta_1\beta_2$.
	\begin{enumerate}[$\star$]
		\item The first system can not occur since it implies:

		$$1=\left(\frac{\pm4}{p_1}\right)=\left(\frac{\alpha\mp2\pm4}{p_1}\right)=\left(\frac{\alpha\pm2}{p_1}\right)=\left(\frac{p}{p_1}\right)=-1.$$
		\item Similarly for the second system:
		$$1=\left(\frac{\pm4}{p}\right)=\left(\frac{\alpha\mp2\pm4}{p}\right)=\left(\frac{\alpha\pm2}{p}\right)=\left(\frac{p_1}{p}\right)=-1.$$
		
		\item Similarly for the third system:
		$$1=\left(\frac{\pm4}{p}\right)=\left(\frac{\alpha\mp2\pm4}{p}\right)=\left(\frac{\alpha\pm2}{p}\right)=\left(\frac{q_1}{p}\right)=-1.$$
		\item Similarly for the fourth system:
		$$1=\left(\frac{\pm4}{p}\right)=\left(\frac{\alpha\mp2\pm4}{p}\right)=\left(\frac{\alpha\pm2}{p}\right)=\left(\frac{q_2}{p}\right)=-1.$$

		\item Similarly for the fifth system:
		$$1=\left(\frac{4}{q_2}\right)=\left(\frac{\alpha-2+4}{q_2}\right)=\left(\frac{\alpha+2}{q_2}\right)=\left(\frac{pp_1}{q_2}\right)=-1,\;\text{or}$$
		$$-1=\left(\frac{-4}{q_1}\right)=\left(\frac{\alpha+2-4}{q_1}\right)=\left(\frac{\alpha-2}{q_1}\right)=\left(\frac{pp_1}{q_1}\right)=1.$$

		\item The same thing for the seventh system:
		
		$$1=\left(\frac{\pm4}{p_1}\right)=\left(\frac{\alpha\mp2\pm4}{p_1}\right)=\left(\frac{\alpha\pm2}{p_1}\right)=\left(\frac{pq_2}{p_1}\right)=-1.$$
		
	\end{enumerate}
	Then it remains the sixth system which implies that $pq_1(x+2)$ or $pq_1(x-2)$ is a square in $\NN$ according to whether   $\left(\frac{q_1}{q_2}\right)=-1$ or not. Therefore $\sqrt{\vep_{pp_1q_1q_2}}=\frac{1}{2}(\beta_1\sqrt{pq_1}+\beta_2\sqrt{p_1q_2})$ and so  $q(\QQ(\sqrt{p},\sqrt{p_1q_1q_2}))=1$.

	\noindent\ding{229} Now if that  $x$, $y$, $a$ and $b$ are semi-integers. Then $\sqrt{\vep_{pp_1q_1q_2}}=\frac{1}{4}(\beta_1\sqrt{pq_1}+\beta_2\sqrt{p_1q_2})$. Moreover we have $z^2-4=t^2p_1q_1q_2$, where $z$ and $t$ are integers such that $x=z/2$ and $y=t/2$, then Lemma \ref{ZekTAo} implies the following possible systems:

	$$(1):\ \left\{ \begin{array}{ll}
		z\pm 2=p_1t_1^2\\
		z\mp 2=q_1q_2 t_2^2,
	\end{array}\right. \quad
	\text{or}\quad
	(2):\ \left\{ \begin{array}{ll}
		z\pm 2=q_1t_1^2\\
		z\mp 2=p_1q_2 t_2^2,
	\end{array}\right.\quad
	\text{or}\quad
	(3):\ \left\{ \begin{array}{ll}
		z\pm 2=q_2t_1^2\\
		z\mp 2=p_1q_1 t_2^2,
	\end{array}\right.\quad
	$$	
	where $t_1$ and $t_2$ are two integers such that  $t=t_1 t_2$. Using the same argument as in the previous two cases we check that $q_2(x+2)$ or $q_2(x-2)$ is square in $\NN$ according to whether $\left(\frac{q_1}{q_2}\right)=-1$ or not. Thus $\sqrt{\vep_{p_1q_1q_2}}=\frac{1}{2}(t_1\sqrt{q_2}+t_2\sqrt{p_1q_1})$ and $q(\QQ(\sqrt{p},\sqrt{p_1q_1q_2}))=1$.

	\noindent\ding{229} Finally if   $x$ and $y$ are semi integers, a and b are integers. Then from the previous sub-cases we have $\sqrt{\vep_{p_1q_1q_2}}=\frac{1}{2}(t_1\sqrt{q_2}+t_2\sqrt{p_1q_1})$  and  $\sqrt{\vep_{pp_1q_1q_2}}=b_1\sqrt{pq_1}+b_2\sqrt{p_1q_2}$. Then  $q(\QQ(\sqrt{p},\sqrt{p_1q_1q_2}))=1$.
 
This completes the proof.
\end{proof}

\begin{theorem}\label{classpd}
Let  $p\equiv1\pmod4$ be a prime number and   $d$ a  positive square-free integer co-prime with $p$. Put $K=\QQ(\sqrt{p},\sqrt{d})$.	Then $h_2(K)=2$ if and only if, $d$ is in one of the following forms: 
\begin{enumerate}[$A)$]
	\item $d=2$, with $p\equiv1\pmod8$, $\left(\frac{2}{p}\right)_4=(-1)^{\frac{p-1}{8}}$ and $h_2(2p)=4.$
	\item $d=p_1,$ with $p_1\equiv1\pmod4$ a prime number such that $h_2(pp_1)=4.$
	
	\item $d=\eta q,$ where $\eta\in \{1,2\}$ and $q\equiv3\pmod4$ is a prime number such that $p\equiv1\pmod8$,  $\left(\frac{p}{q}\right)=1$ and  $h_2(\eta pq)=4$.
	
	\item $d=q_1q_2$, where $q_1\equiv q_2\equiv
	3\pmod4$ are distinct prime numbers  such that  $\left(\frac{p}{q_1}\right)=\left(\frac{p}{q_2}\right)=1$ and $h_2(pq_1q_2)=4.$
	\item $d=p_1q$, where $p_1\equiv-q\equiv1\pmod4$ are distinct prime numbers such that $p_1\equiv5\pmod8$ and one of the following conditions holds: 
	\begin{enumerate}[$1)$]
		\item $p\equiv1\pmod8$ and $\left(\frac{p}{p_1}\right)=\left(\frac{p}{q}\right)=-\left(\frac{p_1}{q}\right)=-1,$
		\item $p\equiv1\pmod8, q\equiv3\pmod8$ and $\left(\frac{p}{p_1}\right)=\left(\frac{p}{q}\right)=\left(\frac{p_1}{q}\right)=-1,$
		
		\item $p\equiv5\pmod8$ and $\left(\frac{p}{p_1}\right)=-\left(\frac{p}{q}\right)=\left(\frac{p_1}{q}\right)=-1.$ 
	\end{enumerate}
	\item $d=2p_1q$ with $p_1\equiv-q\equiv1\pmod4$ are distinct prime numbers satisfying one of the following conditions: \begin{enumerate}[$1)$]
		\item $p\equiv 1\pmod8, p_1\equiv5\pmod8, q\equiv3\pmod8$ and $\left(\frac{p}{p_1}\right)=-\left(\frac{p_1}{q}\right)=\left(\frac{p}{q}\right)=-1,$ 
		\item $p\equiv p_1\equiv5\pmod8$ and $\left(\frac{p}{p_1}\right)=\left(\frac{p}{q}\right)=-\left(\frac{p_1}{q}\right)=-1,$
		
		\item $p\equiv5\pmod8, p_1\equiv1\pmod8$ and $\left(\frac{p}{p_1}\right)=\left(\frac{p}{q}\right)=\left(\frac{p_1}{q}\right)=-1.$
		
	\end{enumerate}
	\item $d=2q_1q_2$ with $q_1\equiv q_2\equiv3\pmod4$ are distinct prime numbers such that satisfying one of the following conditions:
	\begin{enumerate}[$1)$]
		\item $p\equiv5\pmod8, q_1\equiv q_2\equiv3\pmod8$ and $\left(\frac{p}{q_2}\right)=-\left(\frac{p}{q_1}\right)=1,$
		\item $p\equiv5\pmod8, q_1\equiv 7\pmod8, q_2\equiv3\pmod8$ and $\left(\frac{p}{q_2}\right)=-\left(\frac{p}{q_1}\right)=-\left(\frac{q_2}{q_1}\right)=1.$
	\end{enumerate}
	\item $d=p_1q_1q_2$ with $p_1\equiv -q_1\equiv-q_2\equiv1\pmod4$ are distinct prime numbers satisfying one of the following conditions:
	\begin{enumerate}[$1)$]
		\item $\left(\frac{p}{p_1}\right)=\left(\frac{p}{q_1}\right)=\left(\frac{p}{q_2}\right)=\left(\frac{p_1}{q_1}\right)=-\left(\frac{p_1}{q_2}\right)=-1,$
		\item $\left(\frac{p}{p_1}\right)=-\left(\frac{p}{q_1}\right)=\left(\frac{p}{q_2}\right)=\left(\frac{p_1}{q_1}\right)=-\left(\frac{p_1}{q_2}\right)=\left(\frac{q_1}{q_2}\right)=-1.$
	\end{enumerate}
	\item $d=2p_1$ with  $p_1\equiv1\pmod4$ prime number such satisfying one of the following conditions:
	
	\begin{enumerate}[$1)$]
		\item $p_1\equiv5\pmod8$, $h_2(2pp_1)=4$ and \textbf{none} of the two units $\{\vep_{p}\vep_{2p_1}\vep_{2pp_1}, \vep_{2pp_1}\}$ is a square in $K$.
		\item $p_1\equiv1\pmod8, h_2(2pp_1)=4$, $\left(\frac{2}{p_1}\right)_4\neq \left(\frac{p_1}{2}\right)_4$ and $N(\vep_{2pp_1})=-1$.
	\end{enumerate}   
	\item $d=p_1p_2$ with $p_i\equiv1\pmod4$ distinct prime numbers satisfying one of the following conditions:
	\begin{enumerate}[$1)$]
		\item $\left(\frac{p_1}{p_2}\right)=-1,$ $h_2(pp_1p_2)=4$ and \textbf{none} of the two units $\{\vep_{p}\vep_{p_1p_2}\vep_{pp_1p_2}, \vep_{pp_1p_2}\}$ is a square in $K$,
		\item $\left(\frac{p_1}{p_2}\right)=1, \left(\frac{p_1}{p_2}\right)_4\neq\left(\frac{p_2}{p_1}\right)_4$, $h_2(pp_1p_2)=4$ and \textbf{none} of the  units  $\{\vep_{p_1p_2}\vep_{pp_1p_2},\vep_{p}\vep_{p_1p_2}\vep_{pp_1p_2}, \vep_{p}\vep_{pp_1p_2}\}$  is a square in $K$.
	\end{enumerate}

\end{enumerate}
\end{theorem}

\begin{proof}
Let $d$ be a positive square-free integer co-prime with $p$. Notice that the extension  $K/k$, with $k=\QQ(pd)$, is an unramified quadratic extension. Assume that  $h_2(K)=2$. Then $h_2(pd)=4$ and by   Kuroda's class number formula, we deduce that   
$h_2(d)=1$ (in this case $q(K)=2$) or $h_2(d)=2$ (in this case $q(K)=1$). So we have two main cases to investigate.

\noindent\ding{229}	  Assume that $ h_2(d)=1$. By Lemma \ref{realQuadodd}, $d\in \{2,p_1,q,2q, q_1q_2\}$, where $p_1\equiv-q\equiv-q_1\equiv-q_2\equiv1\pmod4$ are distinct prime numbers. In particular, we have:
\begin{enumerate}[$\star$]
	\item Letting $d=2$, then as in the proof of Theorem \ref{class2d}, we deduce that $h_2(K)=2$ if and only if 
	$p\equiv1\pmod8$, $\left(\frac{2}{p}\right)_4=(-1)^{\frac{p-1}{8}}$ and $h_2(2p)=4$.   This justifies the form $A)$.
	\item Letting $d=p_1$, with $p_1\equiv1\pmod4$   a prime number.   Thus, by class field theory,  $h_2(K)=2$ exactly when $h_2(pp_1)=4$. This justifies the form $B)$.
	\item Letting  $d=\eta q$ with $q\equiv3\pmod4$ prime number, we deduce from the condition
	$h_2(\eta pq)=4$ that we have  $p\equiv1\pmod8$ and $\left(\frac{p}{q}\right)=1$ (cf. Lemma \ref{realQuadeven1New}).
	Thus, in this case,   by class field theory  $h_2(K)=2$ if and only if  $p\equiv1\pmod8$,  $\left(\frac{p}{q}\right)=1$ and $h_2(\eta pq)=4$.
	This justifies the form $C)$.
	
	\item $d=q_1q_2$ where $q_1\equiv q_2\equiv3\pmod4$ are distinct prime numbers. Notice that the condition $h_2(pq_1q_2)=4$ implies that  $\left(\frac{p}{q_1}\right)=\left(\frac{p}{q_2}\right)=1$ (cf. Lemma \ref{realQuadeven1New}). So by class field theory  we have   $h_2(K)=2$ if and only of $h_2(pq_1q_2)=4$. This justifies the form $D)$.
\end{enumerate}

\noindent\ding{229}	  Now assume that $ h_2(d)=2$. By Lemma  \ref{realQuadeven1New}, $d\in \{p_1q,2p_1q,2q_1q_2,p_1q_1q_2, 2p_1, p_1p_2\}$, where $p_1\equiv p_2\equiv-q_1\equiv q_2 \equiv 1\pmod4$ are prime numbers that satisfy extra condition as in each case   below.   
\begin{enumerate}[$\star$]
	\item Let $d=p_1q$, where $p_1\equiv-q\equiv1\pmod4$ are  prime numbers such that  $p_1\equiv5\pmod 8$ or $\left(\frac{p_1}{q}\right)=-1$. Thus, 
	$h_2(pp_1q)=4$, then (by letting $p=r$ and $p_1=s$) the primes $p, p_1$ and $q$ satisfy the conditions $c_{15}-c_{20}$ in Table \ref{qrsabelien} (cf. \cite[page 177-178]{BS95}). 
	
	Conversely using Tables \ref{tab7}, \ref{tab8} and  Wada's algorithm (cf. Page \pageref{wadaalgo}) we check that $E_K=\langle{-1,\vep_{p},\vep_{p_1q} , \vep_{pp_1q}\rangle} $ for all such cases of Table \ref{qrsabelien} except the case $c_{17}$ which gives $q(K)=2$ (In fact, in this case we have $E_K=\langle{-1,\vep_{p},\vep_{p_1q} , \sqrt{\vep_{p_1q}\vep_{pp_1q}}\rangle} $). Hence, $h_2(K)=2$ if and only if we have one of the conditions listed from $E)$-$1)$ to $E)$-$3)$.

	\item Let $d=2p_1q$, where $p_1\equiv-q\equiv1\pmod4$ are prime numbers such that $p_1\equiv5\pmod 8$ or $\left(\frac{p_1}{q}\right)=-1$. Notice  that if $h_2(2pp_1q)=4$, then the primes $p, p_1$ and $q$ satisfy one of the conditions $c_{19}-c_{22}$, $c_{25}- c_{26}$ in Table \ref{abelien2qrs} (cf. \cite[Page 176]{BS95}).
	Conversely by invoking Table \ref{tab8} in all previous cases except case $c_{19}$ we check that $E_K=\langle{-1,\vep_{p},\vep_{2p_1q} , \vep_{2pp_1q}\rangle} $, i.e.,   $q(K)=1$, which gives $h_2(K)=2$. For the case $c_{19}$ we get $E_K=\langle{-1,\vep_{p},\vep_{2p_1q} , \sqrt{\vep_{2pp_1q}\vep_{2p_1q}}\rangle}$, then  $h_2(K)=4$ and so this case is rejected. This justifies the forms listed from $F)$-$1)$ to $F)$-$3)$.
	\item Let $d=2q_1q_2$, with $q_1\equiv q_2\equiv3\pmod 4$ are distinct prime numbers such that $q_1$ or $q_2\equiv3\pmod8$.
	We note that the primes satisfying $h_2(2pq_1q_2)=4$ are exactly those satisfying one of the conditions in Table \ref{abelien2pq1q2}. Conversely, by making use of Tables \ref{tab7} and \ref{tab8} and Wada's algorithm we check that in all these  cases we have $q(K)=1$, which means that $h_2(K)=2$. This justifies the forms $G)$-$1)$ and $G)$-$2)$.
	\item Let $d=p_1q_1q_2$, where $p_1\equiv-q_1\equiv-q_2\equiv1\pmod 4$ are  distinct prime numbers such that  $\left(\frac{p_1}{q_1}\right)=-1$ or $\left(\frac{p_1}{q_2}\right)=-1$. Notice that $K/\mathbb{Q}(\sqrt{pp_1q_1q_2})$ is an unramified quadratic extension. So $h_2(K)=2$ implies that  $h_2(pp_1q_1q_2)=4$.
	Moreover,  if $h_2(pp_1q_1q_2)=4$, then according to \cite[page 175, cases $c_{12}-c_{14}$]{BS95}, the primes $p$, $p_1$, $q_1$ and $q_2$ satisfy one of the following conditions: 
	\begin{enumerate}[$\C1)$]
		\item $\left(\frac{p}{p_1}\right)=\left(\frac{p}{q_1}\right)=\left(\frac{p}{q_2}\right)=\left(\frac{p_1}{q_1}\right)=-\left(\frac{p_1}{q_2}\right)=-1,$
		\item $\left(\frac{p}{p_1}\right)=-\left(\frac{p}{q_1}\right)=\left(\frac{p}{q_2}\right)=\left(\frac{p_1}{q_1}\right)=-\left(\frac{p_1}{q_2}\right)=\left(\frac{q_1}{q_2}\right)=-1.$
	\end{enumerate}
	
	Conversely, we check that if one of the   conditions  $\C1)$ and $\C2)$ holds, then making use of  Lemma \ref{unitspp1q1q2}, we check that $q(K)=1$. Since $2h_2(p_1q_1q_2)=h_2(pp_1q_1q_2)=4h_2(p)=4$, it follows that $h_2(K)=2$. This justifies the forms $H)$-$1)$ and $H)$-$2)$.
	\item Let $d=2p_1,$ where $p_1\equiv5\pmod8$ is a prime number. According to \cite[Theorem 2]{kuvcera1995parity} we have $N(\vep_{2p_1})=-1$ and by \cite[Theorem 2.1]{AzZektaous} we check that  $q(K)=2$ or $1$ according to whether one of the units $\{\vep_{p}\vep_{2p_1}\vep_{2pp_1}, \vep_{2pp_1}\}$  is a square in $K$ or not. Furthermore $h_2(K)=2$ exactly when $p_1\equiv5\pmod8, h_2(2pp_1)=4$ and none of the   units $\{\vep_{p}\vep_{2p_1}\vep_{2pp_1}, \vep_{2pp_1}\}$ is a square in $K$. This justifies the the first item in form $I)$. 
	
	\item Let $d=2p_1,$ where $p_1\equiv1\pmod8$   is a prime number such that $\left(\frac{2}{p_1}\right)_4\neq \left(\frac{p_1}{2}\right)_4$. According to \cite[Theorem 2]{kuvcera1995parity} we have $N(\vep_{2p_1})=1$ and by \cite[Theorem 2.1]{AzZektaous} we check that  $q(K)=2$ or $1$. according to whether $N(\vep_{2pp_1})=1$ or not. Furthermore $h_2(K)=2$ exactly when $p_1\equiv1\pmod8, h_2(2pp_1)=4$, $\left(\frac{2}{p_1}\right)_4\neq \left(\frac{p_1}{2}\right)_4$ and $N(\vep_{2pp_1})=-1$.  This justifies the the second item in form $I)$.

	\item Let $d=p_1p_2,$ where $p_1\equiv p_2\equiv1\pmod4$ distinct prime number such that $\left(\frac{p_1}{p_2}\right)=-1$ (cf. Lemma \ref{realQuadeven1New}). Then by \cite[Theorem 1]{kuvcera1995parity} we have $N(\vep_{p_1p_2})=-1$, thus using Wada's method we conclude that $q(K)=2$ or $1$ according to whether one of the units $\{\vep_{p}\vep_{p_1p_2}\vep_{pp_1p_2}, \vep_{pp_1p_2}\}$  is a square in $K$ or not. Hence, $h_2(K)=2$ exactly when $\left(\frac{p_1}{p_2}\right)=-1, h_2(pp_1p_2)=4$ and none of the above units is a square in $K$. This justifies the first item in forms $J)$. 
	\item Let $d=p_1p_2,$ where $p_1\equiv p_2\equiv1\pmod4$ distinct prime number such that  $\left(\frac{p_1}{p_2}\right)=1$ and $\left(\frac{p_1}{p_2}\right)_4\neq\left(\frac{p_2}{p_1}\right)_4$. Then by \cite[Theorem 1]{kuvcera1995parity} we have $N(\vep_{p_1p_2})=1$, moreover using Wada's method we conclude that $q(K)=1$ if and only if none of the units $\{\vep_{p_1p_2}\vep_{pp_1p_2},\vep_{p}\vep_{p_1p_2}\vep_{pp_1p_2}, \vep_{p}\vep_{pp_1p_2}\}$ is a square in $K$. Hence, $h_2(K)=2$ exactly when $\left(\frac{p_1}{p_2}\right)=1, \left(\frac{p_1}{p_2}\right)_4\neq\left(\frac{p_2}{p_1}\right)_4$, $h_2(pp_1p_2)=4$ and none of the above units is a square in $K$.  This justifies the second item in forms $J)$.  
\end{enumerate}This completes the proof of our result.
\end{proof}



\section{\bf {Biquadratic fields that are not QO-fields with $2$-class numbers equal  $2$}}
In this section, we investigate biquadratic fields $K$ that are not QO-fields.
For  an integer $d$, let  $\mathrm{sqf} (d)$ denote  its square-free part. Taking into account the discussions presented in the introduction (i.e., Section \ref{Sec1}), the following theorem provides a complete classification  of the fields $K$ that are not QO-fields with $2$-class numbers equal  $2$.

\begin{theorem}\label{classnoadqua}
Let $d_1$ and $d_2$ be two distinct positive  square-free integers such that  $h_2(d_1)= h_2(d_2)=h_2(d_3)=2$, where $d_3=\mathrm{sqf}(d_1d_2)$. Put $K=\QQ(\sqrt{d_1},\sqrt{d_2})$. Then $h_2(K)=2$ if and only if one of the following conditions holds:
\begin{enumerate}[$1)$]
	\item $K=\QQ(\sqrt{pq_1q_2},\sqrt{2q_1q_2})$, with $p\equiv-q_1\equiv-q_2\equiv1\pmod4$ are distinct prime numbers such that $p\equiv5\pmod8$, $\left(\frac{p}{q_1}\right)=-\left(\frac{p}{q_2}\right)=-1$ and one of the following conditions holds:
	\begin{enumerate}[$i)$]
		\item $q_1\equiv7\pmod8$ and $q_2\equiv3\pmod8$,
		\item $q_1\equiv q_2\equiv3 \pmod8$.
	\end{enumerate}
	
	\item $K=\QQ(\sqrt{p_1q},\sqrt{p_2q})$ or $\QQ(\sqrt{2p_1q},\sqrt{2p_2q})$, with $q\equiv-p_1\equiv-p_2\equiv3\pmod4$ are distinct prime numbers such that $\left(\frac{p_1}{p_2}\right)=-1$ and one of the following conditions holds:\begin{enumerate}[$i)$]
		\item $p_1\equiv1\pmod8, p_2\equiv5\pmod8$ and  $\left(\frac{p_1}{q}\right)=-1,$
		\item $p_1\equiv p_2\equiv5\pmod8$ and $\left(\frac{p_1}{q}\right)\cdot\left(\frac{p_2}{q}\right)=-1.$
	\end{enumerate} 
	\item $K=\QQ(\sqrt{p_1q_1q_2},\sqrt{p_2q_1q_2})$ with $q_1\equiv q_2\equiv-p_1\equiv-p_2\equiv3\pmod4$ are distinct prime numbers such that after a suitable permutation of indices we have:
	$$\left(\frac{p_1}{q_1}\right)=\left(\frac{p_1}{q_2}\right)=\left(\frac{p_2}{q_1}\right)\cdot\left(\frac{p_2}{q_2}\right)=\left(\frac{p_1}{p_2}\right)=-1.$$

	\item $K=\QQ(\sqrt{2p_1},\sqrt{2p_2}),$ with $p_1\equiv p_2\equiv1\pmod4$ are distinct prime numbers satisfying one of the following conditions:
	\begin{enumerate}[$i)$]
		\item $p_1\equiv p_2\equiv5\pmod8$, 	$\left(\frac{p_1}{p_2}\right)=-1$ and  $\vep_{2p_1}\vep_{2p_2}\vep_{p_1p_2}$ is not a square in $K$,
		\item $p_1\equiv p_2\equiv5\pmod8 $,  $\left(\frac{p_1}{p_2}\right)=1$ and  $\left(\frac{p_1}{p_2}\right)_4\neq \left(\frac{p_2}{p_1}\right)_4$,
		\item $p_1\equiv5\pmod8$, $p_2\equiv1\pmod8 $,   $\left(\frac{2}{p_2}\right)_4\neq \left(\frac{p_2}{2}\right)_4$\text{ and } $\left(\frac{p_1}{p_2}\right)=-1$.
	\end{enumerate}

	\item $K=\QQ(\sqrt{p_1p_2},\sqrt{p_2p_3})$ with $p_1\equiv p_2\equiv p_3\equiv1\pmod4$ be distinct prime numbers satisfying one of the following conditions:
	\begin{enumerate}[$i)$]
		\item $\left(\frac{p_1}{p_2}\right)=\left(\frac{p_1}{p_3}\right)=\left(\frac{p_2}{p_3}\right)=-1$ and $\vep_{p_1p_2}\vep_{p_2p_3}\vep_{p_1p_3}$ is not a square in $K$,
		\item $\left(\frac{p_1}{p_2}\right)=\left(\frac{p_2}{p_3}\right)=-1$, $\left(\frac{p_1}{p_3}\right)=1$\text{ and } $\left(\frac{p_1}{p_3}\right)_4\neq \left(\frac{p_3}{p_1}\right)_4$.
	\end{enumerate}

\end{enumerate}
\end{theorem}
\begin{proof}
Let $K=\mathbb{Q}(\sqrt{d_1},\sqrt{d_2})$, with $d_1$ and $d_2$ distinct positive square-free integers such that 
$h_2(d_1)=h_2(d_2)=h_2(d_3)=2$, where $d_3:=\mathrm{sqf}(d_1d_2)$. 
Suppose also that $h_2(K)=2$. 
It follows that the integers $d_1$ and $d_2$ must appear among the  possible forms listed in Table \ref{biquaNotQO} presented in  the introduction. 
For the converse, we treat the cases independently as in what follows.
\begin{enumerate}[$1)$]
	\item Let $d_1=pq_1$ and $d_2=2q_1q_2$ where $p\equiv -q_1 \equiv -q_2 \equiv1\pmod4$ are distinct prime numbers such that   [$\left(\frac{p}{q_1}\right)=\left(\frac{p}{q_2}\right)=-1$ or $p\equiv5\pmod8$] and [$q_1\equiv 3\pmod 8$ or $q_2\equiv 3\pmod 8$]. Then  according to Table \ref{tab7} and Wada's method we have $q(K)\geq 2$. In fact we have :
	
	{\footnotesize	\begin{table}[H]
			\centering
			\small 
			\setlength{\tabcolsep}{6pt} 
			\renewcommand{\arraystretch}{1.6} 
			
			\begin{tabular}{p{4cm}|p{10cm}}
				\toprule
				\textbf{${\varepsilon}$} & \textbf{The form of $\sqrt{\varepsilon}$ } \\
				\midrule
				${\vep_{pq_1}}$ & $[\sqrt{p},\sqrt{q_1}]$, $[\sqrt{2p},\sqrt{2q_1}]$ or $[\sqrt{2},\sqrt{2pq_1}]$ \\
				\hline
				${\vep_{2q_1q_2}}$ & $[\sqrt{2},\sqrt{q_1q_2}]$, $[\sqrt{2q_1},\sqrt{q_2}]$ or $[\sqrt{2q_2},\sqrt{q_1}]$ \\
				\hline
				${\vep_{2pq_2}}$ & $[\sqrt{2p},\sqrt{q_2}]$, $[\sqrt{p},\sqrt{2q_2}]$ or $[\sqrt{2},\sqrt{pq_2}]$ \\
				\hline
				\hline
				Some square in $E_K$ & $\sqrt{\vep_{pq_1}\vep_{2q_1q_2}\vep_{2pq_2}}$, $\sqrt{\vep_{pq_1}\vep_{2pq_2}}$, $\sqrt{\vep_{2q_1q_2}\vep_{2pq_2}}$ or $\sqrt{\vep_{pq_1}\vep_{2q_1q_2}}$ \\
				\toprule
			\end{tabular}
			
			\caption{Square roots of units in $K$}
			\label{formsqrt1}
	\end{table}}
	Where the notation $[m,n]$ is explained in the introduction (cf. Notation \ref{squarenot}). Furthermore such field must me rejected.
	\item Let $d_1=pq_1q_2$ and $d_2=2q_1q_2$ where $p\equiv -q_1 \equiv -q_2 \equiv1\pmod4$ are distinct prime numbers such that   $p\equiv5\pmod8$, [$\left(\frac{p}{q_1}\right)=-1$ or $\left(\frac{p}{q_2}\right)=-1$] and [$q_1\equiv 3\pmod 8$ or $q_2\equiv 3\pmod 8$]. Using Tables \ref{tab7} and \ref{tab8} together and Wada's method we check that  $q(K)=1$ or $2$. More precisely, we have Table \ref{formsqrt2}. 
	
	{\footnotesize	\begin{center}
			\begin{table}[H]
				\begin{tabular}{p{3cm}|p{3cm}|c}
					\toprule
					Form of  $\sqrt{\vep_{pq_1q_2}}$ &  Form of $\sqrt{\vep_{2q_1q_2}}$ & $E_K$ \\
					\toprule
					$[\sqrt{q_2},\sqrt{pq_1}]$ 	&  $[\sqrt{q_2},\sqrt{2q_1}]$ & $\langle{-1,\vep_{2q_1q_2},\vep_{2p},\sqrt{\vep_{pq_1q_2}\vep_{2q_1q_2}}\rangle}$  \\
					\hline
					$[\sqrt{q_1},\sqrt{pq_2}]$ 	&  $[\sqrt{q_1},\sqrt{2q_2}]$ & $\langle{-1,\vep_{2q_1q_2},\vep_{2p},\sqrt{\vep_{pq_1q_2}\vep_{2q_1q_2}}\rangle}$  \\
					\hline
					$[\sqrt{p},\sqrt{q_1q_2}]$ 	&  $[\sqrt{2},\sqrt{q_1q_2}]$ & $\langle{-1,\vep_{2q_1q_2},\vep_{2p},\sqrt{\vep_{pq_1q_2}\vep_{2q_1q_2}}\rangle}$  \\
					\hline
					$Else$ 	&  $Else$ & $\langle{-1,\vep_{2q_1q_2},\vep_{2p},\vep_{pq_1q_2}\rangle}$  \\
					\bottomrule
				\end{tabular}
				\caption{The Forms of Square Roots in $E_K$}\label{formsqrt2}
			\end{table}
	\end{center}}
	It is worth noting that, by virtue of \cite[Theorem 2]{kuvcera1995parity}, we have $N_{\mathbb{Q}(\sqrt{2p})/\mathbb{Q}}(\varepsilon_{2p})=-1$. Moreover, as is evident from the tables cited above, the first three cases appearing in Table \ref{formsqrt2} are realised if and only if 
	\begin{enumerate}[$\mathrm{C}1)$]
		\item $p\equiv5\pmod8, q_1\equiv 3\pmod 8$ and  $q_2\equiv 7\pmod 8$ and  $\left(\frac{p} {q_1}\right)=-\left(\frac{p}{q_2}\right)=-1$,
		\item $p\equiv5\pmod8, q_1\equiv 7\pmod 8$ and  $q_2\equiv 3\pmod 8$ and  $\left(\frac{p} {q_1}\right)=-\left(\frac{p}{q_2}\right)=1$,
		\item $p\equiv5\pmod8, q_1\equiv q_2\equiv 3\pmod 8$  and  $\left(\frac{p} {q_1}\right)=\left(\frac{p}{q_2}\right)=-1$.
	\end{enumerate}
	Which gives the exact cases when $q(K)=2$ using Wada's method. Consequently, by negating these conditions and then intersecting the result with the hypotheses of this sub-case, we derive the precise conditions under which $q(K)=1$, thereby justifying case 1) of the theorem.
	
	\item Let $d_1=pq_1q_2$ and $d_2=2q_1q_2$ where $p\equiv1\pmod8$, $\left(\frac{2}{p}\right)_4\neq \left(\frac{p}{2}\right)_4$,  $[\left(\frac{p}{q_1}\right)=-1$ or $\left(\frac{p}{q_2}\right)=-1]$ \text{ and } $[q_1\equiv 3\pmod 8$ \text{ or }$q_2\equiv 3\pmod 8]$. According to  \cite[Theorem 2]{kuvcera1995parity}, we have $N_{\mathbb{Q}(\sqrt{2p})/\mathbb{Q}}(\varepsilon_{2p})=1$, then $\sqrt{\vep_{2p}}=y_1\sqrt{2}+y_2\sqrt{p}$, for some rational numbers $y_1$ and $y_2$, moreover using Tables \ref{tab7} and \ref{tab8} in the above conditions we  check that a least one of the following square roots $\sqrt{\vep_{2p}\vep_{2q_1q_2}}$, $\sqrt{\vep_{2p}\vep_{pq_1q_2}}$, $\sqrt{\vep_{2p}\vep_{2q_1q_2}\vep_{pq_1q_2}}$ or $\sqrt{\vep_{2q_1q_2}\vep_{pq_1q_2}}$ belongs to $E_K$, more precisely we have 
	
	{\footnotesize	\begin{center}
			\begin{table}[H]
				\begin{tabular}{p{3cm}|p{3cm}|c}
					\toprule
					Form of  $\sqrt{\vep_{pq_1q_2}}$ &  Form of $\sqrt{\vep_{2q_1q_2}}$ & Squares in $E_K$ \\
					\toprule
					$[\sqrt{q_2},\sqrt{pq_1}]$ 	&  $[\sqrt{q_2},\sqrt{2q_1}]$ & $\sqrt{\vep_{pq_1q_2}\vep_{2q_1q_2}}, \sqrt{\vep_{2p}\vep_{2q_1q_2}}$  \\
					\cline{1-2}
					$[\sqrt{q_1},\sqrt{pq_2}]$ 	&  $[\sqrt{q_1},\sqrt{2q_2}]$ & $\sqrt{\vep_{2p}\vep_{pq_1q_2}},\sqrt{\vep_{2p}\vep_{pq_1q_2}\vep_{2q_1q_2}} $  \\
					\cline{1-2}
					$[\sqrt{p},\sqrt{q_1q_2}]$ 	&  $[\sqrt{2},\sqrt{q_1q_2}]$ &   \\
					
					\bottomrule
				\end{tabular}
				\caption{The Forms of Square Roots in $E_K$}\label{formsqrt3}
			\end{table}
	\end{center}}

	Therefore $q(K)\geq 2$, and such case must be rejected. 
	
	\item[4)-5)]  Let $(d_1,d_2)=(p_1q,p_2q)$ or $(2p_1q,p_1p_2)$ where $p_1\equiv p_2 \equiv -q \equiv1\pmod4$ are distinct prime numbers such that  [$\left(\frac{p_1}{q}\right)=-1$ or $p_1\equiv 5\pmod 8$], [$\left(\frac{p_2}{q}\right)=-1$ or $p_2\equiv 5\pmod 8$] \text{ and } $\left(\frac{p_1}{p_2}\right)=-1$. Recalling \cite[Theorem 1]{kuvcera1995parity}, we have $N_{\mathbb{Q}(\sqrt{p_1p_2})/\mathbb{Q}}(\varepsilon_{p_1p_2})=-1$. With the aid of Table \ref{tab7} and Wada's method, we then conclude that $q(K)$ is either $2$ or $1$, the distinction resting on the fulfilment or otherwise of one of the conditions set forth below:
 	\begin{enumerate}[C1)]
		\item $p_1\equiv p_2\equiv1\pmod8$ and $\left(\frac{p_1}{q}\right)=\left(\frac{p_2}{q}\right)=-1$,
		\item $p_1\equiv p_2\equiv5\pmod8$ and $\left(\frac{p_1}{q}\right)=\left(\frac{p_2}{q}\right)$.
	\end{enumerate}
	An entirely analogous procedure to that of the previous paragraph furnishes the conditions that establish cases $2)$--$i)$ and $2)$--$ii)$.
	\item[6)]  Let $d_1=p_1q$ and $d_2=p_2q$ where [$\left(\frac{p_1}{q}\right)=-1$ or $p_1\equiv 5\pmod 8$], [$\left(\frac{p_2}{q}\right)=-1$ or $p_2\equiv 5\pmod 8$],  $\left(\frac{p_1}{p_2}\right)=1$\text{ and } $\left(\frac{p_1}{p_2}\right)_4\neq \left(\frac{p_2}{p_1}\right)_4$. According to  \cite[Theorem 1]{kuvcera1995parity}, we have $N_{\mathbb{Q}(\sqrt{p_1p_2})/\mathbb{Q}}(\varepsilon_{p_1p_2})=1$, then $\sqrt{\vep_{p_1p_2}}=y_1\sqrt{p_1}+y_2\sqrt{p_2},$ for some rational numbers $y_1$ and $y_2$. Combining this with  Table \ref{tab7} in this case, we check that $q(K)\geq2$, in fact at least one of the following square roots $\sqrt{\vep_{p_1p_2}\vep_{p_1q}}$, $\sqrt{\vep_{p_1p_2}\vep_{p_2q}}$, $\sqrt{\vep_{p_1q}\vep_{p_2q}}$ or $\sqrt{\vep_{p_1q}\vep_{p_2q}\vep_{p_1p_2}}$ is in $E_K$. This such case must be rejected.
	\item[7)] Let  $d_1=2p_1q$ and $d_2=2p_2q$ where   [$\left(\frac{p_1}{q}\right)=-1$ or $p_1\equiv 5\pmod 8$], [$\left(\frac{p_2}{q}\right)=-1$ or $p_2\equiv 5\pmod 8$],  $\left(\frac{p_1}{p_2}\right)=1$\text{ and } $\left(\frac{p_1}{p_2}\right)_4\neq \left(\frac{p_2}{p_1}\right)_4$. Using the same argument as in the previous case one can conclude that $q(K)=2$. In fact $E_K=\langle{-1,\vep_{2p_1q},\vep_{2p_2q},\eta \rangle}$ with $\eta=\sqrt{\vep_{p_1p_2}\vep_{2p_1q}}, \sqrt{\vep_{p_1p_2}\vep_{2p_2q}}, \sqrt{\vep_{2p_1q}\vep_{2p_2q}} \;\text{or}\; \sqrt{\vep_{p_1p_2}\vep_{2p_1q}\vep_{2p_2q}}$. Thus such case must be rejected.
	
	\item[8)] Let $d_1=p_1q_1q_2$ and $d_2=p_2q_1q_2$ where $p_1\equiv p_2\equiv -q_1 \equiv -q_2 \equiv1\pmod4$ are distinct prime numbers such that   [$\left(\frac{p_1}{q_1}\right)=-1$ or $\left(\frac{p_1}{q_2}\right)=-1$], [$\left(\frac{p_2}{q_1}\right)=-1$ or $\left(\frac{p_2}{q_2}\right)=-1$] \text{and} $\left(\frac{p_1}{p_2}\right)=-1$. According to  \cite[Theorem 1]{kuvcera1995parity} we have  $N_{\QQ(\sqrt{p_1p_2})/\QQ}(\vep_{p_1p_2})=-1$, it follows that the unit $\sqrt{\vep_{p_1q_1q_2}\vep_{p_2q_1q_2}}$ is the only square candidate belonging to $E_K$, because the other two units, $\sqrt{\vep_{p_1q_2q_2}}$ and $\sqrt{\vep_{p_2q_2q_2}}$, take possible forms that do not belong to $E_K$ (cf. Table \ref{tab8}). By using Table \ref{tab8} and Wada's method we get $q(K)=2$ exactly in one of the following cases:\begin{enumerate}[$\mathrm{C}1)$]
		\item $\left(\frac{p_1}{q_1}\right)=\left(\frac{p_1}{q_2}\right)=\left(\frac{p_2}{q_1}\right)=\left(\frac{p_2}{q_2}\right)=-1,$ 
		\item $\left(\frac{p_1}{q_2}\right)=-\left(\frac{p_1}{q_1}\right)=-1$ and $\left(\frac{p_2}{q_2}\right)=-\left(\frac{p_2}{q_1}\right)=-1.$
	\end{enumerate} 
	Therefore in this sub-case we get   $q(K)=1$ if and only if 
	
	$$\left(\frac{p_1}{q_1}\right)=\left(\frac{p_1}{q_2}\right)=\left(\frac{p_2}{q_1}\right)\cdot\left(\frac{p_2}{q_2}\right)=\left(\frac{p_1}{p_2}\right)=-1.$$
	Which justifies the third case.

	\item[9)] Let $d_1=p_1q_1q_2$ and $d_2=p_2q_1q_2$ where [$\left(\frac{p_1}{q_1}\right)=-1$ or $\left(\frac{p_1}{q_2}\right)=-1$], [$\left(\frac{p_2}{q_1}\right)=-1$ or $\left(\frac{p_2}{q_2}\right)=-1$],  $\left(\frac{p_1}{p_2}\right)=1$ \text{and} $\left(\frac{p_1}{p_2}\right)_4\neq \left(\frac{p_2}{p_1}\right)_4$. As $N_{\mathbb{Q}(\sqrt{p_1p_2})/\mathbb{Q}}(\varepsilon_{p_1p_2})=1$ (cf. \cite[Theorem 1]{kuvcera1995parity}), Then by Table \ref{tab8} we conclude that $q(K)=2$, more precisely we have $E_K=\langle{-1,\vep_{p_1q_1q_2},\vep_{p_2q_1q_2},\eta \rangle}$ with $\eta=\sqrt{\vep_{p_1p_2}\vep_{p_1q_1q_2}}, \sqrt{\vep_{p_1p_2}\vep_{p_2q_1q_2}}, \sqrt{\vep_{p_1q_1q_2}\vep_{p_2q_1q_2}} \;\text{or}\; \sqrt{\vep_{p_1p_2}\vep_{p_1q_1q_2}\vep_{p_2q_1q_2}}$. Thus such case must be rejected.

	\item[10)] If $d_1=2p_1$ and $d_2=2p_2$ for distinct primes $p_1\equiv p_2\equiv 5 \pmod 8$ with $\left(\frac{p_1}{p_2}\right)=-1$, then \cite[Theorems 1 and 2]{kuvcera1995parity} gives that all three fundamental units of the quadratic subfields of $K$ have norm $-1$. Wada's method then implies that $q(K)=2$ or $1$ according to whether  $\sqrt{\varepsilon_{2p_1}\varepsilon_{2p_2}\varepsilon_{p_1p_2}}\in K$ or not.

	\item[11)] If $d_1=2p_1$ and $d_2=2p_2$ for distinct primes  with $p_1\equiv p_2\equiv5\pmod8 $,  $\left(\frac{p_1}{p_2}\right)=1$ \text{ and } $\left(\frac{p_1}{p_2}\right)_4\neq \left(\frac{p_2}{p_1}\right)_4$. By \cite[Theorems 1 and 2]{kuvcera1995parity} we have $N(\vep_{2p_1})=N(\vep_{2p_1})=-N(\vep_{p_1p_2})=-1$, then $q(K)=1$, in fact the only square root candidate to be in $E_K$ is $\sqrt{\vep_{p_1p_2}}$, but it takes the form $[\sqrt{p_1}, \sqrt{p_2}]$ which is not in $K.$

	\item[12)] If $d_1=2p_1$ and $d_2=2p_2$ for distinct primes  with  $p_1\equiv5\pmod8$, $p_2\equiv1\pmod8 $,   $\left(\frac{2}{p_2}\right)_4\neq \left(\frac{p_2}{2}\right)_4$\text{ and } $\left(\frac{p_1}{p_2}\right)=-1$.  According to o \cite[Theorems 1 and 2]{kuvcera1995parity} we have $N(\vep_{2p_1})=N(\vep_{p_1p_2})=-N(\vep_{2p_2})=-1$, thus $q(K)=1$ as $\sqrt{\vep_{2p_2}}\notin K$.
	\item[13)] If $d_1=2p_1$ and $d_2=2p_2$ for distinct primes  with   $p_1\equiv5\pmod8$, $p_2\equiv1\pmod8 $, $\left(\frac{2}{p_2}\right)_4\neq \left(\frac{p_2}{2}\right)_4$, $\left(\frac{p_1}{p_2}\right)=1$  \text{ and }  $\left(\frac{p_1}{p_2}\right)_4\neq \left(\frac{p_2}{p_1}\right)_4$. Then According to o \cite[Theorems 1 and 2]{kuvcera1995parity} we have $N(\vep_{2p_2})=N(\vep_{p_1p_2})=-N(\vep_{2p_1})=1$, moreover $\sqrt{\vep_{2p_2}}=[\sqrt{2},\sqrt{p_2}]$ and $\sqrt{\vep_{p_1p_2}}=[\sqrt{p_1},\sqrt{p_2}]$, then $\sqrt{\vep_{2p_2}\vep_{p_1p_2}}\in K$ and $q(K)=2$.
	\item[14)] If $d_1=2p_1$ and $d_2=2p_2$ for distinct primes  with $p_1\equiv p_2\equiv1\pmod8 $, $\left(\frac{2}{p_1}\right)_4\neq \left(\frac{p_1}{2}\right)_4$, $\left(\frac{2}{p_2}\right)_4\neq \left(\frac{p_2}{2}\right)_4$  \text{ and } $\left(\frac{p_1}{p_2}\right)=-1$.  According to  \cite[Theorems 1 and 2]{kuvcera1995parity} we have $N(\vep_{2p_1})=N(\vep_{2p_2})=-N(\vep_{2p_2})=1$,  and $\sqrt{\vep_{2p_1}}=[\sqrt{2},\sqrt{p_1}]$ and $\sqrt{\vep_{2p_2}}=[\sqrt{2},\sqrt{p_2}]$, thus $\sqrt{\vep_{2p_2}\vep_{2p_2}}\in K$ and $q(K)=2$.
	\item[15)] If $d_1=2p_1$ and $d_2=2p_2$ for distinct primes  with $p_1\equiv p_2\equiv1\pmod8 $, $\left(\frac{2}{p_1}\right)_4\neq \left(\frac{p_1}{2}\right)_4$, $\left(\frac{2}{p_2}\right)_4\neq \left(\frac{p_2}{2}\right)_4$,   $\left(\frac{p_1}{p_2}\right)=1$ \text{ and }    $\left(\frac{p_1}{p_2}\right)_4\neq \left(\frac{p_2}{p_1}\right)_4$. Then according to \cite[Theorems 1 and 2]{kuvcera1995parity} we have $N(\vep_{2p_1})=N(\vep_{2p_2})=N(\vep_{2p_2})=1$,  and $\sqrt{\vep_{2p_1}}=[\sqrt{2},\sqrt{p_1}]$,  $\sqrt{\vep_{2p_2}}=[\sqrt{2},\sqrt{p_2}]$ and $\sqrt{\vep_{p_1p_2}}=[\sqrt{p_1},\sqrt{p_2}]$, thus $\sqrt{\vep_{2p_2}\vep_{2p_2}}$ and $\sqrt{\vep_{2p_2}\vep_{2p_2}}\in K$ and $q(K)=4.$

	\item[16)] If $d_1=p_1p_2$ and $d_2=p_2p_3$ for some distinct prime numbers $p_i\equiv1\pmod4$ and $\left(\frac{p_1}{p_2}\right)=\left(\frac{p_2}{p_3}\right)=\left(\frac{p_1}{p_3}\right)=-1$, this case is similar to the previous, hence $q(K)=2$ or $1$ according to whether $\sqrt{\vep_{p_1p_2}\vep_{p_2p_3}\vep_{p_1p_3}}$ is in $K$ or not.
	
	\item[17)] If $d_1=p_1p_2$ and $d_2=p_2p_3$ for some distinct prime numbers $p_i\equiv1\pmod4$ such that $\left(\frac{p_1}{p_2}\right)=\left(\frac{p_2}{p_3}\right)=-1$, $\left(\frac{p_1}{p_3}\right)=1$\text{ and } $\left(\frac{p_1}{p_3}\right)_4\neq \left(\frac{p_3}{p_1}\right)_4$. Then according to  \cite[Theorems 1 ]{kuvcera1995parity} we have $N(\vep_{p_1p_2})=N(\vep_{p_2p_3})=-N(\vep_{p_1p_3})=-1$, then $q(K)=1$ since the only square root candidate to be in $E_K$ is $\sqrt{\vep_{p_1p_3}}=[\sqrt{p_1},\sqrt{p_3}]$, which is not.
	
	\item[18)] If $d_1=p_1p_2$ and $d_2=p_2p_3$ for some distinct prime numbers $p_i\equiv1\pmod4$ such that $\left(\frac{p_1}{p_2}\right)=-1$, $\left(\frac{p_2}{p_3}\right)=\left(\frac{p_1}{p_3}\right)=1$,  $\left(\frac{p_1}{p_3}\right)_4\neq \left(\frac{p_3}{p_1}\right)_4$ \text{ and } $\left(\frac{p_2}{p_3}\right)_4\neq \left(\frac{p_3}{p_2}\right)_4$. Then  according to  \cite[Theorems 1 ]{kuvcera1995parity} we have $N(\vep_{p_2p_3})=N(\vep_{p_1p_3})=-N(\vep_{p_1p_2})=1$, thus $\sqrt{\vep_{p_2p_3}}=[\sqrt{p_2},\sqrt{p_3}]$ and   $\sqrt{\vep_{p_1p_3}}=[\sqrt{p_1},\sqrt{p_3}]$. Therefore $\sqrt{\vep_{p_2p_3}\vep_{p_1p_3}}\in K$ and $q(K)=2$.
	
	\item[19)] If $d_1=p_1p_2$ and $d_2=p_2p_3$ for some distinct prime numbers $p_i\equiv1\pmod4$ such that $\left(\frac{p_1}{p_2}\right)=\left(\frac{p_2}{p_3}\right)=\left(\frac{p_1}{p_3}\right)=1$,  $\left(\frac{p_1}{p_3}\right)_4\neq \left(\frac{p_3}{p_1}\right)_4$, $\left(\frac{p_2}{p_3}\right)_4\neq \left(\frac{p_3}{p_2}\right)_4$  \text{ and } $\left(\frac{p_1}{p_2}\right)_4\neq \left(\frac{p_2}{p_1}\right)_4$.  Then  according to  \cite[Theorems 1 ]{kuvcera1995parity} we have $N(\vep_{p_2p_3})=N(\vep_{p_1p_3})=N(\vep_{p_1p_2})=1$, thus $\sqrt{\vep_{p_2p_3}}=[\sqrt{p_2},\sqrt{p_3}]$,   $\sqrt{\vep_{p_1p_3}}=[\sqrt{p_1},\sqrt{p_3}]$ and $\sqrt{\vep_{p_1p_2}}=[\sqrt{p_1},\sqrt{p_2}]$ . Therefore $\sqrt{\vep_{p_2p_3}\vep_{p_1p_3}}$ and $\sqrt{\vep_{p_2p_3}\vep_{p_1p_2}}\in K$ and $q(K)=4$.
	
\end{enumerate}
Plugging all the above investigations together, we complete the proof of our   theorem.
\end{proof}

\medskip	

The results established in this section will be useful in the classification of real multiquadratic fields with odd class numbers, which will be presented in the next section.

\section{\bf Real Multiquadratic Number Fields with Odd Class Number}

Let $D(n):=\QQ(\sqrt{\ell_1},\sqrt{\ell_2},\cdots,\sqrt{\ell_n})$ (or more simply $K$),  where $n\geq 1$ is an integer and all the $\ell_i$ are positive square-free integers,   be a multiquadratic field of $2^n$. In this section, we give the complete classification of real multiquadratic fields $D(n)$ whose class number is odd. Let us   recall the following key step    result, due to Fröhlich.

\begin{theorem} [{Fr\"{o}hlich \cite[Theorem 5.6]{frohlich}}]\label{FrohlichResults} Let $L$ be a real abelian number field of $2$-power degree. The class number of $L$ is even whenever the number of finite rational primes that ramify in $L$ is greater than or equal to $5$.
\end{theorem}

\medskip
Thus, the class number of $D(n) $ is even for every $n\geq 5$. Therefore, to complete  the classification, it suffices to consider real multiquadratic fields  $D(n)$ with $n\leq 4$.
\medskip

For $n=1$,	let us recall the following result of Conner and Hurrelbrink that classifies real  quadratic fields with odd class numbers.
\medskip
\begin{lemma}[\cite{{conner88}}, Corollary 18.4] Let $K$ be a real quadratic field. The class number of $K$ is odd   if and only if it takes one of the following forms:
\begin{enumerate}[$    1)$]
	\item $K=\QQ(\sqrt{2})$, $\QQ(\sqrt{p})$ where $p$  is a   prime number congruent to $1\pmod 4$. 
	\item $K=\QQ(\sqrt{q})$, $\QQ(\sqrt{2q})$ or $\QQ(\sqrt{q_1q_2})$  where  $q$, $q_1$, $q_2$ are prime  numbers congruent to $3\pmod 4$.
\end{enumerate}
\end{lemma}

\medskip

Real biquadratic fields, i.e.,   $D(2)$, with odd class numbers are classified by  the well-known results of Conner and Hurrelbrink $($cf. \cite[Theorems 1 and 2]{kuvcera1995parity}$)$. However, their classification contains a minor omission, as some families of real biquadratic number fields with odd class numbers were overlooked. By combining the  corrected  version of their result with Ku{\v{c}}era's work (cf.~\cite[Theorems 1 and 2]{kuvcera1995parity}), we obtain the following theorem that provides a complete  classification of real biquadratic fields with odd class numbers.

\medskip

\begin{lemma}\label{realBiquadCorre} Let $K$ be a real biquadratic number field. The class number of $K$ is odd if and only if   it takes one of the following forms:
\begin{enumerate}[$1)$]
	\item $K=\QQ(\sqrt{q},\sqrt{p})$ or $\QQ(\sqrt{2q},\sqrt{p})$ with $\left(\frac{p}{q}\right)=-1$ or $p\equiv5 \pmod 8$,
	\item	$K=\QQ(\sqrt{2},\sqrt{q_1q_2})$ with   $q_1\equiv3 \pmod 8$ or $q_2\equiv3 \pmod 8$, 
	\item $K=\QQ(\sqrt{p},\sqrt{q_1q_2})$  with $\left(\frac{q_1}{p}\right)=-1$ or $\left(\frac{q_2}{p}\right)=-1$,
	\item  $K=\QQ(\sqrt{2},\sqrt{p})$ with $\left(\frac{2}{p}\right)=-1$ or $[\left(\frac{2}{p}\right)=1$ and $\left(\frac{2}{p}\right)_4\not=\left(\frac{p}{2}\right)_4]$,

	\item $K=\QQ(\sqrt{p_1},\sqrt{p_2})$ with $\left(\frac{p_1}{p_2}\right)=-1$ or $[\left(\frac{p_1}{p_2}\right)=1$ and $\left(\frac{p_1}{p_2}\right)_4\not=\left(\frac{p_2}{p_1}\right)_4]$,
	\item $K=\QQ(\sqrt{2},\sqrt{q})$,   $\QQ(\sqrt{q_1},\sqrt{q_2})$, $\QQ(\sqrt{2q_1},\sqrt{2q_2})$ or $\QQ(\sqrt{q_1q_2},\sqrt{q_2q_3}).$
	
\end{enumerate}
Here, $p$, $p_1$, $p_2$  denote  primes numbers congruent to $1\pmod 4$ and $q$, $q_1$, $q_2$, $q_3$ denote  primes numbers congruent to $3\pmod 4$.
\end{lemma}
\begin{proof}
We have two main cases to investigate.

\noindent\ding{229} Assume that $K$ is unramified over exactly one of its three quadratic subfields denote $k$.
Notice that if  the class number of $K$ is odd, then  $h_2(k)=2$. So  Lemma  \ref{realQuadeven1New}  implies that  $k$
 is in one of the following forms:
\begin{enumerate}[$1)$]
		\item $k=\QQ(\sqrt{pq})$ or $\QQ(\sqrt{2pq})$  with  $\left(\frac{p}{q}\right)=-1$ or $p\equiv5\pmod8$.
	\item $k=\QQ(\sqrt{2q_1q_2})$ with $q_1\equiv3\pmod8$ or $q_2\equiv3\pmod8.$
	\item $k=\QQ(\sqrt{pq_1q_2})$ with $\left(\frac{p}{q_1}\right)=-1$ or $\left(\frac{p}{q_2}\right)=-1.$

	\item $k=\QQ(\sqrt{2p})$ with $\left(\frac{2}{p}\right)=-1$ or $[\left(\frac{2}{p}\right)=1$ and $\left(\frac{2}{p}\right)_4\not=\left(\frac{p}{2}\right)_4]$.

	\item $k=\QQ(\sqrt{p_1p_2})$ with $\left(\frac{p_1}{p_2}\right)=-1$ or $[\left(\frac{p_1}{p_2}\right)=1$ and $\left(\frac{p_1}{p_2}\right)_4\not=\left(\frac{p_2}{p_1}\right)_4]$.
\end{enumerate}
Hence, by  class field theory, the class number of $K$ is odd if and only if $K$ takes one of the forms listed in items enumerated from  $1)$ to $5)$ in the above theorem.

\noindent\ding{229} Now assume that $K$ is ramified over each of its three quadratic subfields. Assume that the class number of $K$ is odd and let $k$ be any of of the quadratic subfields of $K$. Thus we have $h_2(k)=1$ or $2$.
It follows from Lemmas  \ref{realQuadodd} and \ref{realQuadeven1New} that we have 
$$K=\QQ(\sqrt{2},\sqrt{q}),\quad    \QQ(\sqrt{q_1},\sqrt{q_2}),\quad \QQ(\sqrt{2q_1},\sqrt{2q_2})\  \text{ or }\ \QQ(\sqrt{q_1q_2},\sqrt{q_2q_3}).$$
Here $q_1$, $q_2$ and $q_3$ are prime numbers congruent to $3\pmod4$ and  satisfy some extra conditions as follows.
\begin{enumerate}[$\bullet$]
	\item If $K=\QQ(\sqrt{2},\sqrt{q})$,
 then $q\equiv3\pmod4$ is a prime number. Notice that in this case, we deduce from  Lemma \ref{unitsrs} that  we have $E_K= \langle{-1,\vep_{2},\sqrt{\vep_{q}},\sqrt{\vep_{2q}}\rangle}$. Therefore,    Kuroda's class number formula and Lemma  \ref{realQuadodd} implies that the class number of $K$ is odd.
	\item If $K=\QQ(\sqrt{q_1},\sqrt{q_2})$,  then  $q_1\equiv q_2\equiv3\pmod4$ are distinct prime numbers. Notice that in this case, by Lemma \ref{unitsrs} we get $E_K= \langle{-1,\vep_{q_1},\sqrt{\vep_{q_1}\vep_{q_2}},\sqrt{\vep_{q_1q_2}}\rangle}$ and so $q(K)=2$.  Thus, by Kuroda's class number formula, the class number of $K$ is odd.
	\item If $K=\QQ(\sqrt{2q_1},\sqrt{2q_2})$, then  $q_1\equiv q_2\equiv3\pmod4$ are distinct prime numbers. Notice that in this case, Lemma \ref{unitsrs}  implies that $E_K= \langle{-1,\vep_{2q_1},\sqrt{\vep_{2q_1}\vep_{2q_2}},\sqrt{\vep_{2q_1}\vep_{q_1q_2}}\rangle}$ and so $q(K)=2$. Then 
	Kuroda's class number formula implies that the class number of $K$ is odd.
	\item  If $K=\QQ(\sqrt{q_1q_2},\sqrt{q_2q_3})$,  then   $q_1\equiv q_2\equiv q_3\equiv3\pmod4$ are distinct prime numbers. Notice that in this case,  Lemma \ref{unitsrs} gives $E_K= \langle{-1,\vep_{q_1q_2},\sqrt{\vep_{q_1q_2}\vep_{q_2q_3}},\sqrt{\vep_{q_1q_2}\vep_{q_1q_3}}\rangle}$ and so $q(K)=2$. Then by	Kuroda's class number formula, the class number of $K$ is odd.
\end{enumerate}
Hence, putting all of the above together, we obtain the classification stated in this theorem. 
\end{proof}

\medskip

Now let us classify real triquadratic number fields, i.e., $D(3)$, with odd class numbers. 
We regard this classification as a particularly interesting application of the results established in the previous section.
Before stating our result, we note that the real triquadratic fields $ \QQ(\sqrt{\ell_1},\sqrt{\ell_2},\sqrt{\ell_2})$, where $\ell_1\equiv\ell_2\equiv\ell_3\equiv1\pmod4$ are distinct prime numbers, with odd class numbers were classified by   Bulant in (cf. \cite{Bulant}). More generally, we prove the following result. 

\medskip

\begin{theorem}\label{triquawithodd}
Let $K$ be  real triquadratic number field. Then the class number of $ K $ is odd if and only if $K$ is in one of the following forms:
\begin{enumerate}[$1)$]
	\item $K=\QQ(\sqrt{\ell_1},\sqrt{\ell_2},\sqrt{\ell_3})$, where $\ell_1\equiv\ell_2\equiv\ell_3\equiv 1\pmod4$ are distinct prime numbers such that one of the following conditions holds:
	\begin{enumerate}[$a)$]
		\item \( \left( \frac{\ell_1}{\ell_2} \right) = \left( \frac{\ell_1}{\ell_3} \right) = \left( \frac{\ell_2}{\ell_3} \right) = -1 \) and
		$\left( \frac{\ell_1\ell_2}{\ell_3} \right)_4 \left( \frac{\ell_1\ell_3}{\ell_2} \right)_4 \left( \frac{\ell_2\ell_3}{\ell_1} \right)_4 = 1,
		$
		\item \( \left( \frac{\ell_1}{\ell_2} \right) = 1 \), \( \left( \frac{\ell_1}{\ell_3} \right) = \left( \frac{\ell_2}{\ell_3} \right) = -1 \) and $\left(\frac{\ell_1}{\ell_2}\right)_4\not=\left(\frac{\ell_2}{\ell_1}\right)_4.$
	\end{enumerate}
	
	\item $K=\QQ(\sqrt{2},\sqrt{\ell_2},\sqrt{\ell_3})$, where $\ell_2\equiv\ell_3\equiv 1\pmod4$ are distinct prime numbers such that one of the following conditions holds:
	\begin{enumerate}[$a)$]
		\item $\left(\frac{\ell_2}{\ell_3}\right)=-1$, $h(2\ell_2 \ell_3)=4$ and \textbf{none} of the two units $\{\vep_{2}\vep_{\ell_2\ell_3}\vep_{2\ell_2\ell_3}, \vep_{2\ell_2\ell_3}\}$ is a square in $K$,
		\item $\left(\frac{\ell_2}{\ell_3}\right)=1$ and $\left(\frac{\ell_2}{\ell_3}\right)_4\not=\left(\frac{\ell_3}{\ell_2}\right)_4$,
		$h(2\ell_2\ell_3)=4$ and \textbf{none} of the two units $\{ \vep_{\ell_2\ell_3}\vep_{2\ell_2\ell_3}, \vep_{2\ell_2\ell_3}\}$ is a square in $K$.
	\end{enumerate} 
	
	\item $K=\QQ(\sqrt{2},\sqrt{\ell_2},\sqrt{\ell_3})$, where $\ell_2\equiv-\ell_3\equiv 1\pmod4$ are distinct prime numbers such that $\ell_2\equiv5\pmod8$. 
	\item $K=\QQ(\sqrt{2},\sqrt{\ell_2},\sqrt{\ell_3})$, where $\ell_2\equiv \ell_3\equiv 3\pmod4$ distinct prime numbers such that $\ell_2$ or $\ell_3 \equiv 3\pmod8$.
	\item $K=\QQ(\sqrt{\ell_1},\sqrt{\ell_2},\sqrt{\ell_3})$, where $\ell_1 \equiv\ell_2 \equiv\ell_3 \equiv3\pmod4$ are distinct prime numbers such that  one of the following conditions holds: \begin{enumerate}[$a)$]
		\item $\ell_1\equiv \ell_3 \equiv 7\pmod{8}$, $\ell_2\equiv3\pmod4$ and  $\left(\dfrac{\ell_3}{\ell_1}\right)=-\left(\dfrac{\ell_3}{\ell_2}\right)=\left(\dfrac{\ell_1}{\ell_2}\right)=1$,
		\item  $\ell_1\equiv \ell_2 \equiv 3\pmod{8}$ and  $\left(\dfrac{\ell_3}{\ell_1}\right)=-\left(\dfrac{\ell_3}{\ell_2}\right)=\left(\dfrac{\ell_1}{\ell_2}\right)=1$,
		
		
		\item  $\ell_2\equiv7\pmod8, \ell_3 \equiv 3\pmod{8}$, $\ell_1\equiv3\pmod4$  and  $\left(\dfrac{\ell_3}{\ell_1}\right)=\left(\dfrac{\ell_3}{\ell_2}\right)=-\left(\dfrac{\ell_1}{\ell_2}\right)=-1$,
		
		
		\item $\ell_2\equiv3\pmod8, \ell_3 \equiv 7\pmod{8}$, $\ell_1\equiv3\pmod4$  and  $\left(\dfrac{\ell_3}{\ell_1}\right)=\left(\dfrac{\ell_3}{\ell_2}\right)=-\left(\dfrac{\ell_1}{\ell_2}\right)=-1$.
		
		
	\end{enumerate} 
	\item $K=\QQ(\sqrt{\ell_1},\sqrt{\ell_2},\sqrt{\ell_3})$, where $\ell_1\equiv -\ell_2\equiv-\ell_3\equiv 1\pmod4$ distinct prime numbers such that one of the following conditions holds:
	\begin{enumerate}[$a)$]
		\item $\ell_1\equiv5\pmod{8}$ and
		\[
		\left(\left(\frac{\ell_3}{\ell_1}\right),\left(\frac{\ell_3}{\ell_2}\right),
		\left(\frac{\ell_1}{\ell_2}\right)\right)
		\in\{(i,j,k):i,j,k=\pm1\}
		\setminus\{(1,1,1),(1,-1,1)\},
		\]
		
		\item    $\ell_1\equiv1\pmod{8}$ and  $\left(\dfrac{\ell_3}{\ell_1}\right)=\left(\dfrac{\ell_1}{\ell_2}\right)=-1$.
	\end{enumerate}

	\item $K=\QQ(\sqrt{\ell_1},\sqrt{\ell_2},\sqrt{\ell_3})$, where $\ell_1\equiv \ell_2\equiv-\ell_3\equiv 1\pmod4$ distinct prime numbers such that one of the following conditions holds:
	\begin{enumerate}[$a)$]
		\item $\ell_1\equiv5\pmod8, \ell_2\equiv1\pmod8 $ and $\left(\frac{\ell_2}{\ell_1}\right)=\left(\frac{\ell_2}{\ell_3}\right)=-\left(\frac{\ell_1}{\ell_3}\right)=-1,$
		\item $\ell_1\equiv5\pmod8, \ell_3\equiv3\pmod8$ and $\left(\frac{\ell_2}{\ell_1}\right)=\left(\frac{\ell_2}{\ell_3}\right)=\left(\frac{\ell_1}{\ell_3}\right)=-1,$
		
		\item $\ell_1\equiv \ell_2\equiv5\pmod8$ and $\left(\frac{\ell_2}{\ell_1}\right)=-\left(\frac{\ell_2}{\ell_3}\right)=\left(\frac{\ell_1}{\ell_3}\right)=-1.$ 
	\end{enumerate}
	\item $K=\QQ(\sqrt{\ell_1},\sqrt{\ell_2},\sqrt{\ell_3\ell_4})$, where $\ell_1\equiv \ell_2\equiv-\ell_3\equiv-\ell_4\equiv1\pmod4$ are distinct prime numbers such that one of the following conditions holds:
	\begin{enumerate}[$a)$]
		\item $\left(\frac{\ell_1}{\ell_2}\right)=\left(\frac{\ell_1}{\ell_3}\right)=\left(\frac{\ell_1}{\ell_4}\right)=\left(\frac{\ell_1}{\ell_3}\right)=-\left(\frac{\ell_1}{\ell_4}\right)=-1,$
		\item $\left(\frac{\ell_1}{\ell_2}\right)=-\left(\frac{\ell_1}{\ell_3}\right)=\left(\frac{\ell_1}{\ell_4}\right)=\left(\frac{\ell_1}{\ell_3}\right)=-\left(\frac{\ell_1}{\ell_4}\right)=\left(\frac{\ell_3}{\ell_4}\right)=-1.$
	\end{enumerate}
	\item $K=\QQ(\sqrt{2},\sqrt{\ell_2},\sqrt{\ell_3\ell_4})$, where $ \ell_2\equiv-\ell_3\equiv-\ell_4\equiv1\pmod4$ are distinct prime numbers such that one of the following conditions holds:
	\begin{enumerate}[$a)$]
		\item $\ell_2\equiv5\pmod8,\ell_3\equiv \ell_4\equiv3\pmod8$ and $\left(\frac{\ell_2}{\ell_3}\right)=-\left(\frac{\ell_2}{\ell_4}\right)=-1,$ 
		
		\item $\ell_2\equiv5\pmod8,\ell_3\equiv7\pmod8, \ell_4\equiv3\pmod8$ and $\left(\frac{\ell_2}{\ell_3}\right)=-\left(\frac{\ell_2}{\ell_4}\right)=\left(\frac{\ell_4}{\ell_3}\right)=-1.$
	\end{enumerate}
	
	\item $K=\QQ(\sqrt{\ell_1},\sqrt{\ell_2},\sqrt{2\ell_3})$, where $ \ell_1\equiv\ell_2\equiv-\ell_3\equiv1\pmod4$ are distinct prime numbers such that one of the following conditions holds:
	\begin{enumerate}[$a)$]
		\item $\ell_1\equiv5\pmod8, \ell_2\equiv1\pmod{8}$, $\ell_3\equiv3\pmod{8}$ and
		$\left(\dfrac{\ell_3}{\ell_1}\right)=-\left(\dfrac{\ell_3}{\ell_2}\right)=-\left(\dfrac{\ell_1}{\ell_2}\right)=1$,
		
		\item $\ell_1\equiv \ell_2\equiv5\pmod8$  and
		$\left(\dfrac{\ell_3}{\ell_1}\right)=-\left(\dfrac{\ell_3}{\ell_2}\right)=-\left(\dfrac{\ell_1}{\ell_2}\right)=1$,
		\item $\ell_1\equiv1\pmod8, \ell_2\equiv5\pmod8$  and
		$\left(\dfrac{\ell_3}{\ell_1}\right)=\left(\dfrac{\ell_3}{\ell_2}\right)=\left(\dfrac{\ell_1}{\ell_2}\right)=-1$.
	\end{enumerate}
	
	\item $K=\QQ(\sqrt{\ell_1\ell_2},\sqrt{2\ell_1},\sqrt{\ell_4})$, where $\ell_1\equiv\ell_2\equiv-\ell_4\equiv3\pmod4$ are distinct prime numbers such that one of the following conditions holds:
	\begin{enumerate}[$a)$]
		\item $\left(\frac{\ell_4}{\ell_1}\right)=\left(\frac{\ell_4}{\ell_2}\right)=-1,$ 
		\item $\left(\frac{\ell_4}{\ell_1}\right)=-\left(\frac{\ell_4}{\ell_2}\right)=1$ and $\ell_4\equiv5\pmod8.$ 
	\end{enumerate}

	\item $K=\QQ(\sqrt{2\ell_2},\sqrt{2\ell_3},\sqrt{\ell_4})$,  where $\ell_2\equiv\ell_3\equiv-\ell_4\equiv3\pmod4$ are distinct prime numbers such that one of the following conditions holds:
	
		\begin{enumerate}[$a)$]
		\item $\ell_4\equiv5\pmod8$ and $\left(\dfrac{\ell_4}{\ell_3}\right)=-1$ or $[\left(\dfrac{\ell_4}{\ell_3}\right)=-\left(\dfrac{\ell_4}{\ell_2}\right)=1]$,
		
		\item $\ell_4\equiv1\pmod8$ and  $\left(\dfrac{\ell_4}{\ell_3}\right)=\left(\dfrac{\ell_4}{\ell_2}\right)=-1$.
	\end{enumerate}

%

	\item $K=\QQ(\sqrt{\ell_1\ell_2},\sqrt{\ell_1\ell_3},\sqrt{2})$, where $\ell_1\equiv\ell_2\equiv\ell_3\equiv3\pmod4$ are distinct prime numbers such  that $\ell_3\equiv3\pmod8$.
	
	\item $K=\QQ(\sqrt{\ell_1\ell_2},\sqrt{\ell_1\ell_3},\sqrt{\ell_1\ell_4})$, where $\ell_1\equiv\ell_2\equiv\ell_3\equiv\ell_4\equiv3\pmod4$ are distinct prime numbers such  that one of the following conditions holds:
	\begin{enumerate}[$a)$]
		\item $\left(\frac{\ell_2}{\ell_3}\right)=\left(\frac{\ell_2}{\ell_4}\right)=\left(\frac{\ell_4}{\ell_1}\right)=-\left(\frac{\ell_2}{\ell_1}\right)=-\left(\frac{\ell_3}{\ell_4}\right)=1,$
		\item $\left(\frac{\ell_2}{\ell_3}\right)=\left(\frac{\ell_2}{\ell_4}\right)=\left(\frac{\ell_4}{\ell_1}\right)=\left(\frac{\ell_2}{\ell_1}\right)=-\left(\frac{\ell_3}{\ell_1}\right)=-\left(\frac{\ell_3}{\ell_4}\right)=1.$
	\end{enumerate}

	\item $K=\QQ(\sqrt{\ell_1\ell_2},\sqrt{2\ell_1},\sqrt{\ell_1\ell_4})$, where $\ell_1\equiv\ell_2\equiv\ell_3\equiv3\pmod4$ are distinct prime numbers such  that one of the following conditions holds:
	\begin{enumerate}[$a)$]
		\item $\ell_2\equiv \ell_4\equiv7\pmod{8}$, $\ell_1\equiv3\pmod{8}$ and
		$\left(\dfrac{\ell_2}{\ell_1}\right)=-\left(\dfrac{\ell_2}{\ell_4}\right)=\left(\dfrac{\ell_1}{\ell_4}\right)=1$,

		\item $\ell_1\equiv\ell_4\equiv3\pmod{8}$, $\ell_2\equiv7\pmod{8}$ and
		$\left(\dfrac{\ell_2}{\ell_1}\right)=-\left(\dfrac{\ell_2}{\ell_4}\right)=-\left(\dfrac{\ell_1}{\ell_4}\right)=-1$.
	\end{enumerate}
\end{enumerate}
\end{theorem}
\begin{proof}
Let $K=\QQ(\sqrt{\ell_1},\sqrt{\ell_2},\sqrt{\ell_3})$, where $\ell_1, \ell_2$ and $\ell_3$ be distinct positive square-free integers. Notice that if  $h_2(K)=1$, then by Theorem \ref{FrohlichResults}, we have:
$$M_K:=\mathrm{Card}\{p\in \mathbb{P} : p|\delta_K \}\leq 4.$$
Here $\mathbb{P}$ denotes the set of  prime numbers and $\delta_K$ denotes  the absolute discriminant  of $K$. 
Notice that we have $ M_K=3$ or $4$. Thus, we have two main cases to investigate.

\noindent\ding{229} Let us start by assuming that $M_K=3$. 
This implies that  all the integers $\ell_i$ are prime numbers and so we have the following sub-cases:
\begin{enumerate}[$1)$]
	\item Let  $\ell_1 \equiv\ell_2 \equiv\ell_3 \equiv 1\pmod4$ be prime numbers. Then according to \cite{Bulant}, the class number of $K$ is odd if and only if $K$ takes the form in the first item of our theorem.

	\item Let $\ell_1=2$ and $\ell_2\equiv \ell_3 \equiv1\pmod4$ be distinct  prime numbers.
	Notice that $K/\QQ(\sqrt{2},\sqrt{\ell_1\ell_2})$ is a quadratic unramified extension. Therefore, 
	$h(K)$ is odd if and only if  the $2$-class number of $\QQ(\sqrt{2},\sqrt{\ell_1\ell_2})$ equals $2$, which is, by Theorem \ref{class2d} (cf. case $D)$), equivalent to the fact that $K$ takes the forms in the second item of our theorem.

	\item Let  $\ell_1=2$ and $\ell_2\equiv -\ell_3 \equiv1\pmod4$  be distinct   prime numbers. As $K/k$ is an unramified quadratic extension with $k=\QQ(\sqrt{2\ell_2},\sqrt{\ell_3})$, we have  $h(K)$ is odd if and only if $h_2(k)=2$. So by Theorem \ref{classqrs} (cf. case $C)$),   we conclude that, $h_2(k)=2$ exactly when $\ell_2\equiv 5 \pmod8.$ 
	
	\item Let $\ell_1=2$ and $\ell_2\equiv \ell_3 \equiv3\pmod4$ be distinct   prime numbers. As $K/k$ is an unramified quadratic extension with $k=\QQ(\sqrt{2\ell_2},\sqrt{\ell_3})$, then  by Theorem \ref{class2qrs} (cf. case $A)-1)$)  $h(K)$ is odd  if and only if $\ell_2$ or $\ell_3 \equiv 3\pmod8$. 
	
	\item Let $\ell_1\equiv \ell_2\equiv\ell_3\equiv3\pmod4$ be distinct   prime numbers. Notice that $K/k$,  with $k=\QQ(\sqrt{\ell_1\ell_2},\sqrt{\ell_3})$, is an unramified quadratic extension. Therefore, by Theorem \ref{classq1q2d} (cf. Item $D)$), 
	$h(K)$ is odd if and only if $d$ takes one of the forms in $d$ takes one of the forms listed in Item $5)$ of our theorem.
	
	\item  Let $\ell_1\equiv -\ell_2\equiv-\ell_3\equiv 1\pmod4$ be distinct prime numbers. Consider the quadratic unramified extension $K/k$, $k=\QQ(\sqrt{\ell_1\ell_2},\sqrt{\ell_3})$. So by  Theorem \ref{classqrs} (cf. Items $B)-1)$ and $B)-2)$),   $h_2(K)=1$ if and only if  $d$ takes one of the conditions listed in  $6)$ of our theorem.
	
	\item Let $\ell_1\equiv \ell_2\equiv-\ell_3\equiv 1\pmod4$ be distinct prime numbers. Notice that $K/k$, with $k=\QQ(\sqrt{\ell_1\ell_3},\sqrt{\ell_2})$, is an unramified quadratic extension.  So by  Theorem \ref{classpd} (cf. Item $E)$),   $h_2(K)=1$ if and only if $d$ takes one of the conditions listed in  $7)$ of our theorem.
\end{enumerate}

\noindent\ding{229}	Now assume that $M_K=4$.
So $K$ is in one of the following forms: 
$$\QQ(\sqrt{\ell_1},\sqrt{\ell_2},\sqrt{\ell_3\ell_4}),\; \; \QQ(\sqrt{\ell_1\ell_2},\sqrt{\ell_1\ell_3},\sqrt{\ell_4})\;\; or \;\; \QQ(\sqrt{\ell_1\ell_2},\sqrt{\ell_1\ell_3},\sqrt{\ell_1\ell_4}),  $$ where $\ell_i$ are distinct prime numbers.  Thus, we have $3$ main cases to investigate. 
\begin{enumerate}[$\bullet$]
	\item Let $K=\QQ(\sqrt{\ell_1},\sqrt{\ell_2},\sqrt{\ell_3\ell_4})$. Notice that $\QQ(\sqrt{\ell_1},\sqrt{\ell_2},\sqrt{\ell_3},\sqrt{\ell_4})/K$ is ramified if and only if one of the following three congruence conditions holds: \begin{enumerate}[$\mathrm{C}1)$]
		\item $\ell_1\equiv \ell_2\equiv-\ell_3\equiv-\ell_4\equiv1\pmod4,$
		\item $\ell_1=2$ and $\ell_2\equiv-\ell_3\equiv-\ell_4\equiv1\pmod4,$
		\item $\ell_1\equiv\ell_2\equiv-\ell_3\equiv1\pmod4$ and $\ell_4=2.$
	\end{enumerate}
	Otherwise we have $h_2(K)\geq 2$, and such fields must be rejected. Now, for each of the cases above, we choose $k=\QQ(\sqrt{\ell_1\ell_2},\sqrt{\ell_3\ell_4})$. It follows that $K/k$ is an unramified quadratic extension. For the case $\mathrm{C}1)$ we have according to Theorem \ref{classq1q2d} (cf. $F)-1)$ and $F)-2)$), $h_2(K)=1$ if and only of one of the following conditions holds: 
	
	\begin{enumerate}
		\item $\left(\frac{\ell_1}{\ell_2}\right)=\left(\frac{\ell_1}{\ell_3}\right)=\left(\frac{\ell_1}{\ell_4}\right)=\left(\frac{\ell_1}{\ell_3}\right)=-\left(\frac{\ell_1}{\ell_4}\right)=-1,$
		\item $\left(\frac{\ell_1}{\ell_2}\right)=-\left(\frac{\ell_1}{\ell_3}\right)=\left(\frac{\ell_1}{\ell_4}\right)=\left(\frac{\ell_1}{\ell_3}\right)=-\left(\frac{\ell_1}{\ell_4}\right)=\left(\frac{\ell_3}{\ell_4}\right)=-1.$
	\end{enumerate}
	
	For the case $\mathrm{C}2)$ we have according to Theorem \ref{classq1q2d} (cf. $C)-1)$ and $C)-2)$), $h_2(K)=1$ if and only of one of the following conditions holds: 
	
	\begin{enumerate}
		\item $\ell_2\equiv5\pmod8,\ell_3\equiv \ell_4\equiv3\pmod8$ and $\left(\frac{\ell_2}{\ell_3}\right)=-\left(\frac{\ell_2}{\ell_4}\right)=-1,$ or
		
		\item $\ell_2\equiv5\pmod8,\ell_3\equiv7\pmod8, \ell_4\equiv3\pmod8$ and $\left(\frac{\ell_2}{\ell_3}\right)=-\left(\frac{\ell_2}{\ell_4}\right)=\left(\frac{\ell_4}{\ell_3}\right)=-1.$
	\end{enumerate}
	
	For the remaining case $\mathrm{C}3)$  in this sub-case we have  according to Theorem \ref{class2qrs} (cf. $B)-3)$ to $B)-5)$), $h_2(K)=1$ if and only of one of the following conditions holds:
	\begin{enumerate}
		\item $\ell_1\equiv5\pmod8, \ell_2\equiv1\pmod{8}$, $\ell_3\equiv3\pmod{8}$ and
		$\left(\dfrac{\ell_3}{\ell_1}\right)=-\left(\dfrac{\ell_3}{\ell_2}\right)=-\left(\dfrac{\ell_1}{\ell_2}\right)=1$,
		
		\item $\ell_1\equiv \ell_2\equiv5\pmod8$  and
		$\left(\dfrac{\ell_3}{\ell_1}\right)=-\left(\dfrac{\ell_3}{\ell_2}\right)=-\left(\dfrac{\ell_1}{\ell_2}\right)=1$,
		\item $\ell_1\equiv1\pmod8, \ell_2\equiv5\pmod8$  and
		$\left(\dfrac{\ell_3}{\ell_1}\right)=\left(\dfrac{\ell_3}{\ell_2}\right)=\left(\dfrac{\ell_1}{\ell_2}\right)=-1$.
	\end{enumerate}

	\item Let $K=\QQ(\sqrt{\ell_1\ell_2},\sqrt{\ell_1\ell_3},\sqrt{\ell_4})$. Notice that $\QQ(\sqrt{\ell_1},\sqrt{\ell_2},\sqrt{\ell_3},\sqrt{\ell_4})/K$ ramify if and only if one of the following conditions holds : \begin{enumerate}[$\mathrm{C}1)$]
		\item $\ell_1\equiv \ell_2\equiv-\ell_4\equiv3\pmod4$ and $\ell_3=2,$
		\item $\ell_1=2$ and $\ell_2\equiv\ell_3\equiv-\ell_4\equiv3\pmod4,$
		\item $\ell_1\equiv\ell_3\equiv-\ell_4\equiv3\pmod4$ and $\ell_2=2.$
		\item $\ell_1\equiv\ell_2\equiv\ell_3\equiv3\pmod4$ and $\ell_4=2.$
	\end{enumerate}

	Following the same idea as in the previous case, and using the unramified (under all above conditions $C1)$--$C4)$) quadratic extension $K/k$ where  $k=\QQ(\sqrt{\ell_1\ell_2},\sqrt{\ell_1\ell_3\ell_4})$, we obtain the following investigation:\\
	In case $C1)$, it follows from Theorem \ref{classq1q2d} (cf. $J)-1)$ and $J)-2)$) we have $h_2(K)=1$ if and only if one of the following conditions holds:
	
	\begin{enumerate}
		\item $\left(\frac{\ell_4}{\ell_1}\right)=\left(\frac{\ell_4}{\ell_2}\right)=-1,$ 
		\item $\left(\frac{\ell_4}{\ell_1}\right)=-\left(\frac{\ell_4}{\ell_2}\right)=1$ and $\ell_4\equiv5\pmod8.$
	\end{enumerate}

	For the case $\mathrm{C}2)$,  according to Theorem \ref{class2qrs} (cf. $D)-1)$ and $D)-2)$) we have $h_2(K)=1$ if and only if one of the following conditions holds:
	\begin{enumerate}[$a)$]
		\item $\ell_4\equiv5\pmod8$ and $\left(\dfrac{\ell_4}{\ell_3}\right)=-1$ or $[\left(\dfrac{\ell_4}{\ell_3}\right)=-\left(\dfrac{\ell_4}{\ell_2}\right)=1]$,
		
		\item $\ell_4\equiv1\pmod8$ and  $\left(\dfrac{\ell_4}{\ell_3}\right)=\left(\dfrac{\ell_4}{\ell_2}\right)=-1$.
	\end{enumerate}
	
%

	Now for case $C3)$. By applying  Theorems \ref{class2qrs} and \ref{realBiquad} we conclude that $h_2(K)\geq2.$ Then  such field must be rejected.\\
	Finally if condition $C4)$ holds, then from Theorem \ref{classq1q2d} (cf. $H)$) it follows that $h_2(K)=1$ if and only if $\ell_3\equiv 3\pmod8$.

	\item Let $K=\QQ(\sqrt{\ell_1\ell_2},\sqrt{\ell_1\ell_3},\sqrt{\ell_1\ell_4})$. Notice that $\QQ(\sqrt{\ell_1},\sqrt{\ell_2},\sqrt{\ell_3},\sqrt{\ell_4})/K$ ramify if and only if one of the following conditions holds : \begin{enumerate}[$\mathrm{C}1)$]
		\item $\ell_1\equiv \ell_2\equiv\ell_3\equiv\ell_4\equiv3\pmod4$,
		\item $\ell_3=2$ and $\ell_1\equiv\ell_2\equiv\ell_4\equiv3\pmod4,$
		\item $\ell_1=2$ and  $\ell_2\equiv\ell_4\equiv-\ell_3\equiv3\pmod4$.
	\end{enumerate}

	By the same investigation as in the two above forms of $K$, by using the suitable unramified (under conditions $C1)$--$C3)$) quadratic extension  $K/k$, where $k=\QQ(\sqrt{\ell_2\ell_3},\sqrt{\ell_1\ell_4})$ we caracterise eauch case separably:\\
	If conditions C1) holds, then by Theorem \ref{classq1q2d} (cf. $G)-1)$ and $G)-2)$)  we have $h_2(K)=1$ if and only if one of the following conditions holds:
	
	\begin{enumerate}
		\item $\left(\frac{\ell_2}{\ell_3}\right)=\left(\frac{\ell_2}{\ell_4}\right)=\left(\frac{\ell_4}{\ell_1}\right)=-\left(\frac{\ell_2}{\ell_1}\right)=-\left(\frac{\ell_3}{\ell_4}\right)=1,$ or
		\item $\left(\frac{\ell_2}{\ell_3}\right)=\left(\frac{\ell_2}{\ell_4}\right)=\left(\frac{\ell_4}{\ell_1}\right)=\left(\frac{\ell_2}{\ell_1}\right)=-\left(\frac{\ell_3}{\ell_1}\right)=-\left(\frac{\ell_3}{\ell_4}\right)=1.$
	\end{enumerate}
	
	If conditions C2) holds, then by Theorem \ref{class2qrs} (cf. $B)-1)$ and $B)-2)$) we have $h_2(K)=1$ if and only if one of the following conditions holds:

	\begin{enumerate}
		\item $\ell_2\equiv \ell_4\equiv7\pmod{8}$, $\ell_1\equiv3\pmod{8}$ and
		$\left(\dfrac{\ell_2}{\ell_1}\right)=-\left(\dfrac{\ell_2}{\ell_4}\right)=\left(\dfrac{\ell_1}{\ell_4}\right)=1$,

		\item $\ell_1\equiv\ell_4\equiv3\pmod{8}$, $\ell_2\equiv7\pmod{8}$ and
		$\left(\dfrac{\ell_2}{\ell_1}\right)=-\left(\dfrac{\ell_2}{\ell_4}\right)=-\left(\dfrac{\ell_1}{\ell_4}\right)=-1$.
	\end{enumerate}
	
	%
	%
	%
	%
	
	Finally if the case $C3)$ holds. We have by Theorems \ref{class2qrs} and \ref{realBiquad} we have $h_2(K)\geq 2$. Hence such field must be rejected.\\
	
	%
	%
	%
	%
	%

	Hence, putting all of the above together, we obtain the complete classification stated in the theorem. This completes the proof. 
\end{enumerate}
\end{proof}

\begin{remark}
 Let  $\ell_1\equiv\ell_2\equiv\ell_3\equiv 1\pmod4$ three distinct prime numbers such that  $\left( \frac{\ell_1}{\ell_2} \right) = \left( \frac{\ell_1}{\ell_3} \right) = \left( \frac{\ell_2}{\ell_3} \right) = -1$ and $K=\QQ(\sqrt{\ell_1 },\sqrt{ \ell_2},\sqrt{ \ell_3})$. From Theorem \ref{triquawithodd} (more precisely, its first item) and Theorem \ref{classnoadqua} (more precisely, its fifth item), we deduce that we have:
	\[
	\vep_{\ell_1 \ell_2}\vep_{\ell_2\ell_3}\vep_{\ell_1\ell_3} \;\;\text{is {\bf not} a square in}\;\; \QQ(\sqrt{\ell_1\ell_2},\sqrt{\ell_2\ell_3}) \Longleftrightarrow \left( \frac{\ell_1\ell_2}{\ell_3} \right)_4 \left( \frac{\ell_1\ell_3}{\ell_2} \right)_4 \left( \frac{\ell_2\ell_3}{\ell_1} \right)_4 = 1.
	\]
\end{remark}

\medskip

To complete our classification, we invoke    Mouhib's  theorem (cf. The Main Theorem of \cite{paritymohib})
that plays the second main key step in this classification. In his theorem, Mouhib classified all Fröhlich real multiquadratic fields of degree $16$ (i.e., 	$P(n)$ with $n=4$, cf., Page \pageref{frohfi} for this definition)  with odd class numbers. The below result is an extension of that theorem to any real multiquadratic fields of degree $16$ (i.e., 	$D(n)$ with $n=4$).

\medskip

\begin{theorem}\label{duaqdriwithodd}Let $K$ be a real multiquadratic number fields of degree $16$. Then the class number of $K$ is odd if and only if $K=\QQ(\sqrt{\ell_1}, \sqrt{\ell_2}, \sqrt{\ell_3}, \sqrt{\ell_4})$, where $\ell_1, \ell_2, \ell_3$ and $\ell_4$ are  distinct prime numbers satisfying, after a suitable permutation, one of the following conditions:
\begin{enumerate}[$1)$]
	\item $\ell_1=2$, $\ell_2\equiv \ell_3 \equiv \ell_4 \equiv-1 \pmod4$, $\left(\dfrac{2}{\ell_2}\right)=-\left(\dfrac{2}{\ell_3}\right)= -\left(\dfrac{2}{\ell_4}\right)=1$ and $\left(\dfrac{\ell_2}{\ell_3}\right)\left(\dfrac{\ell_2}{\ell_4}\right)=-1.$
	\item $\ell_1=2$, $\ell_2\equiv-\ell_3 \equiv -\ell_4 \equiv1 \pmod4$, $\left(\dfrac{2}{\ell_2}\right)=-1,$ and 
	\begin{enumerate}[$i)$]
		\item $\left(\dfrac{\ell_2}{\ell_3}\right)=\left(\dfrac{\ell_2}{\ell_4}\right)=-1$ and $\left(\dfrac{2}{\ell_3}\right)\left(\dfrac{2}{\ell_4}\right)=-1;$ or
		\item $\left(\dfrac{2}{\ell_3}\right)=\left(\dfrac{2}{\ell_4}\right)=-1$ and $\left(\dfrac{\ell_2}{\ell_3}\right)\left(\dfrac{\ell_2}{\ell_4}\right)=-1;$ or
		\item $\left(\dfrac{\ell_2}{\ell_3}\right)\left(\dfrac{\ell_2}{\ell_4}\right)=\left(\dfrac{2}{\ell_3}\right)\left(\dfrac{2}{\ell_4}\right)=-1$ and $\left(\dfrac{\ell_2}{\ell_3}\right)\neq \left(\dfrac{2}{\ell_3}\right).$
	\end{enumerate}
\end{enumerate}
\end{theorem}
\begin{proof}
Put  $K=\QQ(\sqrt{\ell_1}, \sqrt{\ell_2}, \sqrt{\ell_3}, \sqrt{\ell_4})$, where $\ell_1, \ell_2, \ell_3$ and $\ell_4$ are  distinct positive square-free integers.	Assume that  one of the  integers $\ell_{i}$ is composite. Then, the fact that  all $\ell_i$'s are pairwise distinct, implies that there are at least five finite  primes that ramify in $K/\QQ$; this Theorem \ref{FrohlichResults}, implies that   the class number of $K$ is even. 

It follows  by  Mouhib's Theorem (cf. \cite[The Main Theorem, p. 1206]{paritymohib}) that $h(K)$ is odd if and only if 
$\ell_1, \ell_2, \ell_3$ and $\ell_4$ are  distinct prime numbers satisfying
  one of the conditions in the above theorem.
\end{proof}

\begin{remark}
The congruence condition  $\ell_2\equiv-\ell_3 \equiv -\ell_4 \equiv1 \pmod4$, in the second item of the above theorem,    is written in  \cite{paritymohib}
with $-1$ instead of $1$, but this is nothing more than a typo in the statement of the fundamental result of  \cite{paritymohib}.
\end{remark}

\medskip
This completes  the classification of all real multiquadratic with odd class number!

\section{\bf Appendix: On the Square Root of $\vep_{d}$, the Fundamental Unit of a Real Quadratic Fields $\QQ(\sqrt{d})$}\label{App1}
The determination of the form of the square-root of the fundamental units of real quadratic fields play an important role
in the computation of multiquadratic number fields  (cf. \cite{Benchems,	BS25,BS25I} and the references cited therein).
We outline in this appendix the units used for most of our results. In the section we explore the form of $\sqrt{\varepsilon_{d}}$ when ever $d$ the number of prime divisors of $d$ is $2$, $3$ or $4$. We note that this section is a key step in several proofs in the previous investigations, but it is dependent. 

\medskip
Let us begin by revisiting Lemma 2 in \cite{AzAekhTouCor},
completing its proof
and refining the statement accordingly.
\begin{lemma}\label{ZekTAo}
Let \( d \equiv 1 \pmod{4} \) be a positive square free integer and \( \varepsilon_d = x + y\sqrt{d} \) be the fundamental unit of \( \mathbb{Q}(\sqrt{d}) \). Assume \( N(\varepsilon_d) = 1 \), then
\begin{enumerate}[$1)$]
	\item \( x + 1 \) and \( x - 1 \) are not squares in \( \mathbb{Q} \).
	\item For all integer \(\delta \) dividing \( d \), \( \delta(x + 1) \) and \( \delta(x - 1) \) are not squares in \( \mathbb{Q} \).
\end{enumerate}
\end{lemma}

\begin{proof}

As $d\equiv1\pmod4$, then $\vep_{d}=x+y\sqrt{d}$, where $x$ and $y$ are integers or semi integers :\\
\noindent\ding{229} If $x$ and $y$ are integers then $x+1$ and $x-1$ are not squares in $\NN$ and  $x^2-1=dy^2$ (cf.  \cite[First item in Lemma 2]{AzAekhTouCor}). Assume that there is a divisor $\delta$ of $d$ such that:  \( \delta(x+1) \) or \( \delta(x-1) \) is a square in \( \mathbb{N} \). Then   we have  

$$\left\{ \begin{array}{ll}
	x\pm 1=\delta y_1^2\\
	x\mp 1=d' y_2^2,
\end{array}\right.$$
where $y_1, y_2$ and $d'$ are integers such that $y=y_1y_2$ and $d=\delta d'$. Then 

$$\delta d'y_2^2=\delta(x\mp1)=\delta(x\pm1\mp2)=\delta (x\pm1)\mp2\delta=\delta^2y_1^2\mp2\delta,$$ then $$dy_2^2=\delta^2y_1^2\mp2\delta.$$
As $d\equiv 1\pmod4$ and $\delta$ divides $d$,  is follows that  $\delta\equiv \pm1\pmod4$. Reducing modulo 4 we get $$  \mp 2 \equiv y_1^2 - y_2^2 \pmod 4 $$
which is an absurd.\\

\noindent\ding{229} If $x$ and $y$ are semi integers : $x=a/2$ and $y=b/2$,  for some odd integers $a$ and $b$ such that $a^2-4=db^2$. First of all if $a-2$ or $a+2$ is a square in $\NN$, then we have:

$$\left\{ \begin{array}{ll}
	a\pm 2=b_1^2\\
	a\mp 2=d b_2^2,
\end{array}\right.$$
where $b_1$ and $b_2$ are two integers such that $b=b_1b_2$. Then $\vep_d=\frac{1}{4}(b_1^2+b_2^2 d+2b\sqrt{d})$, thus $\sqrt{\vep_d}\in \QQ(\sqrt{d})$, which is absurd. Therefore, \(a+2\) and \(a-2\) are not squares in \(\mathbb{N}\); hence, a fortiori, \(x-1\) and \(x+1\) are not squares in \(\mathbb{Q}\).

\medskip

Now  if there is some divisor  $\delta$ of $d$ such that  $\delta(x+1)$ or $\delta(x-1)$ is a square in $\QQ$. Then as $a$ is odd we have $\delta(x\pm1)=\frac{\delta(a\pm2)}{2}$ which can not be a square in $\QQ$ because $\delta(a\pm2)$ and $2$ are coprime and both square free (A rational number is a square in $\mathbb{Q}$ if and only if both its numerator and its denominator are squares in $\mathbb{N}$.).

\end{proof}

\begin{lemma}\cite[Lemma 6]{BS25}\label{unitsrs}
Let \( r, s \) be distinct primes \( \equiv 3 \mod 4 \). Then 
\begin{enumerate}[$1)$]
	\item  $\sqrt{\varepsilon_{rs}} = \frac{1}{2}(a\sqrt{r} + b\sqrt{s}), \text{ for some } a, b \in \mathbb{Z} \text{ with } a \equiv b \mod 2;$   
	moreover, $ \frac{1}{4}(a^2r - b^2s) = (r/s) = -(s/r)$; 
	
	\item  $\sqrt{\varepsilon_{2r}} = a\sqrt{2} + b\sqrt{r}, \text{ for some } a, b \in \mathbb{Z};$   
	moreover,  $a^2 - b^2 = (2/r) = (-r/2);$
	\item  $\sqrt{\varepsilon_r} = \frac{1}{2}(a\sqrt{2r} + b\sqrt{2}), \text{ for some odd } a, b \in \mathbb{Z};$  
	moreover,  $\frac{1}{2}(a^2r - b^2) = -(r/2). $
\end{enumerate}
\end{lemma}
\begin{remark}
For the other possible congruent cases  of $r$ and $s$ we have the Table \ref{tab7}.
\end{remark}

\begin{lemma}\label{unilem1}
Let $q\equiv3\pmod4$, $r$ and $s$ distinct prime numbers and and
$\varepsilon_{qrs}=x+y\sqrt{qrs}$ where $x,y$ are two integers of semi integers.

\item If $r\equiv s\equiv5\pmod{8}$,
$\left(\dfrac{q}{r}\right)=-\left(\dfrac{q}{s}\right)=-1$, then 
\begin{enumerate}[$1)$]
	\item If $\left(\dfrac{r}{s}\right)=1$, then
	$2s(x+1)$ is a square in $\NN$ and
	$\sqrt{\varepsilon_{qrs}}=y_1\sqrt{s}+y_2\sqrt{qr},$
	\item If $\left(\dfrac{r}{s}\right)=-1$,
	then  $s(x\pm1)$ is a square in $\NN$, and $\sqrt{2\varepsilon_{qrs}}=y_1\sqrt{s}+y_2\sqrt{qr}$.
	
\end{enumerate}
For some integers $y_1$ and $y_2$. \begin{remark}
	For the rest of all possible cases of $q$, $r$ and $s$ we sum up the results concerning $\sqrt{\vep_{qrs}}$ and $\sqrt{\vep_{2qrs}}$ in Table \ref{tab8}. 
\end{remark}

\end{lemma}

\begin{proof}
We use the same techniques as in the proof of 	Lemma 	\ref{unitspp1q1q2}.
\end{proof}

\medskip 	
Moreover, using the same techniques as in the proof of Lemma 	\ref{unitspp1q1q2}, we prove all   results in the following  tables.

\newpage	 
\renewcommand{\arraystretch}{1.3}
\begin{table}[H]
\caption{ \textbf{Forms of $\sqrt{\vep_{rs}}$ and $\sqrt{\vep_{2rs}}$ }}\label{tab7}
\centering
\begin{tabular}[h!]{|c|p{2cm}|c|c|c|c|c|} 
	
	\hline
	&\textbf{\tiny{Legendre Symbole}} & \tiny{Congrunce of $r$ and $s$ $\pmod 8$} & \tiny{Result 1} & \tiny Form of $\sqrt{\vep_{rs}}$ & \tiny Result 2 & \tiny Form of $\sqrt{\vep_{2rs}}$ \\
	\hline
	1& \multirow{14}{2cm}{$\left(\dfrac{r}{s}\right)=1$} & $(5,5)$   & $2r(a\pm1)$ & $[\sqrt{r},\sqrt{s}]$ & **** & **** \\
	\cline{1-1}\cline{3-7}
	2& &\multirow{3}{4cm}{$(1,1)$}& \multirow{3}{2cm}{$2r(a\pm1)$} &  \multirow{3}{2cm}{$[\sqrt{r},\sqrt{s}]$} & $\alpha\pm1$ & $[\sqrt{2},\sqrt{rs}]$\\
	\cline{1-1}\cline{6-7}
	3& & & & & $r(\alpha\pm1)$ & $[\sqrt{2r},\sqrt{s}]$\\
	\cline{1-1}\cline{6-7}
	4& & & & &$2r(\alpha\pm1)$ &$[\sqrt{r},\sqrt{2s}]$ \\
	
	\cline{1-1}\cline{3-7}
	5& & $(5,1)$   & $r(a\pm2)$ & $[\sqrt{r},\sqrt{s}]$ & $r(\alpha\pm1)$  & $[\sqrt{2r},\sqrt{s}]$ \\
	\cline{1-1}\cline{3-7}
	6& & $(1,3)$   & $a-1$ & $[\sqrt{2},\sqrt{2rs}]$ & $\alpha-1$ & $[\sqrt{2},\sqrt{rs}]$ \\
	\cline{1-1}\cline{4-7}
	7& &    & $r(a-1)$ & $[\sqrt{2r},\sqrt{2s}]$ &$2r(\alpha+1)$  &$[\sqrt{r},\sqrt{2s}]$  \\
	
	\cline{1-1}\cline{4-7}
	8& &    &$2r(a+1)$ & $[\sqrt{r},\sqrt{s}]$ & $r(\alpha-1)$ &$[\sqrt{2r},\sqrt{s}]$  \\
	
	\cline{1-1}\cline{3-7}
	9& & $(5,3)$   & $2r(a+1)$ & $[\sqrt{r},\sqrt{s}]$ & $r(\alpha-1)$ & $[\sqrt{2r},\sqrt{s}]$ \\
	\cline{1-1}\cline{3-7}
	10& & $(1,7)$   & $a+1$ & $[\sqrt{2},\sqrt{2rs}]$ & $\alpha+1$ &$[\sqrt{2},\sqrt{rs}]$  \\
	\cline{1-1}\cline{4-7}
	11& &    &  $r(a+1)$ & $[\sqrt{2r},\sqrt{2s}]$ &$2r(\alpha+1)$  & $[\sqrt{r},\sqrt{2s}]$ \\
	\cline{1-1}\cline{4-7}
	12& &    & $2r(a+1)$  & $[\sqrt{r},\sqrt{s}]$ &$r(\alpha+1)$ &$[\sqrt{2r},\sqrt{s}]$  \\
	\cline{1-1}\cline{3-7}
	13& & $(5,7)$    & $2r(a+1)$ & $[\sqrt{r},\sqrt{s}]$ & $r(\alpha+1)$ & $[\sqrt{2r},\sqrt{s}]$  \\
	\cline{1-1}\cline{3-7}
	14& & $(3,3)$    & Lemma \ref{unitsrs} & Lemma \ref{unitsrs} & $\alpha-1$ & $[\sqrt{2},\sqrt{rs}]$ \\
	\cline{1-1}\cline{3-7}
	15& & $(3,7)$    & Lemma \ref{unitsrs} & Lemma \ref{unitsrs} & $r(\alpha+1)$ & $[\sqrt{2r},\sqrt{s}]$ \\
	\cline{1-1}\cline{3-7}
	16& & $(7,3)$    & Lemma \ref{unitsrs} & Lemma \ref{unitsrs} & $2r(\alpha+1)$ & $[\sqrt{r},\sqrt{2s}]$ \\
	\hline
	17& \multirow{5}{2cm}{$\left(\dfrac{r}{s}\right)=-1$} & $(1,3)$   & $a-1$ & $[\sqrt{2},\sqrt{2rs}]$ & $\alpha-1$ & $[\sqrt{2},\sqrt{rs}]$  \\
	\cline{1-1}\cline{3-7}
	18& & $(5,3)$   & $r(a+1)$ &$[\sqrt{2r},\sqrt{2s}]$ & $2r(\alpha-1)$ & $[\sqrt{r},\sqrt{2s}]$ \\
	\cline{1-1}\cline{3-7}
	19& & $(1,7)$   & $a+1$  &$[\sqrt{2},\sqrt{2rs}]$  & $\alpha+1$&  $[\sqrt{2},\sqrt{rs}]$\\
	\cline{1-1}\cline{3-7}
	20& & $(5,7)$   & $r(a-1)$ & $[\sqrt{2r},\sqrt{2s}]$ & $2r(\alpha-1)$ & $[\sqrt{r},\sqrt{2s}]$ \\
	\cline{1-1}\cline{3-7}
	21& & $(3,3)$   & Lemma \ref{unitsrs} & Lemma \ref{unitsrs} & $\alpha-1$ & $[\sqrt{2},\sqrt{rs}]$ \\

	\hline
	
\end{tabular}

\end{table}

\renewcommand{\arraystretch}{1.3}

\centering
\begin{longtable}[h!]{|c|p{2.5cm}|p{2.5cm}|c|c|c|c|} 
\caption{\textbf{Some forms of $\sqrt{\vep_{qrs}}$ and $\sqrt{\vep_{2qrs}}$ }}\label{tab8} \\
\hline 
&\tiny{$r,s,q\equiv \pmod 8$} & \tiny{$(\left(\dfrac{q}{r}\right),\left(\dfrac{q}{s}\right),\left(\dfrac{r}{s}\right))$}  & Result 1 & Form of $\sqrt{\vep_{qrs}}$ & Result 2  & Form of $\sqrt{\vep_{2qrs}}$ \\
\hline
1&\multirow{12}{3.5cm}{$(5,7,3)$}& \multirow{3}{3.5cm}{$(1,1,1)$}   & $2q(x+1)$ & $[\sqrt{q},\sqrt{rs}]$ &  &  \\
\cline{1-1}\cline{4-5}
2& & & $2r(x+1)$ & $[\sqrt{r},\sqrt{qs}]$& $2s(z-1)$& $[\sqrt{s},\sqrt{2qr}]$  \\
\cline{1-1}\cline{4-5}
3&& & $2s(x-1)$ & $[\sqrt{s},\sqrt{qr}]$& & \\ 
\cline{1-1}\cline{3-7}

4&&$(1,1,-1)$ & $2q(x+1)$ & $[\sqrt{q},\sqrt{rs}]$& $s(z+1)$& $[\sqrt{2s},\sqrt{qr}]$  \\
\cline{1-1}\cline{3-7}

5&& $(-1,-1,1)$& $2s(x+1)$ & $[\sqrt{s},\sqrt{qr}]$& $2s(z+1)$& $[\sqrt{s},\sqrt{2qr}]$  \\
\cline{1-1}\cline{3-7}

6&&$(-1,-1,-1)$ & $2r(x-1)$ & $[\sqrt{r},\sqrt{qs}]$& $s(z-1)$& $[\sqrt{2s},\sqrt{qr}]$  \\ 
\cline{1-1}\cline{3-7}

7&&$(-1,1,1)$ & $2s(x-1)$ & $[\sqrt{s},\sqrt{qr}]$& $2s(z-1)$& $[\sqrt{s},\sqrt{2qr}]$  \\
\cline{1-1}\cline{3-7}

8&&$(-1,1,-1)$ & $2r(x-1)$ & $[\sqrt{r},\sqrt{qs}]$&$s(z+1)$  & $[\sqrt{2s},\sqrt{qr}]$  \\
\cline{1-1}\cline{3-7}

9&& \multirow{3}{3.5cm}{$(1,-1,1)$}   & $2s(x+1)$ & $[\sqrt{s},\sqrt{qr}]$ &  &  \\
\cline{1-1}\cline{4-5}
10&& & $2r(x+1)$ & $[\sqrt{r},\sqrt{qs}]$& $2s(z+1)$& $[\sqrt{s},\sqrt{2qr}]$  \\
\cline{1-1}\cline{4-5}
11&& & $2q(x-1)$ & $[\sqrt{q},\sqrt{rs}]$& & \\ 
\cline{1-1}\cline{3-7}

12&& $(1,-1,-1)$& $2q(x-1)$ & $[\sqrt{q},\sqrt{rs}]$& $s(z-1)$& $[\sqrt{2s},\sqrt{qr}]$  \\

\hline
\hline

13&\multirow{16}{3.5cm}{$(5,3,3)$}& \multirow{3}{3.5cm}{$(1,1,1)$}   & $q(x+2)$ & $[\sqrt{q},\sqrt{rs}]$ &  &  \\
\cline{1-1}\cline{4-5}
14&& & $s(x-2)$ & $[\sqrt{s},\sqrt{qr}]$& $r(z-1)$& $[\sqrt{2r},\sqrt{qs}]$  \\
\cline{1-1}\cline{4-5}
15&& & $r(x+2)$ & $[\sqrt{r},\sqrt{qs}]$& & \\ 
\cline{1-1}\cline{3-7}

16&& $(-1,-1,1)$   & $s(x+2)$ & $[\sqrt{s},\sqrt{qr}]$ & $q(z+1)$ & $[\sqrt{2q},\sqrt{rs}]$ \\
\cline{1-1}\cline{3-7}

17&&$(-1,1,1)$   & $s(x-2)$ & $[\sqrt{s},\sqrt{qr}]$ & $q(z-1)$ & $[\sqrt{2q},\sqrt{rs}]$ \\

\cline{1-1}\cline{3-7}
18&& \multirow{3}{3.5cm}{$(1,-1,1)$}   & $r(x+2)$ & $[\sqrt{r},\sqrt{qs}]$ & $r(z-1)$ & $[\sqrt{2r},\sqrt{qs}]$ \\
\cline{1-1}\cline{4-5}
19&& & $s(x+2)$ & $[\sqrt{s},\sqrt{qr}]$& &   \\
\cline{1-1}\cline{4-5}
20&& & $q(x-2)$ & $[\sqrt{q},\sqrt{rs}]$& &   \\

\cline{1-1}\cline{3-7}

21&&$(1,1,-1)$   & $q(x+2)$ & $[\sqrt{q},\sqrt{rs}]$ & $s(z+1)$ & $[\sqrt{2s},\sqrt{qr}]$ \\

\cline{1-1}\cline{3-7}

22&&\multirow{3}{3cm}{$(-1,-1,-1)$}   & \multirow{3}{3cm}{$r(x-2)$}  &\multirow{3}{3cm}{$[\sqrt{r},\sqrt{qs}]$}  & $q(z+1)$ & $[\sqrt{2q},\sqrt{rs}]$ \\
\cline{1-1}\cline{6-7}
23&&  &  &  & $r(z+1)$ & $[\sqrt{2r},\sqrt{qs}]$ \\
\cline{1-1}\cline{6-7}
24&&  &  &  & $s(z-1)$ & $[\sqrt{2s},\sqrt{qr}]$ \\

\cline{1-1}\cline{3-7}

25&& \multirow{3}{3cm}{$(-1,1,-1)$}   &\multirow{3}{3cm}{$r(x-2)$}  & \multirow{3}{3cm}{$[\sqrt{r},\sqrt{qs}]$} & $q(z-1)$ & $[\sqrt{2q},\sqrt{rs}]$ \\
\cline{1-1}\cline{6-7}
26&& & & & $r(z+1)$ & $[\sqrt{2r},\sqrt{qs}]$ \\
\cline{1-1}\cline{6-7}
27&& & & & $s(z+1)$ & $[\sqrt{2s},\sqrt{qr}]$ \\

\cline{1-1}\cline{3-7}

28&&$(1,-1,-1)$   & $q(x-2)$ & $[\sqrt{q},\sqrt{rs}]$ & $s(z-1)$ & $[\sqrt{2s},\sqrt{qr}]$ \\

\hline
\hline

29&\multirow{12}{3.5cm}{$(5,3,7)$}& \multirow{3}{3.5cm}{$(1,1,1)$}   & $2q(x+1)$ & $[\sqrt{q},\sqrt{rs}]$ &  &  \\
\cline{1-1}\cline{4-5}
30&& & $2r(x+1)$ & $[\sqrt{r},\sqrt{qs}]$& $2q(z+1)$& $[\sqrt{q},\sqrt{2rs}]$  \\
\cline{1-1}\cline{4-5}
31&& & $2s(x-1)$ & $[\sqrt{s},\sqrt{qr}]$& & \\ 
\cline{1-1}\cline{3-7}

32&& \multirow{3}{3.5cm}{$(1,-1,1)$}   & $2q(x-1)$ & $[\sqrt{q},\sqrt{rs}]$ & $2q(z-1)$ & $[\sqrt{q},\sqrt{2rs}]$ \\
\cline{1-1}\cline{4-5}
33&& & $2r(x+1)$ & $[\sqrt{r},\sqrt{qs}]$& &   \\
\cline{1-1}\cline{4-5}
34&& & $2s(x+1)$ & $[\sqrt{s},\sqrt{qr}]$& &   \\

\cline{1-1}\cline{3-7}

35&&$(-1,1,1)$ & $2s(x-1)$ & $[\sqrt{s},\sqrt{qr}]$& $q(z-1)$& $[\sqrt{2q},\sqrt{rs}]$  \\ 
\cline{1-1}\cline{3-7}

36&&$(1,1,-1)$ & $2q(x+1)$ & $[\sqrt{q},\sqrt{rs}]$& $2q(z+1)$& $[\sqrt{q},\sqrt{2rs}]$  \\

\cline{1-1}\cline{3-7}

37&&$(-1,-1,-1)$ & $2r(x-1)$ & $[\sqrt{r},\sqrt{qs}]$& $q(z+1)$ & $[\sqrt{2q},\sqrt{rs}]$ \\

\cline{1-1}\cline{3-7}

38&&$(-1,1,-1)$ & $2r(x-1)$ & $[\sqrt{r},\sqrt{qs}]$& $q(z-1)$& $[\sqrt{2q},\sqrt{rs}]$  \\
\cline{1-1}\cline{3-7}

39&&$(1,-1,-1)$ & $2q(x-1)$ & $[\sqrt{q},\sqrt{rs}]$& $2q(z-1)$& $[\sqrt{q},\sqrt{2rs}]$  \\
\cline{1-1}\cline{3-7}

40&& $(-1,-1,1)$  & $2s(x+1)$ & $[\sqrt{s},\sqrt{qr}]$ & $q(z+1)$& $[\sqrt{2q},\sqrt{rs}]$  \\

\hline
\hline

41&\multirow{6}{3.5cm}{$(5,7,7)$}& $(1,1,-1)$& $q(x+2)$& $[\sqrt{q},\sqrt{rs}]$& $2q(z+1)$&$[\sqrt{q},\sqrt{2rs}]$ \\
\cline{1-1}\cline{3-7}
42&& $(1,-1,-1)$& $q(x-2)$& $[\sqrt{q},\sqrt{rs}]$&$2q(z-1)$ &$[\sqrt{q},\sqrt{2rs}]$ \\
\cline{1-1}\cline{3-7}
43&& $(-1,1,1)$& $s(x-2)$& $[\sqrt{s},\sqrt{qr}]$&$2s(z-1)$ &$[\sqrt{s},\sqrt{2qr}]$ \\
\cline{1-1}\cline{3-7}
44&& $(-1,1,-1)$& $r(x-2)$& $[\sqrt{r},\sqrt{qs}]$& $r(z-1)$&$[\sqrt{2r},\sqrt{qs}]$ \\
\cline{1-1}\cline{3-7}
45&& $(-1,-1,1)$& $s(x+2)$& $[\sqrt{s},\sqrt{qr}]$&$2s(z+1)$ &$[\sqrt{s},\sqrt{2qr}]$ \\
\cline{1-1}\cline{3-7}
46&& $(-1,-1,-1)$& $r(x-2)$& $[\sqrt{r},\sqrt{qs}]$&$r(z-1)$ &$[\sqrt{2r},\sqrt{qs}]$ \\
\hline
\hline
47&\multirow{6}{3.5cm}{$(1,7,3)$}& $(1,1,-1)$& $2q(x+1)$& $[\sqrt{q},\sqrt{rs}]$& $q(z+1)$&$[\sqrt{2q},\sqrt{rs}]$ \\
\cline{1-1}\cline{3-7}
48&& $(1,-1,-1)$& $2q(x-1)$& $[\sqrt{q},\sqrt{rs}]$&$q(z-1)$ &$[\sqrt{2q},\sqrt{rs}]$\\
\cline{1-1}\cline{3-7}
49&& $(-1,1,1)$& $2s(x-1)$& $[\sqrt{s},\sqrt{qr}]$&$2s(z-1)$ &$[\sqrt{s},\sqrt{2qr}]$ \\
\cline{1-1}\cline{3-7}
50&& $(-1,1,-1)$& $2r(x-1)$& $[\sqrt{r},\sqrt{qs}]$& $2r(z-1)$&$[\sqrt{r},\sqrt{2qs}]$ \\
\cline{1-1}\cline{3-7}
51&& $(-1,-1,1)$& $2s(x+1)$& $[\sqrt{s},\sqrt{qr}]$&$2s(z+1)$ &$[\sqrt{s},\sqrt{2qr}]$ \\
\cline{1-1}\cline{3-7}
52&& $(-1,-1,-1)$& $2r(x-1)$& $[\sqrt{r},\sqrt{qs}]$&$2r(z-1)$ &$[\sqrt{r},\sqrt{2qs}]$ \\
\hline
\hline
53&\multirow{6}{3.5cm}{$(1,3,3)$}& $(1,1,-1)$& $2q(x+1)$& $[\sqrt{q},\sqrt{rs}]$& $z-1$&$[\sqrt{2},\sqrt{qrs}]$ \\
\cline{1-1}\cline{3-7}
54&& $(1,-1,-1)$& $2q(x-1)$& $[\sqrt{q},\sqrt{rs}]$&$z-1$ &$[\sqrt{2},\sqrt{qrs}]$ \\
\cline{1-1}\cline{3-7}
55&& $(-1,1,1)$& $2s(x-1)$& $[\sqrt{s},\sqrt{qr}]$&$z-1$ &$[\sqrt{2},\sqrt{qrs}]$ \\
\cline{1-1}\cline{3-7}
56&& \multirow{3}{3cm}{$(-1,1,-1)$} & \multirow{3}{3cm}{$2r(x-1)$} & \multirow{3}{3cm}{$[\sqrt{r},\sqrt{qs}]$}&$z-1$ & $[\sqrt{2},\sqrt{qrs}]$\\
\cline{1-1}\cline{6-7}
57& & & & & $r(z+1)$& $[\sqrt{2r},\sqrt{qs}]$ \\
\cline{1-1}\cline{6-7}
58& & & & &$2r(z-1)$ & $[\sqrt{r},\sqrt{2qs}]$ \\

\cline{1-1}\cline{3-7}
59&& $(-1,-1,1)$& $2s(x+1)$& $[\sqrt{s},\sqrt{qr}]$&$z-1$ &$[\sqrt{2},\sqrt{qrs}]$ \\
\cline{1-1}\cline{3-7}
60&&  \multirow{3}{3cm}{$(-1,-1,-1)$}  & \multirow{3}{3cm}{$2r(x-1)$}&  \multirow{3}{3cm}{$[\sqrt{r},\sqrt{qs}]$}& $z-1$ & $[\sqrt{2},\sqrt{qrs}]$ \\
\cline{1-1}\cline{6-7}
61& & & & & $r(z+1)$&$[\sqrt{2r},\sqrt{qs}]$ \\
\cline{1-1}\cline{6-7}
62& & & & & $2r(z-1)$&$[\sqrt{r},\sqrt{2qs}]$ \\

\hline
\hline
63&\multirow{6}{3.5cm}{$(1,3,7)$}& $(1,1,-1)$& $2q(x+1)$& $[\sqrt{q},\sqrt{rs}]$& $2q(z+1)$&$[\sqrt{q},\sqrt{2rs}]$ \\
\cline{1-1}\cline{3-7}
64& &$(1,-1,-1)$& $2q(x-1)$& $[\sqrt{q},\sqrt{rs}]$&$2q(z-1)$ &$[\sqrt{q},\sqrt{2rs}]$ \\
\cline{1-1}\cline{3-7}
65&& $(-1,1,1)$& $2s(x-1)$& $[\sqrt{s},\sqrt{qr}]$&$s(z-1)$ &$[\sqrt{2s},\sqrt{qr}]$ \\
\cline{1-1}\cline{3-7}
66&& $(-1,1,-1)$& $2r(x-1)$& $[\sqrt{r},\sqrt{qs}]$& $2r(z-1)$&$[\sqrt{r},\sqrt{2qs}]$ \\
\cline{1-1}\cline{3-7}
67& &$(-1,-1,1)$& $2s(x+1)$& $[\sqrt{s},\sqrt{qr}]$&$s(z+1)$ &$[\sqrt{2s},\sqrt{qr}]$ \\
\cline{1-1}\cline{3-7}
68& &$(-1,-1,-1)$& $2r(x-1)$& $[\sqrt{r},\sqrt{qs}]$&$2r(z-1)$ &$[\sqrt{r},\sqrt{2qs}]$ \\
\hline		\hline

69 &\multirow{18}{3.5cm}{$(1,7,7)$}&  \multirow{3}{3cm}{$(1,1,-1)$}&\multirow{3}{3cm}{$2q(x+1)$} &\multirow{3}{3cm}{$[\sqrt{q},\sqrt{rs}]$} & $z+1$& $[\sqrt{2},\sqrt{qrs}]$\\
\cline{1-1}\cline{6-7}
70& & & & &$q(z+1)$ &$[\sqrt{2q},\sqrt{rs}]$ \\
\cline{1-1}\cline{6-7}
71& & & & & $2q(z+1)$&$[\sqrt{q},\sqrt{2rs}]$ \\

\cline{1-1}\cline{3-7}
72&& \multirow{3}{3cm}{$(1,-1,-1)$} &\multirow{3}{3cm}{$2q(x-1)$} &\multirow{3}{3cm}{$[\sqrt{q},\sqrt{rs}]$} & $z+1$&$[\sqrt{2},\sqrt{qrs}]$ \\
\cline{1-1}\cline{6-7}
73& & & & &$q(z-1)$ &$[\sqrt{2q},\sqrt{rs}]$ \\
\cline{1-1}\cline{6-7}
74& & & & & $2q(z-1)$&$[\sqrt{q},\sqrt{2rs}]$ \\

\cline{1-1}\cline{3-7}
75&& \multirow{3}{3cm}{$(-1,1,1)$}&\multirow{3}{3cm}{$2s(x-1)$} & \multirow{3}{3cm}{$[\sqrt{s},\sqrt{qr}]$}& $z+1$&$[\sqrt{2},\sqrt{qrs}]$ \\
\cline{1-1}\cline{6-7}
76& & & & & $2s(z-1)$& $[\sqrt{s},\sqrt{2qr}]$\\
\cline{1-1}\cline{6-7}
77& & & & & $s(z-1)$& $[\sqrt{2s},\sqrt{qr}]$\\

\cline{1-1}\cline{3-7}
78&& \multirow{3}{3cm}{$(-1,1,-1)$}&\multirow{3}{3cm}{$2r(x-1)$} & \multirow{3}{3cm}{$[\sqrt{r},\sqrt{qs}]$}&$z+1$ &$[\sqrt{2},\sqrt{qrs}]$ \\
\cline{1-1}\cline{6-7}
79& & & & &$2r(z-1)$ &$[\sqrt{r},\sqrt{2qs}]$ \\
\cline{1-1}\cline{6-7}
80& & & & &$r(z-1)$ &$[\sqrt{2r},\sqrt{qs}]$ \\

\cline{1-1}\cline{3-7}
81&& \multirow{3}{3cm}{$(-1,-1,1)$} & \multirow{3}{3cm}{$2s(x+1)$}& \multirow{3}{3cm}{$[\sqrt{s},\sqrt{qr}]$}&$z+1$&$[\sqrt{2},\sqrt{qrs}]$ \\
\cline{1-1}\cline{6-7}
82& & & & &$2s(z+1)$ &$[\sqrt{s},\sqrt{2qr}]$ \\
\cline{1-1}\cline{6-7}
83& & & & & $s(z+1)$&$[\sqrt{2s},\sqrt{qr}]$ \\

\cline{1-1}\cline{3-7}
84&& \multirow{3}{3cm}{$(-1,-1,-1)$} &\multirow{3}{3cm}{$2r(x-1)$} & \multirow{3}{3cm}{$[\sqrt{r},\sqrt{qs}]$} &$z+1$ &$[\sqrt{2},\sqrt{qrs}]$ \\
\cline{1-1}\cline{6-7}
85& & & & & $2r(z-1)$&$[\sqrt{r},\sqrt{2qs}]$ \\
\cline{1-1}\cline{6-7}
86& & & & & $r(z-1)$&$[\sqrt{2r},\sqrt{qs}]$ \\
\hline
\hline

87 &$(5,5,3 \;\text{or}\; 7)$ & $(-1,1,1)$&$2s(x+1)$ & $[\sqrt{s},\sqrt{qr}]$& $2r(z-1)$& $[\sqrt{r},\sqrt{2qs}]$ \\
\hline		\hline

88 &$(5,5,3)$ & $(-1,1,-1)$&$s(x-1)$ & $[\sqrt{2s},\sqrt{2qr}]$ & $r(z+1)$& $[\sqrt{2r},\sqrt{qs}]$\\ 

\hline		\hline

89 &$(5,5,7 )$ & $(-1,1,-1)$&$s(x+1)$ &$[\sqrt{2s},\sqrt{2qr}]$ & $r(z-1)$& $[\sqrt{2r},\sqrt{qs}]$\\ 
\hline		\hline

90 &$(3,7,3 )$ & $(-1,-1,1)$& $2q(x-1)$&$[\sqrt{q},\sqrt{rs}]$ & $2r(z+1)$& $[\sqrt{r},\sqrt{2qs}]$\\

\hline		\hline

91 &\multirow{3}{3.5cm}{$(3,7,3)$} & \multirow{3}{3.5cm}{$(-1,-1,-1)$}& $r(x-1)$&$[\sqrt{2r},\sqrt{2qs}]$ &  & \\
\cline{1-1}\cline{4-5}
92 && & $s(x-1)$ & $[\sqrt{2s},\sqrt{2qr}]$ &$s(z-1)$ &$[\sqrt{2s},\sqrt{qr}]$ \\
\cline{1-1}	\cline{4-5}
93 && & $2q(x-1)$ &  $[\sqrt{q},\sqrt{rs}]$ & &\\

\hline		\hline

94&\multirow{3}{3.5cm}{$(3,3,7 )$} & $(-1,1,1)$&$2s(x-1)$ &$[\sqrt{s},\sqrt{qr}]$ & $2r(z+1)$& $[\sqrt{r},\sqrt{2qs}]$\\
\cline{1-1}\cline{3-7}
95	&& $(1,-1,1)$&$r(x-1)$ &$[\sqrt{2r},\sqrt{2qs}]$ & $s(z+1)$& $[\sqrt{2s},\sqrt{qr}]$\\

\cline{1-1}\cline{3-7}
96&& $(-1,1,-1)$&$s(x-1)$ &$[\sqrt{2s},\sqrt{2qr}]$ & $r(z+1)$& $[\sqrt{2r},\sqrt{qs}]$\\

\hline		\hline

97&\multirow{3}{3.5cm}{$(7,3,3 )$} & $(-1,1,1)$& $q(x-1)$ &$[\sqrt{2q},\sqrt{2rs}]$ & $q(z-1)$& $[\sqrt{2q},\sqrt{rs}]$\\
\cline{1-1}\cline{4-5}
98&& &$r(x-1)$ & $[\sqrt{2r},\sqrt{2qs}]$& $r(z-1)$& $[\sqrt{2r},\sqrt{qs}]$\\
\cline{1-1}\cline{4-5}
99&& &$2s(x-1)$ &$[\sqrt{s},\sqrt{qr}]$ & & \\
\hline		\hline

100&$(3,7,3)$ & $(-1,1,-1)$& $r(x-1)$& $[\sqrt{2r},\sqrt{2qs}]$& $q(z+1)$& $[\sqrt{2q},\sqrt{rs}]$\\

\hline		\hline

101&\multirow{3}{3.5cm}{$(3,7,3)$} & $(-1,1,1)$&$2s(x-1)$ & $[\sqrt{s},\sqrt{qr}]$& $q(z+1)$& $[\sqrt{2q},\sqrt{rs}]$\\
\cline{1-1}\cline{6-7}
102&& & & & $2r(z+1)$& $[\sqrt{r},\sqrt{2qs}]$\\
\cline{1-1}\cline{6-7}
103&& & & & $2s(z-1)$& $[\sqrt{s},\sqrt{2qr}]$\\
\hline		\hline

104	&$(7,3,3)$ & $(1,1,1)$&$2s(x-1)$ &$[\sqrt{s},\sqrt{qr}]$ & $2q(z+1)$& $[\sqrt{q},\sqrt{2rs}]$\\
\hline		\hline

105	& \multirow{3}{3.5cm}{$(7,3,3)$}  &$(1,1,-1)$&$2r(x-1)$ &$[\sqrt{r},\sqrt{qs}]$ & $s(z+1)$& $[\sqrt{2s},\sqrt{qr}]$\\
\cline{1-1}\cline{6-7}
106&& & & & $2r(z-1)$& $[\sqrt{r},\sqrt{2qs}]$\\
\cline{1-1}\cline{6-7}
107&	& & & & $2q(z+1)$& $[\sqrt{q},\sqrt{2rs}]$\\
\hline		\hline

108&$(7,7,7)$ &$(-1,1,-1)$&$x+1$ & $[\sqrt{2},\sqrt{2qrs}]$&$z+1$&$[\sqrt{2},\sqrt{qrs}]$ \\

\hline		\hline

109&$(7,3,7)$ & $(-1,1,-1)$&$r(x+1)$ &$[\sqrt{2r},\sqrt{2qs}]$ &$r(z+1)$ &$[\sqrt{2r},\sqrt{qs}]$ \\

\hline 		\hline

110&$(3,3,3)$ &$(-1,1,-1)$&$x-1$ &$[\sqrt{2},\sqrt{2qrs}]$ &$z-1$ & $[\sqrt{2},\sqrt{qrs}]$\\

\hline		\hline

111&$(7,3,7)$ & $(1,1,-1)$& $2r(x-1)$& $[\sqrt{r},\sqrt{qs}]$& $2r(z-1)$& $[\sqrt{r},\sqrt{2qs}]$\\
\hline		\hline

112&\multirow{3}{3.5cm}{$(7,7,3)$} &\multirow{3}{3.5cm}{$(1,1,-1)$} &$q(x+1)$ &$[\sqrt{2q},\sqrt{2rs}]$ &$s(z+1)$ &$[\sqrt{2s},\sqrt{qr}]$ \\
\cline{1-1}\cline{4-7}
113	&& &$s(x+1)$ &$[\sqrt{2s},\sqrt{2qr}]$ & $2q(z+1)$&$[\sqrt{q},\sqrt{2rs}]$\\
\cline{1-1}\cline{4-7}
114&& & $2r(x-1)$&$[\sqrt{r},\sqrt{qs}]$ &$2r(z-1)$ &$[\sqrt{r},\sqrt{2qs}]$\\
\hline		\hline

115&$(3,7,3)$ & $(1,1,-1)$&$2r(x-1)$ &$[\sqrt{r},\sqrt{qs}]$ &$2q(z+1)$ &$[\sqrt{q},\sqrt{2rs}]$ \\
\hline		\hline

116&\multirow{3}{3.5cm}{$(3,3,7)$} &\multirow{3}{3.5cm}{$(1,1,-1)$}  &$q(x-1)$ &$[\sqrt{2q},\sqrt{2rs}]$ & & \\
\cline{1-1}\cline{4-5}
117	&&&$s(x-1)$&$[\sqrt{2s},\sqrt{2qr}]$&$q(x-1)$&$[\sqrt{2q},\sqrt{rs}]$\\
\cline{1-1}	\cline{4-5}
118&	&&$2r(x-1)$&$[\sqrt{r},\sqrt{qs}]$&&\\
\hline		\hline

119&$(1,5,7)$ & $(-1,1,-1)$& $r(x-1)$&$[\sqrt{2r},\sqrt{2qs}]$ & $r(z-1)$ &$[\sqrt{2r},\sqrt{qs}]$ \\

\hline		\hline

120&$(1,5,3)$ & $(-1,1,-1)$& $r(x+1)$& $[\sqrt{2r},\sqrt{2qs}]$&$r(z+1)$ &$[\sqrt{2r},\sqrt{qs}]$ \\
\hline		\hline

121&\multirow{3}{3.5cm}{$(5,5,7)$} & \multirow{3}{3.5cm}{$(-1,-1,-1)$}&\multirow{3}{3cm}{$q(x-1)$} & \multirow{3}{3cm}{$[\sqrt{2q},\sqrt{2rs}]$} & $q(z-1)$&$[\sqrt{2q},\sqrt{rs}]$ \\
\cline{1-1}\cline{6-7}
122&& & & & $s(z-1)$&$[\sqrt{2s},\sqrt{qr}]$ \\
\cline{1-1}\cline{6-7}
123&& & & & $r(z-1)$&$[\sqrt{2r},\sqrt{qs}]$ \\
\hline		\hline

124&$(1,5,3)$ & $(-1,-1,-1)$&$r(x+1)$ & $[\sqrt{2r},\sqrt{2qs}]$& $r(z+1)$&$[\sqrt{2r},\sqrt{qs}]$ \\
\hline		\hline

125&$(5,5,7)$ & $(1,-1,-1)$& $r(x+1)$&$[\sqrt{2r},\sqrt{2qs}]$ & $s(z-1)$&$[\sqrt{2s},\sqrt{qr}]$ \\
\hline		\hline

126&$(5,5,3)$ & $(1,-1,-1)$&$r(x-1)$ &$[\sqrt{2r},\sqrt{2qs}]$ &$s(z+1)$ &$[\sqrt{2s},\sqrt{qr}]$ \\

\hline
\hline

127&$(5,1,7)$ & $(1,1,-1)$&$2q(x-1)$ &$[\sqrt{q},\sqrt{rs}]$ &$2q(z-1)$ &$[\sqrt{q},\sqrt{2rs}]$ \\

\hline		\hline

128&$(5,1,3)$ & $(1,-1,-1)$&$s(x+1)$ & $[\sqrt{2s},\sqrt{2qr}]$ &$s(z+1)$ &$[\sqrt{2s},\sqrt{qr}]$ \\

\hline		\hline

129	&$(1,5,7)$ & $(-1,-1,-1)$&$r(x-1)$ &$[\sqrt{2r},\sqrt{2qs}]$ &$r(z-1)$ &$[\sqrt{2r},\sqrt{qs}]$ \\

\hline		\hline

130	&$(1,5,3)$ & $(-1,-1,-1)$& $r(x+1)$& $[\sqrt{2r},\sqrt{2qs}]$&$r(z+1)$ &$[\sqrt{2r}, \sqrt{qs}]$ \\

\hline

\end{longtable}


\end{document}